\documentclass[12pt,reqno]{amsart}

\usepackage[colorlinks,linkcolor={blue},
citecolor={blue},urlcolor={blue}]{hyperref}

\usepackage{mathrsfs}
\usepackage{bbm}
\usepackage{amsfonts}
\usepackage{amsfonts}
\usepackage{amsfonts}
\usepackage{amsfonts}
\usepackage{amsfonts}
\usepackage{amsfonts}
\usepackage{amsfonts}
\usepackage{amsfonts}
\usepackage{amsfonts}
\usepackage{amsfonts}
\usepackage{amsfonts}
\usepackage{amsfonts}
\usepackage{amsfonts}
\usepackage{amsfonts}
\usepackage{amsfonts}
\usepackage{amsfonts}
\usepackage{amsfonts}
\usepackage{amsfonts}
\usepackage{amsfonts}
\usepackage{}
\usepackage{pifont}
\usepackage[shortlabels]{enumitem}
\usepackage{empheq}
\usepackage{amsfonts}
\usepackage{amsmath, amsrefs, amsthm, amssymb}
\usepackage[active]{srcltx}		
\usepackage{pdfsync}					
\usepackage{graphicx}
\usepackage [latin1]{inputenc}
\usepackage{bbm}
\usepackage{amsfonts}
\usepackage{amsthm}
\usepackage{graphicx}
\usepackage{framed}
\usepackage{multicol,multienum}
\usepackage{graphicx} 
\usepackage{booktabs} 
\usepackage{mathrsfs}
\usepackage{tcolorbox}
\usepackage{color}
\usepackage{xcolor}

\newtheorem{theorem}{Theorem}[section]

\newtheorem{lemma}[theorem]{Lemma}

\theoremstyle{definition}
\newtheorem{definition}[theorem]{Definition}

\newtheorem{assumption}[theorem]{Assumption}

\newtheorem{remark}[theorem]{Remark}

\numberwithin{equation}{section}

\begin{document}

\title [High-dimensional MEP2 with random noises]{{{\rmfamily  Effect of random noises on pathwise solutions to the high-dimensional modified Euler-Poincar\'{e} system}}}

\author{Lei Zhang}
\address{School of Mathematics and Statistics, Hubei Key Laboratory of Engineering Modeling  and Scientific Computing, Huazhong University of Science and Technology,  Wuhan 430074, Hubei, P.R. China.}
\email{lei\_zhang@hust.edu.cn}


\date{\today}
\maketitle
\begin{abstract}
In this paper, we study the Cauchy problem for the stochastically perturbed high-dimensional modified Euler-Poincar\'{e} system (MEP2) on the torus $\mathbb{T}^d$, $d\geq 1$. We first establish a local well-posedness framework in the sense of Hadamard for the MEP2 driven by general nonlinear multiplicative noises. Then two kinds of global existence and uniqueness results are demonstrated: One indicates that the MEP2 perturbed by nonlocal-type random noises with proper intensity admits a unique large global strong solution; The other one infers that, if the initial data is sufficiently small, then the MEP2 perturbed by linear multiplicative noise has a unique global solution with high probability. In the case of one dimension, we find that the stochastic MEP2  will break down in finite time when the initial data meets appropriate shape condition.
\end{abstract}

\section{Introduction}
The Camassa-Holm (CH) equation, which can be used to describe the unidirectional propagation of shallow water waves over a flat bottom \cite{Camassa-Holm,Johnson} or the propagation of axially symmetric waves in hyperelastic rods \cite{dai1998model}, has been studied extensively during the past decades. The most remarkable features of CH equation are the   existence of peakon solutions \cite{Camassa-Holm,Constantin-Strauss2000,el2009stability,lenells2004stability} and the wave breaking phenomena \cite{constantin1998,constantin1-2000,constantin2000,brandolese2014local,brandolese2014permanent}, which can not be characterized by the KdV equation \cite{gardner1967method,miles1981korteweg}. Recently, Holm et al. \cite{holm2009singular} extend the Euler-Poincar\'{e} equation \cite{holm1998euler} to the modified Euler-Poincar\'{e} system (MEP2) so as to combine its integrability property with free-surface elevation dynamics in its shallow-water interpretation. The MEP2 is defined as geodesic motion on the semidirect product Lie group with respect to a certain metric and is given as a set of Euler-Poincar\'{e} equations on the dual of the corresponding Lie algebra. To be more precise, considering the variational principle
$\delta\int \mathcal {L}(u,\overline{\rho}) \mathrm{d} t =0 $ with the Lagrangian
\begin{equation*}
\begin{split}
\mathcal {L}(u,\overline{\rho})= \frac{1}{2}\int_{\mathbb{R}^d} u\cdot (1-\alpha_1^2\Delta)u \mathrm{d} x + \frac{g}{2}\int_{\mathbb{R}^d}(\overline{\rho}-\overline{\rho}_0)(1-\alpha_2^2\Delta)(\overline{\rho}-\overline{\rho}_0)\mathrm{d} x,
\end{split}
\end{equation*}
where $\Delta$ denotes the $d$-dimensional Laplacian operator,  $\alpha_1,\alpha_2 \in \mathbb{R}^+$ are two length scales and $g > 0$ is the downward constant acceleration of gravity in application to shallow water waves. By substituting the variational derivatives for Lagrangian $\mathcal {L}(u,\overline{\rho})$ into the semidirect-product Euler-Poincar\'{e} equations, one obtains the MEP2 system in $\mathbb{R}^{d}$, $d\geq1$ formulated in coordinates, we refer to the works \cite{holm2009singular,marsden2013introduction,holm1998euler} for more details.

In this paper, we study the Cauchy problem for the following stochastically perturbed modified Euler-Poincar\'{e} system (SMEP2) on  the torus $\mathbb{T}^d \triangleq (\mathbb{R}/2\pi \mathbb{Z})^d$, $d\geq1$:
\begin{equation}\label{1.4}
\left\{
\begin{aligned}
&\mathrm{d}m+\left(u\cdot \nabla m+(\nabla u)^Tm+(\textrm{div}  u)m+\rho \nabla \overline{\rho} \right)\mathrm{d} t =g_1(t,m,\rho)\mathrm{d}\mathcal {W}_1,\\
&\mathrm{d}\rho+  \textrm{div}  (\rho u)= g_2(t,m,\rho)\mathrm{d}\mathcal {W}_2,\\
&m=(1-\alpha_1^2\Delta) u,\\
&\rho=(1- \alpha_2^2\Delta)(\overline{\rho}-\overline{\rho}_0),
\end{aligned} \quad t\geq 0,~x\in \mathbb{T}^d,
\right.
\end{equation}
which is endowed with the initial conditions
\begin{equation}\label{1.5}
\begin{split}
m(0,x)=m_0(x),\quad \rho(0,x)=\rho_0(x),\quad x\in \mathbb{T}^d.
\end{split}
\end{equation}
Here, $u$ denotes the velocity of fluid,  $m$ with the component $m_j=u_j- \Delta u_j$, $j=1,2,...d$ represents the momentum. The scalar functions $\rho $ and $\overline{\rho}$ stand for the total depth of the free surface (or density) and averaged depth, respectively. 
The driven stochastic processes $\mathcal {W}_1$ and $\mathcal {W}_2$ are independent cylindrical Wiener processes defined on separable Hilbert spaces. The precise assumptions on the coefficients $g_1 $ and $g_2 $ as well as further details are given in Subsection \ref{sec1.1}. Without loss of generality, we will assume that $\alpha_1=\alpha_2=g=1$.

The deterministic MEP2 without random noises ($g_1=g_2\equiv0$ in \eqref{1.4}) is closely related to two kinds of important models. The first one is the Euler-Poincar\'{e} equation (EP), which takes the form of
\begin{equation}\label{1.2}
\left\{
\begin{aligned}
&\partial_t m +u\cdot \nabla m+(\nabla u)^T \cdot m+m (\textrm{div} u)=0,\\
&m=(1-\alpha^2\Delta)u,\quad \alpha>0.
\end{aligned}
\right.
\end{equation}
The system \eqref{1.2} was first introduced by Holm et al. \cite{holm1998euler,holm1998euler1} as a framework for modeling and analyzing fluid dynamics, particularly for nonlinear shallow water waves, geophysical fluids and turbulence modeling.  EP can be considered as an evolutionary equation for a geodesic motion on a diffeomorphism group \cite{ebin1970groups,khesin2008geometry,younes2010shapes,holm2009geometric}, and it has important applications in computational anatomy (cf. \cite{holm2009geometric,younes2010shapes}).  EP has many further interpretations beyond fluid applications. For instance, it is exactly the same as the averaged template matching equation for computer vision (cf. \cite{holm2004soliton,holm2005momentum,hirani2001averaged}). The rigorous analysis of EP was initiated by Chae and Liu \cite{chae2012blow}, in which the authors established a fairly complete well-posedness theory for both weak and strong solutions. In \cite{li2013euler}, Li et al. proved that for a large class of smooth initial data, the corresponding solution to EP with $\alpha\neq 0$ blows up in finite time, which settles an open problem raised in \cite{chae2012blow}. The local well-posedness result is improved to Besov spaces by Yan and Yin \cite{yan2015initial}. The blow-up phenomena and ill-posedness problem for EP on torus $\mathbb{T}^d$ are investigated by Luo and Yin \cite{luo2020blow}. Moreover, it is shown that the data-to-solution map for EP is not uniformly continuous in \cite{zhao2018non,li2019non}. Besides, Tang \cite{tang2020noise} considered the effect of random noise on the dynamic behavior of pathwise solutions to EP. Especially, when $d=1$, Eq.\eqref{1.2} reduces to the celebrated Camassa-Holm equation  introduced in \cite{Camassa-Holm}, which has received much attention during the past twenty years since its derivation from the shallow water regime.

The other one is the so-called two-component Euler-Poincar\'{e} system (EP2):
\begin{equation}\label{1.3}
\left\{
\begin{aligned}
& \partial_t m +u\cdot \nabla m+(\nabla u)^T \cdot m+m (\textrm{div} u)=-g\rho \nabla\rho,\\
&\partial_t \rho +\textrm{div} (\rho u)=0,\\
&m= (1-\alpha^2\Delta)u.
\end{aligned}
\right.
\end{equation}
EP2 in one dimension was first introduced by Chen and Zhang \cite{chen2006two}, Falqui \cite{falqui2005camassa} and Constantin and Ivanov \cite{constantin2008integrable}, where the latest work gave a rigorous justification of the derivation in the context of shallow water regime. Later, Holm et al. \cite{holm2009singular,kohlmann2012note}; see also Kohlmann \cite{kohlmann2012note} and Holm and Tronci \cite{11holm2009geodesic}, extended the system to multi-dimensional case by considering the Hamilton principle $\delta \int \mathcal {L}(u,\rho) \mathrm{d} t=0$ with the Lagrangian given by
\begin{equation}\label{lg}
\begin{split}
\mathcal {L}(u,\rho)= \frac{1}{2}\int_{\mathbb{R}^d} u\cdot (1-\Delta)u \mathrm{d} x + \frac{g}{2}\int_{\mathbb{R}^d}(\rho-\rho_0)^2\mathrm{d} x.
\end{split}
\end{equation}

MEP2 (with $g_1=g_2=0$) can be regarded as a modification of the EP2 \eqref{1.3}, where the modification amounts to strengthening the norm for $\overline{\rho}$ from $L^2$ to $H^1$ in the Lagrangian \eqref{lg}. This main difference leads to the fact that, EP2 does not admit singular solutions in the variable $\rho$, while MEP2 admits peaked soliton solutions in both variables $u$ and $\rho$.  After its derivation, EP2 has been studied by several authors. For instance, in \cite{duan2014cauchy}, Duan and Xiang investigated the Cauchy problem for EP2 in Sobolev spaces by using the energy method. Later, Li and Yin \cite{li2017well} established the local well-posedness of EP2 in nonhomogeneous Besov spaces. In terms of the abstract Cauchy-Kowalevski lemma, they also proved the existence of local-in-time analytical solutions. In the case of $d=1$, EP2 has also attracted much attention owing the fact that the it describing the wave-breaking phenomena in finite time and admits solitary wave solutions interacting like solitons. To name a few, we would like to refer the readers to \cite{guan2010global,guan2011global,gui2010global,gui2011cauchy,grunert2012global} and the references therein.

The  local and global existence problem for the one-dimensional MEP2 in Sobolev spaces and Besov spaces have been studied by several authors, see e.g., \cite{guo2012wave,guan2010well,guan2011global,tan2011global} and the works cited therein. To our best knowledge, few works are available for the Cauchy problem of high dimensional MEP2 besides the resent work \cite{yan2022initial}, in which the author established the local theory of strong solutions in nonhomogeneous Besov spaces. It is worth pointing out that the existence of global solutions to the high-dimensional MEP2 is still an open problem, even though related results have been established for the one-dimensional MEP2 (cf. \cite{guan2010well,guan2011global}). The main difficulty arising from the fact that the well-known sign condition (cf. \cite{constantin1998,constantin1-2000}) can not be generalized to the high-dimensional cases. As far as we aware, the most relevant work to this problem is \cite{li2013euler} for EP \eqref{1.2}, in which the authors proved a global in time result by transforming EP into a scalar equation via special radial functions, while some structural conditions, such as non-positive monotone initial data similar to the one dimension equations, are still needed.

The importance of incorporating stochastic effects in the modeling of complex systems has been recognized during the past decades, and the dynamic behavior of fluid models perturbed by different kinds of noises has been widely studied.  To just mention a few, see for example \cite{27,28,29,30,31,32,33,34,52,53,zhang2024global} for the PDEs theory on some stochastic fluid models, and refer to recent works on stochastic dispersive equations \cite{35,36,37,54,rohde2021stochastic,ren2020distribution,rohde2020stochastic,miao2021well,
alonso2021local,galimberti2024global,albeverio2021stochastic}. Recently, in \cite{miao2021well}, Miao, Rohde and Tang studied the Cauchy problem for one dimensional Camassa-Holm (CH) type equations with high nonlinearity in the sense of Hadamard. In \cite{rohde2020stochastic, rohde2021stochastic}, the existence of global solutions and blow-up criteria for some shallow water wave equations under appropriate random noises are investigated. Also, the authors in \cite{ren2020distribution} provides a well-posedness result for abstract system which involves a series of CH-type equations. In \cite{holden2023global}, Holden et al.  proved the existence and uniqueness of global weak solutions for viscous CH equation (with dissipation term $\epsilon \partial_x^2 u$, $\epsilon>0$) perturbed by a convective, position-dependent noise. Later in \cite{galimberti2024global}, the authors further established the global-in-time existence result for dissipative weak martingale solutions by taking the limit $\epsilon \rightarrow 0$ in reasonable sense. Moreover, the compressible fluid flows perturbed by stochastic forcing has been systematically studied by Breit, Feireisl and Hofmanov\'{a} \cite{breit2016incompressible,breit2017compressible,breit2018local,breit2019stationary}.
It is worth pointing out that the appearance of stochastic perturbation in PDEs might lead to new phenomena. For instance, while uniqueness may fail for the deterministic transport equation, Flandoli et al. \cite{39} proved that a multiplicative stochastic perturbation of Brownian type is enough to render the equation well-posed; see also \cite{38}. In \cite{40}, Brze{\'z}niak et al. proved that the 2D Navier-Stokes system driven by degenerate noise has a unique invariant measure and hence exhibits ergodic behavior in the sense that the time average of a solution is equal to the average over all possible initial data, which is quite different with the deterministic case. Being inspired by the aforementioned results from the stochastic PDEs, especially \cite{39} and \cite{41}, it is interesting to ask the following question:

\vspace{0.3cm}
\textsf{Does proper random noise  have  regularisation effects on MEP2, allowing it to admit a global solution without additional structural conditions on initial data?}
\vspace{0.3cm}

The aim of this paper is devoted to provide an affirmative answer to the above question. More precisely, by perturbing the MEP2 with appropriate random noises, one can look for adequate conditions that allow the high-dimensional stochastic MEP2 admits global strong solution without further shape conditions. As a matter of fact, the regularization effect by noises has been discovered by several authors, for example, Flandoli \cite{39} showed that an ill-posed transport equation becomes well-posed when it was perturbed by stochastic forcing; Glatt-Holtz and Vicol \cite{41} proved that the 3D Euler equation with linear multiplicative noises admits global solutions in bounded domain; similar result has also been proved for the Boussinesq equations \cite{30}.

\subsection{Preliminaries}\label{sec1.1}

\subsubsection{Deterministic background}
Denote by $\mathscr{S}(\mathbb{T}^d;\mathbb{R}^n)$ the Schwartz space of all rapidly decreasing infinitely functions from $\mathbb{T}^d$ to $\mathbb{R}^n$. The space of tempered distributions is denoted by $\mathscr{S}'(\mathbb{T}^d;\mathbb{R}^n)$. Let $L^2(\mathbb{T}^d;\mathbb{R}^n)$ be the usual square-integrable Lebesgue space on $\mathbb{T}^d$ with the inner product and norm denoted by $(\cdot,\cdot)_{L^2}$ and $\|\cdot\|_{L^2}$, respectively. Define the complex trigonometric polynomials $e_{m}(x)=\exp(i m\cdot x)$, $m=(m_1,...,m_d)\in \mathbb{Z}^d$, and $\overline{e}_{m}$ denotes the complex conjugate. For any $s\in  \mathbb{R}$, the Sobolev space $H^s(\mathbb{T}^d;\mathbb{R}^n)$ of periodic functions can be characterized as $f\in \mathscr{S}'(\mathbb{T}^d;\mathbb{R}^n)$ such that
$$
 \|f\|_{H^s}^2\triangleq \sum_{m\in \mathbb{Z}^d}(|m|^2+1)^{s}a_{m}^2[f]<\infty,
$$
where $a_{m} [f]= (2\pi)^{-d}(f,\overline{e}_{m})_{L^2}$ denotes the Fourier coefficients of $f$. The spaces $H^s(\mathbb{T}^d;\mathbb{R}^n)$ are separable Hilbert spaces endowed with the inner product
$$
(f,g)_{H^s}=\sum_{m\in \mathbb{Z}^d}(|m|^2+1)^{s}a_{m} [f] \overline{a} _{m}[g] = ((1-\Delta)^\frac{s}{2} f,(1-\Delta)^\frac{s}{2} g)_{L^2}.
$$


Since the solutions to SMEP2 are not expected to be differentiable in time, we need to consider the fractional Sobolev spaces: For $1\leq q < \infty$, $ s\in \mathbb{R}$, the space $L^q([0,T];H^s(\mathbb{T}^d;\mathbb{R}^n))$ consists of all measurable functions $f:[0,T]\rightarrow H^s(\mathbb{T}^d;\mathbb{R}^n)$ such that
$ \int_0^T\|f(t)\|_{H^s}^qdt<\infty$. For any $\theta\in (0,1)$, we define
\begin{equation*}
\begin{split}
&W^{\theta,q}([0,T];H^s(\mathbb{T}^d;\mathbb{R}^n))\\
&\quad =\left\{f\in L^q([0,T];H^s(\mathbb{T}^d;\mathbb{R}^n)); \|f\|_{L^q([0,T];H^s)}^q+\int_0^T\int_0^T\frac{\|f(t)-f(t')\|_{H^s}^q}{|t-t'|^{1+\theta q}}\mathrm{d}t\mathrm{d}t'<\infty\right\}.
\end{split}
\end{equation*}

Since the unknown variable $(u,\gamma)$ in \eqref{1.6} is $\mathbb{R}^{d+1}$-valued and defined on $\mathbb{T}^d$, in order to write the vector field in a single form, we introduce the following notations:
\begin{equation*}
\begin{split}
\mathbb{L}^p(\mathbb{T}^d;\mathbb{R}^{d+1})&\triangleq L^p(\mathbb{T}^d;\mathbb{R}^{d})\times L^p(\mathbb{T}^d;\mathbb{R}),\quad 1\leq p \leq \infty,\\
\mathbb{H}^s(\mathbb{T}^d;\mathbb{R}^{d+1})&\triangleq H^s(\mathbb{T}^d;\mathbb{R}^{d})\times H^s(\mathbb{T}^d;\mathbb{R}),\quad s\in\mathbb{R},
\end{split}
\end{equation*}
and
\begin{equation*}
\begin{split}
\mathbb{W}^{k,p}(\mathbb{T}^d;\mathbb{R}^{d+1})&\triangleq W^{k,p}(\mathbb{T}^d;\mathbb{R}^{d})\times W^{k,p}(\mathbb{T}^d;\mathbb{R}),\quad k\geq 1,~~1\leq p \leq\infty.
\end{split}
\end{equation*}
Here for given two Banach spaces $\mathcal {X}$ and $\mathcal {Y}$, the Cartesian product space $\mathcal {X}\times\mathcal {Y}$ is again  a Banach spaces, which is equipped with the Cartesian product norm $$\|(u_1,u_2)\|_{\mathcal {X}\times\mathcal {Y}}^2=  \|u_1\|_{\mathcal {X}}^2+\|u_2\|_{\mathcal {Y}}^2  ,$$ for any $(u_1,u_2)\in \mathcal {X}\times\mathcal {Y}$. Moreover, if $\mathcal {X}$ and $\mathcal {Y}$ are Hilbert spaces, then $\mathcal {X}\times\mathcal {Y}$ is also a Hilbert space with the inner product
$$
(u,v)_{\mathcal {X}\times\mathcal {Y}}=(u_1,v_1)_{\mathcal {X}}+(u_2,v_2)_{\mathcal {Y}},\quad u=(u_1,u_2),~~v=(v_1,v_2) \in \mathcal {X}\times\mathcal {Y}.
$$
For the sake of simplicity, when a function is defined on $\mathbb{T}^d$ with values in $\mathbb{R}^n$, where $k,n$ are clear from the context, we shall omit the parentheses in notations of function spaces. For example, $\mathbb{H}^s(\mathbb{T}^d;\mathbb{R}^{d+1})=\mathbb{H}^s(\mathbb{T}^d)$, $\mathbb{W}^{k,p}(\mathbb{T}^d;\mathbb{R}^{d})=\mathbb{W}^{k,p}(\mathbb{T}^d)$ and so on.

\subsubsection{Stochastic setting}
To make sense of the stochastic forcing, let $\mathcal {S} =(\Omega,\mathcal {F},\mathbb{P},(\mathcal {F}_t)_{t\geq0})$ be a fixed complete filtered probability space, and $(\beta_j^i )_{j\geq1}$, $i=1,2$ be mutually independent real-valued standard Wiener processes relative to $(\mathcal {F}_t )_{t\geq0}$. Let $(e_j ^i)_{j\geq1}$ be a complete orthonormal system in a separate Hilbert space $\mathfrak{A}_i$, then one can formally define the mutually independent cylindrical Wiener processes $\mathcal {W}_i$ on $\mathfrak{A}_i$ by
\begin{eqnarray*}
\mathcal {W}_i(t,\omega) =\sum_{j\geq 1}  e_j^i\beta_j ^i(t,\omega),\quad i=1,2.\nonumber
\end{eqnarray*}
To ensue the convergence of the last series, we introduce an auxiliary space
\begin{eqnarray*}
\mathfrak{A}_{0,i}\triangleq \left\{u=\sum_{j\geq1} a_j^ie_j^i;~\sum_{j\geq1} \frac{(a_j^i)^2}{j^2}<\infty\right\}\supset \mathfrak{A}_i,\quad i=1,2,
\end{eqnarray*}
which is endowed with the norm $\|u\|_{\mathfrak{A}_{0,i}}^2=\sum_{j\geq1} \frac{(a_j^i)^2}{j^2}$, for any $u=\sum_{j\geq1} a_j^ie_j ^i\in \mathfrak{A}_i$, $i=1,2$. Note that the canonical injection $\mathfrak{A}_i\hookrightarrow\mathfrak{A}_{0,i}$ is Hilbert-Schmidt, which implies that $\mathcal {W} _i \in \mathcal {C}([0, T];\mathfrak{A}_{0,i})$ for any $T>0$, $\mathbb{P}$-almost surely, $i=1,2$.  We   denote by $L_2(\mathcal {Y},\mathcal {Z})$ the collection of Hilbert-Schmidt operators from a separable Hilbert space $\mathcal {Y}$ into another separable Hilbert space $\mathcal {Z}$ with the norm $
\|H\|_{L_2(\mathcal {Y};\mathcal {Z})}^2=\sum_{j\geq 1} \|H v_j\|_{\mathcal {Z}}^2<\infty $,
where $(v_j)_{j\geq1}$ is a complete orthogonal basis in $\mathcal {Y}$.

Now let $H$ be a $\mathcal {Z}$-valued predictable process in $ L^2(\Omega;L_{loc}^2([0,\infty);L_2(\mathfrak{A}_i,\mathcal {Z})))$. One can define the It\^{o} stochastic integration
\begin{eqnarray*}
\int_0^tH(r)\mathrm{d}\mathcal {W}_i=\sum_{j\geq 1} \int_0^tH(r)e_j^i\mathrm{d}\mathcal \beta_j^i(r) ,\quad i=1,2,
\end{eqnarray*}
which is actually a continuous $\mathcal {Z}$-valued square integrable martingale. Note that the above definition of the stochastic integration does not depend on the choice of $\mathfrak{A}_{0,i}$ (cf. \cite{43}).


In the following, we shall reformulate the SMEP2 into a single form. To this purpose, we define the Cartesian products $\mathfrak{A}=\mathfrak{A}_1\times \mathfrak{A}_2$ and the auxiliary space $\mathfrak{A}_0=\mathfrak{A}_{0,1}\times \mathfrak{A}_{0,2}$, then the canonical injection $\mathfrak{A}\hookrightarrow\mathfrak{A}_{0}$ is Hilbert-Schmidt, and $\mathcal {W}=(\mathcal {W}_1,\mathcal {W}_2)^T$ defines a cylindrical Wiener process on $\mathfrak{A}$, which belongs to $ \mathcal {C}_{\textrm{loc}}([0,\infty);\mathfrak{A}_0)$ $\mathbb{P}$-almost surely. Moreover, as the noise coefficients for the rewritten SMEP2 \eqref{1.8} becomes a matrix-valued Hilbert-Schmidt operator, for instance,
$$
 \begin{array}{l}
M =\left(                 
  \begin{array}{ccc}   
    M_{11}&M_{12}\cr  
    M_{21}&M_{22} \cr  
  \end{array}
\right),
\end{array}\quad
$$
with $M_{ij}\in L_2(V_j;U_j)$, $i,j\in \{1,2\}$, where $V_j$ and $U_j$, $j=1,2$ are separable Hilbert spaces, let us define the  canonical norm for $M$ by
$$
\|M\|_{L_2(V ;U )}^2\triangleq \sum_{i,j=1}^2\|M_{ij}\|_{L_2(V_j;U_j)}^2,
$$
where $V=V_1\times V_2$ and $U=U_1\times U_2 $.

\subsection{Assumptions and main results}
To give the statement of the main results for \eqref{1.4}, let us first transform the system \eqref{1.6} into convenient forms. It follows from \eqref{1.4}$_3$ that $u=\Lambda^{-2} m$, and from \eqref{1.4}$_4$ that $\gamma \triangleq \overline{\rho}-\overline{\rho}_0= \Lambda^{-2}\rho$, where $\Lambda^{s} =( 1 -\Delta)^{ \frac{s}{2}}$, $s\in \mathbb{R}$ denotes the Bessel potentials. By applying $\Lambda^{-2}$ to the first two equations in  \eqref{1.4}, and using the similar calculations for the deterministic counterpart \cite{yan2022initial}, the Cauchy problem \eqref{1.4}-\eqref{1.5} can be reformulated as
\begin{equation}\label{1.6}
\left\{
\begin{aligned}
&\mathrm{d}u+\left( u\cdot \nabla u + \mathscr{L}_1(u) + \mathscr{L}_2 (\gamma) \right)\mathrm{d} t=\Lambda^{-2}g_1(t,m,\rho)\mathrm{d}\mathcal {W}_1,\\
&\mathrm{d}\gamma+ \left(u\cdot \nabla \gamma +\mathscr{L}_3 (u,\gamma)\right)\mathrm{d} t= \Lambda^{-2}g_2(t,m,\rho)\mathrm{d}\mathcal {W}_2,\\
&u|_{t=0}=u_0= \Lambda^{-2} m_0,\\
& \gamma|_{t=0}=\gamma_0= \Lambda^{-2} \rho_0,
\end{aligned}\quad t\geq 0,~x\in \mathbb{T}^d,
\right.
\end{equation}
where
\begin{equation}\label{1.7}
\begin{split}
 \mathscr{L}_1 (u)= &\Lambda^{-2}\textrm{div}\left (\frac{1}{2}|\nabla u|^2I_d+ \nabla u \nabla u+\nabla u(\nabla u)^T-(\nabla u)^T  \nabla u-(\textrm{div}u) \nabla u  \right) \\
  &+\Lambda^{-2} \left ((\textrm{div}u) u +u\cdot ( \nabla u)^T\right),
\\
 \mathscr{L}_2 (\gamma)=& \Lambda^{-2}\textrm{div}\left (  \frac{1}{2}\left(\gamma^2+ |\nabla\gamma|^2\right)I_d-(\nabla\gamma)^T\nabla\gamma\right),\\
 \mathscr{L}_3 (u,\gamma)=&\Lambda^{-2} \textrm{div}\left (\nabla\gamma\nabla u +(\nabla\gamma)\cdot \nabla u -(\textrm{div}u)\nabla\gamma \right) + \Lambda^{-2}  \left((\textrm{div}u) \gamma\right),
\end{split}
\end{equation}
and $I_d$ denotes the $d\times d$ unit matrix. The system \eqref{1.6} can be regarded as a nonlocal transport system perturbed by the nonlinear multiplicative noise.

Defining $\textbf{y}=\begin{pmatrix}
u \\ \gamma \end{pmatrix}$, $\textbf{y}_0=\begin{pmatrix}
u_0 \\ \gamma_0 \end{pmatrix}$ and $
\mathcal {W}=\begin{pmatrix}
\mathcal {W}_1 \\ \mathcal {W}_2 \end{pmatrix}$, then \eqref{1.6}  can be understood in the following compact form:
\begin{equation}\label{1.8}
\left\{
\begin{aligned}
&\mathrm{d}\textbf{y}+ B(\textbf{y},\textbf{y})\mathrm{d} t+F(\textbf{y})\mathrm{d} t=G(t,\textbf{y}) \mathrm{d}\mathcal {W},\\
&\textbf{y}(\omega,0,x)=\textbf{y}_0(\omega,x),
\end{aligned}
\right.\quad t> 0,~x\in \mathbb{T}^d,
\end{equation}
where the bilinear form $B(\cdot,\cdot)$ is defined by
$$
B(\textbf{y}_1,\textbf{y}_2)=
\begin{pmatrix}
u_1\cdot \nabla u_2\\
     u_1\cdot \nabla \gamma_2
\end{pmatrix}
, \quad \textrm{for any}~~ \textbf{y}_i=
\begin{pmatrix}
u_i \\ \gamma_i
\end{pmatrix}, \quad i=1,2,
$$
and the nonlinear terms are defined by
$$
F(\textbf{y})=
\begin{pmatrix}
\mathscr{L}_1(u) + \mathscr{L}_2 (\gamma)\\
\mathscr{L}_3(u,\gamma)
\end{pmatrix},\quad G(t,\textbf{y})=\begin{pmatrix}
\Lambda^{-2}g_1(t,m,\rho)&0 \\
0&\Lambda^{-2}g_2(t,m,\rho)
\end{pmatrix}.
$$
for any $m=\Lambda^2 u$ and $\rho=\Lambda^2 \gamma$. Note that   \eqref{1.8} does not have the cancelation property, i.e., $(B(\textbf{y},\textbf{y}),\textbf{y})_{\mathbb{L}^2}=0$, due to the loss of divergence-free condition $\nabla\cdot  \textbf{y}=0$, which makes the construction of approximate solutions to be more subtle.

Let us give the rigorous definition of local/global strong pathwise solutions to \eqref{1.8}.

\begin{definition}\label{def1}
Let $s> 1+\frac{d}{2}$, $d\geq1$, and the initial data $\textbf{y}_0 \in \mathbb{H}^s(\mathbb{T}^d)$ be a $\mathcal {F}_0$-measurable random variable such that $\mathbb{E}(\|\textbf{y}_0\|_{\mathbb{H}^s }^2)<\infty$.

\begin{itemize}
\item [(1)]  A \textsf{local strong pathwise solution} of SMEP2 \eqref{1.8} is a pair $(\textbf{y}_0,\mathbbm{t})$, where $\mathbbm{t}$ is a $\mathbb{P}$-almost surely positive stopping time, i.e., $\mathbb{P}\{\mathbbm{t}>0\}=1$, and $\textbf{y}(\cdot)$ is a $\mathbb{H}^s(\mathbb{T}^d)$-valued $\mathcal {F}_t$-predictable processes satisfying
 $
\textbf{y}(\cdot\wedge \mathbbm{t}) \in L^2(\Omega, \mathcal {C} ([0,\infty),\mathbb{H}^s(\mathbb{T}^d)) $,
and
$$
 \textbf{y}(t\wedge \mathbbm{t})+ \int_0^{t\wedge \mathbbm{t}}B(\textbf{y}(r),\textbf{y}(r))\mathrm{d} r+\int_0^{t\wedge \mathbbm{t}}F(\textbf{y}(r))\mathrm{d} r=\textbf{y}_0+\int_0^{t\wedge \mathbbm{t}}G(r,\textbf{y}(r)) \mathrm{d}\mathcal {W}(r),
$$
for all $t>0$, $\mathbb{P}$-almost surely.

\item [(2)] The strong pathwise solution $(\textbf{y},\bar{\mathbbm{t}})$ is said to be \textsf{maximal}, if $\mathbb{P}\{\bar{\mathbbm{t}}>0\}=1$ and there is a sequence of stopping times $\mathbbm{t}_n$ increasingly tending to $\bar{\mathbbm{t}}$ as $n\rightarrow\infty$ such that for any $n\in\mathbb{N}^+$, $(\textbf{y},\mathbbm{t}_n)$ is a local strong pathwise solution such that
$$
 \sup_{t\in[0,\mathbbm{t}_n]}\|\textbf{y}(t)\|_{\mathbb{W}^{1,\infty}}> n  \quad \mbox{on}~ \{\mathbbm{t}<\infty\}.
$$
In addition, if $\mathbb{P} \{\bar{\mathbbm{t}}=\infty\}=1$, then the solution is said to be \textsf{global}.

\item [(3)] The strong pathwise solution is said to be \textsf{pathwise unique}, if for any given two local strong pathwise solutions $(\textbf{y}_1,\mathbbm{t}_1)$ and $(\textbf{y}_2,\mathbbm{t}_2)$, we have
$$
\mathbb{P}\left\{\textbf{1}_{\{\textbf{y}_1(0)=\textbf{y}_2(0)\}}\left(\textbf{y}_1(t)=\textbf{y}_2(t)\right),~~~ \forall t\in [0,\mathbbm{t}_1\wedge\mathbbm{t}_2]\right\}=1.
$$
\end{itemize}
\end{definition}

Note that in Definition 1.1, the local/global solutions are required to be strong both in the probabilistic sense (the stochastic basis is presupposed) and in the PDE sense.

The following assumptions will be valid throughout this paper.

\begin{assumption} \label{assume}
For all $s>1+\frac{d}{2}$,  we assume that the functions $g_i(t,m,\rho): [0,\infty)\times H^{s-2}(\mathbb{T}^3)\times H^{s-2}(\mathbb{T}^3)\mapsto \mathcal {L}_2(\mathfrak{A}_i,H^{s-2}(\mathbb{T}^3))$, $i=1,2$, are continuous in $(t,m,\rho)$ and the following conditions hold true:
\begin{itemize}
\item [(1)] (\textsf{Growth condition}) There exists two non-decreasing locally bounded continuous scaler functions $\mu_i,\chi_i:\mathbb{R}^+\mapsto\mathbb{R}^+$ such that for $m=\Lambda^{2}u$ and $\rho=\Lambda^{2}\gamma$
\begin{equation*}
\begin{split}
\|g_i(t,m,\rho)\|_{\mathcal {L}_2(\mathfrak{A}_i,H^{s-2})} \leq \mu_i (t) \chi_i(\|(u,\gamma)\|_{\mathbb{W}^{1,\infty}})(1+\|(u,\gamma)\|_{\mathbb{H}^{s}}),\quad i=1,2.
\end{split}
\end{equation*}
\item [(2)] (\textsf{Locally Lipschitz continuity})  There exists two non-decreasing locally bounded continuous scaler functions $\tilde{\mu}_2,\tilde{\chi}_2:\mathbb{R}^+\mapsto\mathbb{R}^+$ such that for $m_j=\Lambda^{2}u_j$ and $\rho_j=\Lambda^{2}\gamma_j$, $j=1,2$,
\begin{equation*}
\begin{split}
&\|g_i(t,m_1,\rho_1)-g_i(t,m_2,\rho_2)\|_{\mathcal {L}_2(\mathfrak{A}_1,H^{s-2})} \\
&\quad \leq \tilde{\mu}_i (t) \tilde{\chi}_i(\|(u_1,\gamma_1)\|_{\mathbb{H}^s}+\|(u_2,\gamma_2)\|_{\mathbb{H}^s})
\|(u_1-u_2,\gamma_1-\gamma_2)\|_{\mathbb{H}^{s}}
,\quad i=1,2.
\end{split}
\end{equation*}
\end{itemize}
\end{assumption}

\begin{remark} \label{rem1}  Comments on the above assumptions are provided.
\begin{itemize}
\item [(1)] An explicit example for Assumption \ref{assume} is as follows:
\begin{equation*}
\begin{split}
g_1(t,m,\rho) =c (1+\|\textbf{y}\|_{\mathbb{W}^{1,\infty}})^{\delta} m\quad \textrm{and} \quad g_2(t,m,\rho)&=c (1+\|\textbf{y}\|_{\mathbb{W}^{1,\infty}})^{\delta} \rho,
\end{split}
\end{equation*}
for all $m=\Lambda^{2}u$ and $\rho=\Lambda^{2}\gamma$, where $c $ and $\delta$ are non-negative numbers which describe the intensity of the effect of   random noises on the  equations, and the Wiener processes are replaced by a standard one-dimensional Brownian motions $W(t)$.  This type of noise  will be considered in analysing the existence of global strong solutions (cf. Theorem \ref{th3} and Theorem \ref{th4}) and the blow-up phenomena (cf. Theorem \ref{th5}).

\item [(2)] Note that the pseudo-differential operator $\Lambda^{-2}=(1-\Delta)^{-1}$ is a $S^{-2}$ multiplier (cf.  \cite[Proposition 2.78]{bahouri2011fourier}), which implies that, under the conditions in Assumption \ref{assume}, the diffusion  matric $G(t,\textbf{y})$ satisfies
\begin{equation}\label{a1}
\begin{split}
\|G(t,\textbf{y})\|_{\mathcal {L}_2(\mathfrak{A},\mathbb{H}^{s})} \leq \mu (t) \chi(\|\textbf{y}\|_{\mathbb{W}^{1,\infty}})(1+\|\textbf{y}\|_{\mathbb{H}^{s}}),
\end{split}
\end{equation}
and
\begin{equation}\label{a2}
\begin{split}
\|G(t,\textbf{y})-G(t,\textbf{z})\|_{\mathcal {L}_2(\mathfrak{A},\mathbb{H}^{s})} \leq \tilde{\mu} (t) \tilde{\chi}(\|\textbf{y}\|_{\mathbb{H}^{s}}
+\|\textbf{z}\|_{\mathbb{H}^{s}})\|\textbf{y}-\textbf{z}\|_{\mathbb{H}^{s}},
\end{split}
\end{equation}
where $\mu \triangleq\max\{\mu _1,\mu _2 \}$, $\chi  \triangleq\max\{\chi _1,\chi _2 \}$, $\tilde{\mu } \triangleq\max\{\tilde{\mu} _1 ,\tilde{\mu} _2\}$ and $\tilde{\chi}  \triangleq\max\{\tilde{\chi} _1 ,\tilde{\chi} _2 \}$ are locally bounded nondecreasing continuous functions.
\end{itemize}
\end{remark}

\vspace{0.1cm}

Our first main result is concerned with the local well-posedness of strong pathwise solution to the SMEP2 driven by nonlinear multiplicative noise.
\begin{theorem} [\textsf{Local existence with general noises}]\label{th1}
Let $s>1+\frac{d}{2}$, $d\geq1$, and $\textbf{y}_0$ be a $\mathbb{H}^s$-valued $\mathcal {F}_0$-measurable random variable such that $\mathbb{E}\|\textbf{y}_0\|_{\mathbb{H}^s}^2<\infty$. Under the  Assumption \ref{assume}, we obtain the following conclusions:
\begin{itemize}
\item [(1)] The system \eqref{1.8} admits a unique maximal local strong pathwise solutions $(\textbf{y},\mathbbm{t})$ in the sense of Definition \ref{def1}. Moreover, for any fixed $\epsilon>0$ and $T>0$, there is a sufficiently small  $\delta=\delta(\epsilon,T,\textbf{y}_0)>0$ such that if
$
\|\textbf{y}_0-\textbf{z}_0\|_{L^ \infty(\Omega;\mathbb{H}^s)}<\delta,$
then a stopping time $\mathbbm{t}\in (0,T]$ exists such that
\begin{equation*}
\begin{split}
\mathbb{E}\sup_{t\in [0,\mathbbm{t}]}\|\textbf{y}(t)-\textbf{z}(t)\|_{\mathbb{H}^s} ^2 <\epsilon,\quad \mathbb{P}\textrm{-a.s.}
\end{split}
\end{equation*}

\item [(2)] The local solution $(\textbf{y},\mathbbm{t})$ is also a $\mathbb{W}^{1,\infty}$-valued $\mathcal {F}_t$ adapted process for all $t< \mathbbm{t} $, and the norm inflation of $\| \textbf{y}(t)\|_{\mathbb{H}^s}$ and the norm inflation $\| \textbf{y} \|_{\mathbb{W}^{1,\infty}}$ has the following relationship:
$$
 \mathbb{P}\left(\textbf{1}_{\{\limsup \limits_{t\rightarrow\mathbbm{t}}\| \textbf{y}(t)\|_{\mathbb{H}^s}=\infty\}} =\textbf{1}_{\{\limsup\limits _{t\rightarrow\mathbbm{t}}\| \textbf{y}(t)\|_{\mathbb{W}^{1,\infty}}=\infty\}}\right)=1.
$$
\end{itemize}
\end{theorem}

\begin{remark} \label{rem2} We would like to make a few comments on Theorem \ref{th1}:
\begin{itemize}
\item [(1)] The proof of Theorem \ref{th1}(1) relies on looking at the SMEP2 as a system of SDEs in Hilbert spaces due to the lack of \textsf{cancelation property}, i.e., $(B(\textbf{y}_\epsilon,\textbf{y}_\epsilon),\textbf{y}_\epsilon)_{\mathbb{L}^2}=0$, and this can be achieved by mollifying the convection terms $u\cdot \nabla u$ and $u\cdot \nabla \gamma$ in \eqref{1.8}. The main difficulty in carrying out this construction is the appearance of the norm $\|\textbf{y}\|_{\mathbb{W}^{1,\infty}}$ in $L^2$ moment estimates (cf. \eqref{2.12}), which prevent us from closing the a priori estimate for $\textbf{y}_\epsilon$ in $L^2(\Omega;H^s(\mathbb{T}^d))$. The usual approach is to introduce the exiting times $\mathbbm{t}_\epsilon=\inf_{t\geq0}\{\|\textbf{y}_\epsilon(t)\|_{\mathbb{W}^{1,\infty}}>r\}$ for $r>0$. However, the current case is strongly  different from the deterministic counterpart \cite{yan2022initial}, due to the lack of  the efficient method for estimating $\inf_{\epsilon>0} \mathbbm{t}_\epsilon$ which may degenerate  to zero. To overcome this difficulty, we shall introduce $W^{1,\infty}$-truncation  functions to the nonlinear terms in system \eqref{1.8} to obtain new approximations $\{\textbf{y}_{R,\epsilon}\}$. The second difficulty arises  from the loss of the compact embedding from $L^2(\Omega; \mathcal {X})$ into $L^2(\Omega; \mathcal {Y})$ even though $\mathcal {X}\subset\subset\mathcal {Y}$ (i.e., $\mathcal{X}$ is compactly embedded into $\mathcal{Y}$), and thus one can not directly extract a  convergence subsequence of $\textbf{y}_{R,\epsilon}$. Our method is first to prove the tightness of the measures $\{\mu_{R,\epsilon}\}$ induced by $\{\textbf{y}_{R,\epsilon}\}$. Then we prove that the regularized SMEP2 with truncation admits a smooth global martingale solution when $s>4+\frac{d}{2}$. After proving a pathwise uniqueness result for the SMEP2, one can prove by Gy\"{o}ngy-Krylov Lemma (cf. \cite[Theorem 2.8]{47}) that the original system \eqref{1.8} has a local unique strong pathwise solution in $\mathbb{H}^s(\mathbb{T}^d)$ with $s>4+\frac{d}{2}$. Thanks to the density embedding $H^s(\mathbb{T}^d)\subset H^{s-3}(\mathbb{T}^d)$, it is successfully proved by a density-stability argument (cf. \cite{41}) that the SMEP2 admits a local pathwise solution in the sharp case of $s>1+\frac{d}{2}$.


\item [(2)]  Theorem \ref{th1}(2) provides a blow-up criteria of strong pathwise solution in Sobolev spaces, which   informs us that
although the Sobolev embedding implies $\| \textbf{y} \|_{\mathbb{W}^{1,\infty}}\leq C\| \textbf{y}\|_{\mathbb{H}^s}$, for $s>1+\frac{d}{2}$,  the $\mathbb{H}^s$-norm of the solution $\textbf{y}(t)$ will not blow up before the $\mathbb{W}^{1,\infty}$-norm of $\textbf{y}(t)$. Indeed, this characteristic also appears in the study for classic CH-type equation in one dimension (cf. \cite{constantin1-2000,constantin1998,guan2010well}), which has been used as a cornerstone to prove the wave breaking mechanism in finite time.

\item [(3)]  Note that the system (1.11) is driven by two independent noises $\mathcal {W}_1$ and $\mathcal {W}_2$, while the case that $\mathcal {W}_1=\mathcal {W}_2= \mathscr{W}$ for some cylindrical Wiener process $\mathscr{W}$
can be handled in the same way. There is no loss of generality in assuming that there are two separable Hilbert spaces $\mathcal {U}$ and $\mathcal {U}_0$ such that $\mathcal {U}\mapsto \mathcal {U}_0$ is Hilbert-Schmidt. Indeed, assuming that $\{e_k\}$ is a complete orthonormal basis of $\mathcal {U}$ and $\{W_k\}$ be a sequence of one-dimensional Brownian motion, then we can define $\mathscr{W}=\sum_{k\geq 1} e_kW_k \in \mathcal {C}([0,T];\mathcal {U}_0)$, and the system (1.11) can be written as
$$\mathrm{d}\textbf{y}+ B(\textbf{y},\textbf{y})\mathrm{d} t+F(\textbf{y})\mathrm{d} t=\mathscr{G}(t,\textbf{y}) \mathrm{d}\mathscr{W},$$
where $B(\textbf{y},\textbf{y})$, $F(\textbf{y})$ are defined as before, and
$
\mathscr{G}(t,\textbf{y})=\begin{pmatrix}
\Lambda^{-2}g_1(t,m,\rho)\\
\Lambda^{-2}g_2(t,m,\rho)
\end{pmatrix}.
$
 \end{itemize}
 \end{remark}

Now let us address the global existence problem by virtue of the random noises with specific structure. The first result is stated by the following theorem.

\begin{theorem} [\textsf{I. Global existence with large initial data}] \label{th3}

Let $s>1+\frac{d}{2}$, $d\geq1$, and $(u_0,\gamma_0)$ be a $\mathbb{H}^s$-valued $\mathcal {F}_0$-measurable initial random variable in $L^2(\Omega;\mathbb{H}^s(\mathbb{T}^d))$. Assume that the real valued parameters $ \delta$ and $c$ satisfy one of the following two conditions:
$$
\delta>\frac{1}{2}\quad and \quad c\neq0,
$$
or
$$
\delta=\frac{1}{2} \quad and \quad |c|> \sqrt{\varrho},
$$
where $\varrho>0$ is a general constant obtained in the estimate \eqref{4..3} below.  Then the corresponding local maximal strong solution $(u,\gamma,\mathbbm{t})$ to the system
\begin{equation}\label{1.14}
\left\{
\begin{aligned}
&\mathrm{d}u+\left( u\cdot \nabla u + \mathscr{L}_1(u) + \mathscr{L}_2 (\gamma) \right)\mathrm{d} t=c(1+\|(u,\gamma)\|_{\mathbb{W}^{1,\infty}})^{\delta} u\mathrm{d}W,\\
&\mathrm{d}\gamma+ \left(u\cdot \nabla \gamma +\mathscr{L}_3 (u,\gamma)\right)\mathrm{d} t= c(1+\|(u,\gamma)\|_{\mathbb{W}^{1,\infty}})^{\delta} \gamma\mathrm{d}W,\\
& u(0,x)=u_0(x),~~~\gamma (0,x)= \gamma_0 (x),
\end{aligned}\quad t\geq 0,~x\in \mathbb{T}^d
\right.
\end{equation}
exists globally in time $\mathbb{P}$-almost surely, that is,  $\mathbb{P}\{\mathbbm{t}=\infty\}=1$.
\end{theorem}

\begin{remark} \label{rem3} 
As far as we know, Theorem \ref{th3} seems to be the first result concerning  the existence of global solutions for MEP2 in high dimensions from the probability point of view, which informs us that the random noises $c(1+\|\textbf{y}\|_{\mathbb{W}^{1,\infty}})^{\delta} \textbf{y}$ of nonlocal-type (since the norm $\|\textbf{y}\|_{\mathbb{W}^{1,\infty}}$ depends on the whole information of $\textbf{y}$ on $\mathbb{T}^d$) with proper intensity have a regularization effect to the solutions of MEP2.
\end{remark}

Whether or not the MEP2 perturbed by nonlocal-type noises admits global solutions when the intensity belongs to the region $0\leq\delta<\frac{1}{2}$ (maybe with some restrictions on $c$)?
The next theorem provides a partial positive answer to above problem when $\delta=0$.

\begin{theorem} [\textsf{II. Global existence with small initial data}] \label{th4}
Let $s>1+\frac{d}{2}$, $d\geq1$ and $ c \neq0$. Assume that $(u_0,\gamma_0)$ is a $\mathbb{H}^s$-valued $\mathcal {F}_0$-measurable random variable in $L^2(\Omega;\mathbb{H}^s(\mathbb{T}^d))$. Let $(u,\gamma,\mathbbm{t})$ be the corresponding maximal strong pathwise solution to the system
\begin{equation}\label{1.15}
\left\{
\begin{aligned}
&\mathrm{d}u+\left( u\cdot \nabla u + \mathscr{L}_1(u) + \mathscr{L}_2 (\gamma) \right)\mathrm{d} t=c u\mathrm{d}W,\\
&\mathrm{d}\gamma+ \left(u\cdot \nabla \gamma +\mathscr{L}_3 (u,\gamma)\right)\mathrm{d} t= c \gamma \mathrm{d}W,\\
& u(0,x)=u_0(x),~~~\gamma (0,x)= \gamma_0 (x),
\end{aligned}\quad t\geq 0,~x\in \mathbb{T}^d.
\right.
\end{equation}
For any arbitrary parameters $R>1$ and $\kappa \geq2$, there exists a positive constant $\widetilde{C}$, depending only on $s$ and $d$, such that whenever
\begin{equation}\label{tt}
\begin{split}
\|(u_0,\gamma_0)\|_{\mathbb{H}^s}\leq \frac{c^2}{2\widetilde{C}R\kappa} ,\quad \mathbb{P}\textrm{-a.s.,}
\end{split}
\end{equation}
then we have $
 \mathbb{P}\{\mathbbm{t}=\infty\} \geq  1- \frac{1}{R^{ \frac{\kappa-1}{2\kappa}}}$.
In other words, the strong pathwise solution to the system \eqref{1.15} exists globally in time with high probability.
\end{theorem}

\begin{remark}\label{rem4}
In view of  Theorem \ref{th4}, when the nonlinear multiplicative noise is reduced to the linear case (i.e., $\delta=0$ in \eqref{1.14}, which means that the strength of the noise becomes weaker), the SMEP2 still has a global strong pathwise solution for sufficiently small initial data. Moreover, estimate \eqref{tt} implies that the smaller norm $\|(u_0,\gamma_0)\|_{\mathbb{H}^s}$ (resp. $\kappa$ is larger) corresponds to larger probability $ \mathbb{P}\{\mathbbm{t}=\infty\}$   (resp. $1- 1/R^{ \frac{\kappa-1}{2\kappa}}$ is larger).
\end{remark}

Our final result gives a negative answer to the global existence with proper structure condition on initial data. Due to the technique reasons,  the system will be restricted in one dimension, that is,
\begin{equation}\label{1.19}
\left\{
\begin{aligned}
&\mathrm{d}u+ (u u_x +\partial_xG\star (u^2+\frac{1}{2}u_x^2+\frac{1}{2}\gamma^2-\gamma_x^2))\mathrm{d} t=c  u\mathrm{d}W,\\
&\mathrm{d}\gamma+u\gamma_x+ G\star ((u_x\gamma_x)_x+u_x\gamma)=c \gamma\mathrm{d}W ,\\
&u|_{t=0}=u_0,\quad \gamma|_{t=0}=\gamma_0 ,
\end{aligned}\quad t\geq 0,~x\in \mathbb{T}^d.
\right.
\end{equation}
Here the sign $\star$ denotes the spatial convolution, and $G(\cdot)$ is the associated Green's function of the operator $\Lambda^{-2}=(1-\partial_x^2)^{-1}$, which can be formulated explicitly by
$$
\Lambda^{-2} f= G\star f, \quad G(x)=\frac{\cosh (x-2\pi [\frac{x}{2\pi}]-\pi)}{2\sinh (\pi)},\quad \forall f\in L^2(\mathbb{T}).
$$

Our main result is stated by the following theorem.
\begin{theorem} [\textsf{Blow-up criteria}]\label{th5}
Let $s>\frac{3}{2}$, $ c\in \mathbb{R} \backslash \{0\}$, $\lambda\in (0,1)$ and $(u,\gamma,\mathbbm{t})$ be the unique local maximal pathwise solution to system \eqref{1.19} with respect to a $\mathcal {F}_0$-measurable initial data $(u_0,\gamma_0)$ in $L^2(\Omega;\mathbb{H}^s(\mathbb{T}))$ in the sense of Definition \ref{def1}. If there exists a point $x_0\in \mathbb{T}$ such that
\begin{equation}\label{1.17}
\begin{split}
\mathbb{P}\left\{(\partial_xu_0)(x_0)< -\frac{ c^2  }{2\lambda}-\sqrt{\frac{c^4}{4\lambda^2}+  \|u_0\|_{H^1}^2+\|\gamma_0\|_{H^1}^2} \right\}=1.
\end{split}
\end{equation}
Then the solution $(u,\gamma,\mathbbm{t})$ will blow up in finite time. Moreover, the wave breaking phenomena occurs with the positive probability
$$
\mathbb{P}\left\{\lim _{t\rightarrow \mathbbm{t}} \inf_{x\in \mathbb{T}}u_x(t,x)=-\infty\right\} >0.
$$
\end{theorem}

\begin{remark}\label{rem5}
Comparing Theorem \ref{th5} with Theorem \ref{th4}, one find that the effect of the structure of initial data (cf. \eqref{1.17}) on existence results is larger than that of linear multiplicative noise. For any sufficiently small initial data $(u_0,\gamma_0)$, although Theorem \ref{th4} (when $d=1$) ensures the existence of global solutions, only one point property for the initial data will make the solution blow up in finite time.
\end{remark}

\subsection{Plan of the paper}

Section \ref{section2} is devoted to the proof of Theorem \ref{th1}. First, we establish the relationship between the exploring time of $\|\textbf{y}\|_{\mathbb{H}^s}$ and the exploring time of $\|\textbf{y}\|_{W^{1,\infty}}$, which is also important in the proof of Theorem \ref{th5}. Then by using the Gy\"{o}ngy-Krylov lemma and the abstract Cauchy theorem, we establish the local well-posedness for the SMEP2 in $\mathbb{H}^s(\mathbb{T}^d)$ with $s>1+\frac{d}{2}$. In Section \ref{section3}, we first prove that the MEP2 perturbed by nonlocal-type noise admits a unique global strong solution in the regime $\delta \geq \frac{1}{2}$.  Then when $\delta=0$,  we show that the SMEP2 with small initial data has a unique global-in-time solution with high probability. Finally, we establish a blow-up criteria in one dimension.

\section{Local well-posedness}\label{section2}
The aim of this section is to prove Theorem \ref{th1}. To this end, we would like to first prove the second part Theorem \ref{th1}(2), the proof of the existence and uniqueness of local strong solutions is long, which will be achieved in next sections by applying approximate scheme.

\subsection{Proof of Theorem \ref{th1}(2)}

\begin{proof}[\emph{\textbf{Proof.}}]  The proof will be divided into two steps.

{\textsf{Step 1:}} We prove that for any $m,n\in \mathbb{N}^+$, if $\mathbbm{t}_{1}$ and $\mathbbm{t}_{2}$ are stopping times defined by
\begin{equation*}
\begin{split}
&\mathbbm{t}_{1}=\lim\limits_{m\rightarrow\infty}\mathbbm{t}_{1,m},\quad \textrm{where}~~\mathbbm{t}_{1,m}=\inf\{t\geq 0; ~ \|\textbf{y}(t)\|_{\mathbb{H}^s}\geq m\},\\
 &\mathbbm{t}_{2}= \lim\limits_{n\rightarrow\infty}\mathbbm{t}_{2,n},\quad \textrm{where}~~ \mathbbm{t}_{2,n}=\inf\{t\geq 0; ~ \|\textbf{y}(t)\|_{\mathbb{W}^{1,\infty}}\geq n\}.
\end{split}
\end{equation*}
then we have
\begin{equation}\label{yy}
\begin{split}
\mathbbm{t}_{1}=\mathbbm{t}_{2},\quad \mathbb{P}\textrm{-a.s.}
\end{split}
\end{equation}
Clearly, $\mathbbm{t}_{1,m}$ and $\mathbbm{t}_{2,n}$ are both nondecreasing stopping times. Indeed, from the Sobolev embedding $\mathbb{H}^s(\mathbb{T}^d)\hookrightarrow \mathbb{W}^{1,\infty}(\mathbb{T}^d)$ for $s>1+\frac{d}{2}$, we infer that there exists a constant $C>0$ such that
$\|\textbf{y}\|_{\mathbb{W}^{1,\infty}}\leq C \|\textbf{y}\|_{\mathbb{H}^s}$. Then we get
\begin{equation*}
\begin{split}
\sup_{t\in [0,\mathbbm{t}_{1,m}]} \|\textbf{y}(t)\|_{\mathbb{W}^{1,\infty}} \leq C \sup_{t\in [0,\mathbbm{t}_{1,m}]}  \|\textbf{y}(t)\|_{\mathbb{H}^s}\leq ([C]+1)m.
\end{split}
\end{equation*}
Hence, it follows from the definition of $\mathbbm{t}_{1,m}$ and $\mathbbm{t}_{2,n}$ that $\mathbbm{t}_{1,m}\leq \mathbbm{t}_{2,([C]+1)m}\leq \mathbbm{t}_{2}$, for all $m\geq1$, which implies that
\begin{equation}\label{2.1}
\begin{split}
 \mathbb{P}\{\mathbbm{t}_{1}\leq \mathbbm{t}_{2}\}=1.
\end{split}
\end{equation}

Next, we prove the inverse inequality, i.e.,
$$
\mathbb{P}\{\mathbbm{t}_{1}\geq\mathbbm{t}_{2}\}=1,
$$
which combined with \eqref{2.1} leads to \eqref{yy}. To this end, we first claim that
\begin{equation}\label{2.2}
\begin{split}
 \mathbb{P}\left\{\sup_{t\in [0,\mathbbm{t}_{2,n}\wedge k]}\|\textbf{y}(t)\|_{\mathbb{H}^s}<\infty\right\}=1,\quad \textrm{for all}~ n,k\in\mathbb{N}^+.
\end{split}
\end{equation}
Since the strong solution $\textbf{y}(t)$ is a $\mathcal {F}_t$-predictable process taking velues in $\mathbb{H}^s(\mathbb{T}^d)$, the convection term $B(\textbf{y},\textbf{y})=(u\cdot \nabla u, u\cdot \nabla \gamma)$ is just a $\mathbb{H}^{s-1}$-valued process, which prevents us applying the It\^{o} formula in Hilbert space to  \eqref{1.8}. We overcome this difficulty by regularizing   SMEP2 via the Fredriches mollifier $J_\epsilon$ (cf. \cite[Section 2]{tang2020noise}) as
\begin{equation}\label{2.3}
\begin{split}
\mathrm{d}J_\epsilon\textbf{y}+ J_\epsilon B(\textbf{y},\textbf{y})\mathrm{d} t+ J_\epsilon F(\textbf{y})\mathrm{d} t=J_\epsilon G(t,\textbf{y}) \mathrm{d}\mathcal {W}.
\end{split}
\end{equation}
Then \eqref{2.3} can be regarded as a system of SDEs in $\mathbb{H}^s(\mathbb{T}^d)$. By applying the It\^{o} formula in Hilbert spaces (cf.  \cite[Theorem 4.32]{43}) to $\|J_\epsilon\textbf{y}(t)\|_{\mathbb{H}^s}^2$, we get
\begin{equation}\label{2.4}
\begin{split}
\|J_\epsilon\textbf{y}(t)\|_{\mathbb{H}^s}^2=& \|J_\epsilon\textbf{y}(0)\|_{\mathbb{H}^s}^2+2\int_0^t (J_\epsilon\textbf{y}, J_\epsilon G(r,\textbf{y})\mathrm{d}\mathcal {W})_{\mathbb{H}^s}\\
   &+ 2\int_0^t(\Lambda^sJ_\epsilon\textbf{y},\Lambda^sJ_\epsilon(B(\textbf{y},\textbf{y}) +J_\epsilon F(\textbf{y})))_{\mathbb{L}^2}\mathrm{d} r +\int_0^t \|J_\epsilon G(r,\textbf{y})\|_{L_2(\mathfrak{U}_1,\mathbb{H}^{s})}^2\mathrm{d}r.
\end{split}
\end{equation}
By using the Burkholder-Davis-Gundy (BDG) inequality (cf. \cite{43}) and   Assumption \ref{assume}, we get for any $k\geq 1$
\begin{equation}\label{2.5}
\begin{split}
&\mathbb{E}\sup_{t\in[0,\mathbbm{t}_{2,n}\wedge k]}\left|\int_0^t (J_\epsilon\textbf{y} , J_\epsilon G(r,\textbf{y} )\mathrm{d}\mathcal {W})_{\mathbb{H}^s} \right|\\
&\quad \leq C\mathbb{E}  \left(\sup_{t\in [0,\mathbbm{t}_{2,n}\wedge k]}\|J_\epsilon\textbf{y} (t) \|_{\mathbb{H}^{s}}^2\int_0^{\mathbbm{t}_{2,n}\wedge k} \mu^2  (t)\chi ^2( \|J_\epsilon\textbf{y} \|_{\mathbb{W}^{1,\infty}} )(1+\|J_\epsilon\textbf{y}(r) \|_{\mathbb{H}^{s}})^2 \mathrm{d}r\right)^{\frac{1}{2}}\\
&\quad\leq \frac{1}{2}\mathbb{E}\sup_{t\in [0,\mathbbm{t}_{2,n}\wedge k]}\|J_\epsilon\textbf{y}(t) \|_{\mathbb{H}^{s}}^2+C_{R}\mathbb{E}\int_0^{\mathbbm{t}_{2,n}\wedge k} \mu^2  (r)\chi ^2( \|J_\epsilon\textbf{y} \|_{\mathbb{W}^{1,\infty}} )(1+\|J_\epsilon\textbf{y} (r)\|_{\mathbb{H}^{s}})^2 \mathrm{d}r\\
&\quad\leq \frac{1}{2}\mathbb{E}\sup_{t\in [0,\mathbbm{t}_{2,n}\wedge k]}\|J_\epsilon\textbf{y}(t) \|_{\mathbb{H}^{s}}^2+C_{R}\chi ^2( n)\mathbb{E}\int_0^{\mathbbm{t}_{2,n}\wedge k} \mu^2  (r)(1+\|\textbf{y} (r)\|_{\mathbb{H}^{s}}^2) \mathrm{d}r,
\end{split}
\end{equation}
where the last inequality used the boundedness of $J_\epsilon$ (cf. \cite[estimate (2.2)]{tang2020noise}) and the definition of $\mathbbm{t}_{2,n}$.
Similarly, the forth term on the R.H.S. of \eqref{2.4} can be estimated as
\begin{equation}\label{2.6}
\begin{split}
 \mathbb{E}\sup_{t\in[0,\mathbbm{t}_{2,n}\wedge k]}\left|\int_0^t \|J_\epsilon G(r,\textbf{y})\|_{L_2(\mathfrak{U}_1,\mathbb{H}^{s})}^2\mathrm{d}r\right| \leq  C_{R}\chi ^2( n)\mathbb{E}\int_0^{\mathbbm{t}_{2,n}\wedge k} \mu^2  (t) (1+\|\textbf{y} (r)\|_{\mathbb{H}^{s}})^2 \mathrm{d}r.
\end{split}
\end{equation}
For the third term on the R.H.S.  of \eqref{2.4}, we get by using the symmetry property of $J_\epsilon$ (cf. \cite{tang2020noise}) that
\begin{equation}\label{2.7}
\begin{split}
&(\Lambda^sJ_\epsilon\textbf{y},\Lambda^s J_\epsilon B(\textbf{y},\textbf{y})  )_{\mathbb{L}^2}+(\Lambda^sJ_\epsilon\textbf{y},\Lambda^sJ_\epsilon F(\textbf{y}) )_{\mathbb{L}^2}\\
&\quad = (\Lambda^sJ_\epsilon^2 u,[\Lambda^s,  (u\cdot \nabla)] u)_{L^2} +(\Lambda^sJ_\epsilon^2 u,(u\cdot \nabla) \Lambda^s u)_{L^2}\\
&\quad \quad+(\Lambda^sJ_\epsilon^2 \gamma,[\Lambda^s,  (u\cdot \nabla)] \gamma)_{L^2} +(\Lambda^sJ_\epsilon^2 \gamma,(u\cdot \nabla) \Lambda^s \gamma)_{L^2}\\
&\quad \quad+(\Lambda^sJ_\epsilon u,\Lambda^sJ_\epsilon ( \mathscr{L}_1(u) + \mathscr{L}_2 (\gamma))_{L^2}+(\Lambda^sJ_\epsilon \gamma,\Lambda^s J_\epsilon\mathscr{L}_3(u,\gamma))_{L^2}\\
 &\quad\triangleq  A_{\epsilon,1}+ A_{\epsilon,2}+\cdots+ A_{\epsilon,6}.
\end{split}
\end{equation}
To estimate the first and third terms in \eqref{2.7}, we need the following commutator estimate:
\begin{lemma}[\cite{kato1988commutator}] \label{lem:commutator}
Let $s>0$, and $f,g\in H^s(\mathbb{T}^d)\cap W^{1,\infty}(\mathbb{T}^d)$, there holds
$$
\|[\Lambda^s,f\cdot \nabla]g\|_{L^2}\leq C \left(\|\Lambda^s f \|_{L^{2}} \|\nabla g\|_{L^\infty} +  \|\nabla f\|_{L^\infty} \|\Lambda^{s-1}g\|_{L^2} \right).
$$
\end{lemma}
By Lemma \ref{lem:commutator}, we have
\begin{equation}\label{2.8}
\begin{split}
A_{\epsilon,1}+A_{\epsilon,3} \leq& \|\Lambda^sJ_\epsilon^2 u\|_{L^2}\|[\Lambda^s,  (u\cdot \nabla)] u\|_{L^2}+\|\Lambda^sJ_\epsilon^2 \gamma\|_{L^2}\|[\Lambda^s,  (u\cdot \nabla)] \gamma\|_{L^2}\\
\leq& C\| u\|_{H^s}(\|\Lambda^s u\|_{L^2}\|\nabla u\|_{L^\infty}+\|\nabla u\|_{L^\infty}\|\Lambda^{s-1} \nabla u\|_{L^2} )\\
 & +C\|\gamma\|_{H^s}(\|\Lambda^s u\|_{L^2}\|\nabla \gamma\|_{L^\infty}+\|\nabla u\|_{L^\infty}\|\Lambda^{s-1} \nabla \gamma\|_{L^2} )  \\
\leq& C\|\nabla u\|_{L^\infty}\| u\|_{H^s}^2+C(\|\nabla u\|_{L^\infty}+\|\nabla \gamma\|_{L^\infty})\| u\|_{H^s} \| \gamma\|_{H^s}.
\end{split}
\end{equation}
Moreover, we get by integrating parts that
\begin{equation*}
\begin{split}
A_{\epsilon,2}+A_{\epsilon,4}  = &(\Lambda^sJ_\epsilon  u,[J_\epsilon,(u\cdot \nabla)]\Lambda^s u )_{L^2} + (\Lambda^sJ_\epsilon \gamma,[J_\epsilon,(u\cdot \nabla)]\Lambda^s \gamma )_{L^2}\\
 & -\int_{\mathbb{T}^d}\textrm{div}u (\Lambda^sJ_\epsilon  u)^2 \mathrm{d} x-\int_{\mathbb{T}^d}\textrm{div}u (\Lambda^sJ_\epsilon  \gamma)^2 \mathrm{d} x.
\end{split}
\end{equation*}
It   follows again from   Lemma \ref{lem:commutator} that
\begin{equation} \label{2.9}
\begin{split}
A_{\epsilon,2}+A_{\epsilon,4}   \leq& \|\Lambda^sJ_\epsilon  u\|_{L^2}\|[J_\epsilon,(u\cdot \nabla)]\Lambda^s u]\|_{L^2} + \|\Lambda^sJ_\epsilon \gamma\|_{L^2}\|[J_\epsilon,(u\cdot \nabla)]\Lambda^s \gamma\|_{L^2}\\
 & +\|\textrm{div}u\|_{L^\infty}\|\Lambda^sJ_\epsilon  u\|^2_{L^2}+\|\textrm{div}u\|_{L^\infty} \|\Lambda^sJ_\epsilon  \gamma\|^2_{L^2}\\
\leq& \|\nabla u\|_{L^\infty}(\| u\|_{H^s}^2 + \| \gamma\|_{H^s}^2).
\end{split}
\end{equation}
To deal with the first term involved in $A_{\epsilon,5} $, we need
\begin{lemma}[\cite{grafakos2014kato}]\label{lem:moser}
Let $p\in [1,\infty)$, $p,p_i,q_i\in (1,\infty]$, $i=1,2$ such that $
\frac{1}{p}=\frac{1}{p_1}+\frac{1}{q_1}=\frac{1}{p_2}+\frac{1}{q_2}$.
Then for any $s>0$, there exists a constant $c>0$ such that
$$
\|\Lambda^s(fg)\|_{L^p}\leq C(\|\Lambda^sf\|_{L^{p_1}}\|g\|_{L^{q_1}}+\| f\|_{L^{p_2}}\|\Lambda^sg\|_{L^{q_2}}).
$$
\end{lemma}
Due to the facts that $\Lambda^{-2}\textrm{div}$ is a $S^{-1}$-multiplier, $\Lambda^{-2}$ is a $S^{-2}$-multiplier and $H^{s-1}(\mathbb{T}^d)\subset H^{s-2}(\mathbb{T}^d)$, we get from Lemma \ref{lem:moser} that
\begin{equation*}
\begin{split}
&(\Lambda^sJ_\epsilon u,\Lambda^sJ_\epsilon \mathscr{L}_1(u) )_{L^2}\\
&\quad \leq  C\| u\|_{H^s}(\|\frac{1}{2}|\nabla u|^2\textbf{I}_d+ \nabla u \nabla u+\nabla u(\nabla u)^T-(\nabla u)^T  \nabla u-(\textrm{div}u) \nabla u\|_{H^{s-1}} \\
&\quad\quad+\left\|(\textrm{div}u) u +u\cdot ( \nabla u)^T\right\|_{H^{s-2}})\\
&\quad \leq   C\| u\|_{H^s}(\|\nabla u\|_{L^\infty} \|\nabla u\|_{H^{s-1}}+ \|\textrm{div}u \|_{L^\infty} \|\nabla u \|_{H^{s-1}} + \|\textrm{div}u \|_{H^{s-1}} \|\nabla u \|_{L^\infty}\\
&\quad\quad+\|\textrm{div}u\|_{L^\infty}\| u\|_{H^{s-1}} +\|\textrm{div}u\|_{H^{s-1}}\| u\|_{L^\infty} +\|u\|_{L^\infty}\| \nabla u \|_{H^{s-1}}) \\
&\quad \leq   C\| u\|_{W^{1,\infty}} \| u\|_{H^s} ^2.
\end{split}
\end{equation*}
For the second term involved in $A_{\epsilon,5} $, we have
\begin{equation*}
\begin{split}
(\Lambda^sJ_\epsilon u,\Lambda^sJ_\epsilon \mathscr{L}_2(\gamma) )_{L^2} &\leq   C\| u\|_{H^s} \|\frac{1}{2} (\gamma^2+ |\nabla\gamma|^2 )\textbf{I}_d-(\nabla\gamma)^T\nabla\gamma \|_{H^{s-1}}\\
&\leq   C\| u\|_{H^s}  (\|\gamma\|_{L^\infty}   \|\gamma\|_{H^{s-1}}+\|\nabla\gamma\|_{L^\infty}    \|\nabla\gamma\|_{H^{s-1}})\\
&\leq   C\| \gamma\|_{W^{1,\infty}}  \| u\|_{H^s}  \|\gamma\|_{H^{s}} .
\end{split}
\end{equation*}
It follows from the last two estimates that
\begin{equation}\label{2.10}
\begin{split}
A_{\epsilon,5} \leq  C\left(\| u\|_{W^{1,\infty}} \| u\|_{H^s} ^2+\| \gamma\|_{W^{1,\infty}}  \| u\|_{H^s}  \|\gamma\|_{H^{s}}\right).
\end{split}
\end{equation}
For $A_{\epsilon,6} $, one can estimate as
\begin{equation}\label{2.11}
\begin{split}
A_{\epsilon,6}
\leq& C\|\gamma\|_{H^{s}} (\|\nabla\gamma\|_{L^\infty} \|\nabla u\|_{H^{s-1}} +\|\nabla\gamma\| _{H^{s-1}}\|\nabla u\| _{L^\infty} +\| \textrm{div}u\|_{L^\infty} \|\nabla \gamma\|_{H^{s-1}}\\
 &+\| \textrm{div}u\|_{H^{s-1}} \|\nabla \gamma\| _{L^\infty}+\| \textrm{div}u\|_{H^{s-1}} \| \gamma\| _{L^\infty}+\| \textrm{div}u\| _{L^\infty}\| \gamma\|_{H^{s-1}} )\\
\leq& C(\| \gamma\|_{W^{1,\infty}}\|u\|_{H^{s}}\|\gamma\|_{H^{s}}+\|\nabla u\| _{L^\infty}\|\gamma\|_{H^{s}}^2).
\end{split}
\end{equation}

Putting the estimates \eqref{2.8}-\eqref{2.11} into \eqref{2.7},  using the H\"{o}lder inequality and the definition of $\mathbbm{t}_{2,n}$, the third term on the R.H.S. of \eqref{2.4} can be estimated as
\begin{equation} \label{2.12}
\begin{split}
 &\mathbb{E}\sup_{t\in[0,\mathbbm{t}_{2,n}\wedge k]}\left| \int_0^t(\Lambda^sJ_\epsilon\textbf{y},\Lambda^sJ_\epsilon(B(\textbf{y},\textbf{y}) +J_\epsilon F(\textbf{y})))_{\mathbb{L}^2}\mathrm{d} r \right|\\
 &\quad\leq C\mathbb{E}\int_0^{\mathbbm{t}_{2,n}\wedge k}(\|u\|_{W^{1,\infty}}+\| \gamma\|_{W^{1,\infty}})(\|u\|_{H^{s}}^2+\|\gamma\|_{H^{s}}^2)\mathrm{d} r \\
 &\quad\leq Cn\mathbb{E}\int_0^{\mathbbm{t}_{2,n}\wedge k} \|\textbf{y}(r)\|_{\mathbb{H}^{s}}^2 \mathrm{d} r.
\end{split}
\end{equation}
Thereby, we get from \eqref{2.4}-\eqref{2.6} and \eqref{2.12} that
\begin{equation} \label{2.13}
\begin{split}
 \mathbb{E}\left(1+\sup_{t\in [0,\mathbbm{t}_{2,n}\wedge k]}\|J_\epsilon\textbf{y}(t) \|_{\mathbb{H}^{s}}^2\right)\leq& 1+2\mathbb{E}\|\textbf{y}(0)\|_{\mathbb{H}^s}^2+C_{R}(\chi ^2( n)+n)\\
 &\times\int_0^{ k} (\mu^2  (r)+1)\mathbb{E}\left(1+\sup_{r'\in [0,\mathbbm{t}_{2,n}\wedge r]}\|\textbf{y} (r')\|_{\mathbb{H}^{s}}^2\right) \mathrm{d}r.
\end{split}
\end{equation}
Notice that the R.H.S. of \eqref{2.13} is independent of $\epsilon$, and the convergence $J_\epsilon \textbf{y}\rightarrow \textbf{y}$ holds strongly in $\mathcal {C}([0,T];\mathbb{H}^s(\mathbb{T}))$, $\mathbb{P}$-almost surely. After taking the limit as $\epsilon\rightarrow 0$ in \eqref{2.13} by applying the Dominated Convergence Theorem,  we get by applying the Gronwall inequality that
\begin{equation*}
\begin{split}
   \mathbb{E}\sup_{t\in [0,\mathbbm{t}_{2,n}\wedge k]}\| \textbf{y}(t) \|_{\mathbb{H}^{s}}^2\leq
   (1+2\mathbb{E}\|\textbf{y}(0)\|_{\mathbb{H}^s}^2)e^ {C_{R}(\chi ^2( n)+n)\int_0^{ k} (\mu^2  (r)+1)\mathrm{d}r },
\end{split}
\end{equation*}
which proves the claim in \eqref{2.2}.  From \eqref{2.2} and the monotonicity of stopping times $\mathbbm{t}_{1,m}$, we deduce that for all $n,k\in \mathbb{N}^+$
\begin{equation*}
\begin{split}
 \mathbb{P}(\mathbbm{t}_{2,n}\wedge k\leq \mathbbm{t}_{1})&\geq \mathbb{P}\left(\bigcup_{m\geq1}\{\mathbbm{t}_{2,n}\wedge k\leq \mathbbm{t}_{1,m}\}\right)\\
 & \geq \mathbb{P}\left(\bigcup_{m\geq1}\{ \sup_{t\in [0,\mathbbm{t}_{2,n}\wedge k]}\| \textbf{y}(t) \|_{\mathbb{H}^{s}} <m       \}\right)\\
 &= \mathbb{P}\left( \sup_{t\in [0,\mathbbm{t}_{2,n}\wedge k]}\| \textbf{y}(t) \|_{\mathbb{H}^{s}} <\infty       \right) =1,
\end{split}
\end{equation*}
which means that all the sets $\{\mathbbm{t}_{2,n}\wedge k\leq \mathbbm{t}_{1}\}_{n,k\geq 1}$ have full measure. Therefore, we deduce from the nondecreasing property of $\mathbbm{t}_{2,n}$ that
\begin{equation} \label{2.14}
\begin{split}
 \mathbb{P}(\mathbbm{t}_{2} \leq \mathbbm{t}_{1})&=  \mathbb{P}\left(\lim_{n\rightarrow\infty}\mathbbm{t}_{2,n} \leq \mathbbm{t}_{1}\right)=\mathbb{P}\left(\bigcap_{n\geq1} \{\mathbbm{t}_{2,n} \leq \mathbbm{t}_{1}\}\right) \\
 &=\mathbb{P}\left(\bigcap_{n\geq1}\bigcap_{k\geq1}\{\mathbbm{t}_{2,n}\wedge k\leq \mathbbm{t}_{1}\}\right)= 1.
\end{split}
\end{equation}
 By  \eqref{2.1} and  \eqref{2.14}, we get
$
\mathbb{P}(\mathbbm{t}_{2} = \mathbbm{t}_{1})=1.
$

{\textsf{Step 2:}} We verify that $\mathbbm{t}_1=\mathbbm{t}_2\triangleq\mathbbm{t}^*$ is actually the maximal existence time $\mathbbm{t}$ of solution $\textbf{y}$.  Otherwise, we assume that $\mathbbm{t}^*< \mathbbm{t}$ on $\{\mathbbm{t}<\infty\}$. Then by the uniqueness of solution, the pair $(\textbf{y},\mathbbm{t}^*)$ is a local strong pathwise solution. Note that for a given $n>0$, we may have $\mathbb{P}(\mathbbm{t}_{2,n}=0)\neq0$. However, for almost every $\omega\in \Omega$, there exists $n>0$ such that $\mathbbm{t}_{2,n}(\omega)>0$. In terms of the fact of $\mathbbm{t}_{2,n}\nearrow \mathbbm{t}_2$, we deduce from the Sobolev embedding $\mathbb{H}^s(\mathbb{T}^d)\subset \mathbb{W}^{1,\infty}(\mathbb{T}^d)$ and the  continuity of $\textbf{y}(t)$ on $\mathbb{W}^{1,\infty}$ that, for any $n\geq 1$,
\begin{equation*}
\begin{split}
n =\sup_{t\in [0,\mathbbm{t}_{2,n}]}\|\textbf{y}(t)\| _{\mathbb{W}^{1,\infty}} &\lesssim\sup_{t\in [0,\mathbbm{t}_{2,n}]}\|\textbf{y}(t)\| _{\mathbb{H}^s} \\
&\lesssim  \sup_{t\in [0,\mathbbm{t}^*]}\|\textbf{y}(t)\| _{\mathbb{H}^s} \leq C, \quad \textrm{on} ~~ \{\mathbbm{t}<\infty\},
\end{split}
\end{equation*}
for some positive constant $C$ independent of $n$. This is a contradiction, and we get from the uniqueness of solution that $\mathbbm{t}^*= \mathbbm{t}$ $\mathbb{P}$-almost surely. Moreover, by using the definitions of $\mathbbm{t}_1$ and $\mathbbm{t}_2$, we see that
$
 \textbf{1}_{\{\limsup \limits_{t\rightarrow\mathbbm{t}}\| \textbf{y}(t)\|_{\mathbb{H}^s}=\infty\}} =\textbf{1}_{\{\limsup\limits _{t\rightarrow\mathbbm{t}}\| \textbf{y}(t)\|_{\mathbb{W}^{1,\infty}}=\infty\}}$,  $\mathbb{P}$-a.s. The proof of Theorem \ref{th1}(2) is now completed.
\end{proof}

\subsection{Regularization of SMEP2}

For each $R>0$, we introduce the function
$$
\varpi_R (x) \triangleq \varpi \left(\frac{x}{R}\right), \quad \textrm{for all} ~ x\geq0
$$
with support in $[0,2R]$, where $\varpi:[0,\infty)\mapsto [0,1]$ is a smooth decreasing function given by
\begin{equation*}
\varpi (x) \triangleq
\begin{cases}
1,&0\leq x \leq 1,\cr
0,& x > 2,
\end{cases}\quad \textrm{and}\quad \sup_{x\in [0,\infty)}|\varpi'(x)| \leq C<\infty.
\end{equation*}
For example, the function $\varpi (\cdot)$ can be obtained by mollifying the   function $f(x)=1$ when $0\leq x \leq 1$; $f(x)=0$ when $x>1$. Then the truncated SMEP2 can be formulated by
\begin{equation}
\left\{
\begin{aligned}\label{2.15}
&\mathrm{d}\textbf{y}+ \varpi_R  (\|\textbf{y}\|_{\mathbb{W}^{1,\infty}}) B(\textbf{y},\textbf{y})\mathrm{d} t\\
&\quad =\varpi_R  (\|\textbf{y}\|_{\mathbb{W}^{1,\infty}})F(\textbf{y})\mathrm{d} t+\varpi _R (\|\textbf{y}\|_{\mathbb{W}^{1,\infty}})G(t,\textbf{y}) \mathrm{d}\mathcal {W},\quad t>0,~x\in \mathbb{T}^d,\\
&\textbf{y}(\omega,0,x)=\textbf{y}_0(\omega,x) ,\quad x\in \mathbb{T}^d.
\end{aligned}
\right.
\end{equation}

Unlike the fluid models such as the Euler equations and Navier-Stokes equations, the SMEP2 do not possess the cancelation property, that is,
 $
(B(\textbf{y},\textbf{y}),\textbf{y})_{L^2}=0.
 $
Hence the existence and uniqueness of global approximate solutions to \eqref{2.15} can not be guaranteed by classical Garlekin approximation method.

We overcome the difficulty by regularizing the convection term $B(\textbf{y},\textbf{y})$ and considering
\begin{equation}
\left\{
\begin{aligned}\label{2.16}
&\mathrm{d}\textbf{y}=\mathcal {F}_{R,\epsilon}(\textbf{y})\mathrm{d} t+\mathcal {G}_R(t,\textbf{y})\mathrm{d}\mathcal {W},\\
& \mathcal {F}_{R,\epsilon}(\textbf{y}) = -\varpi_R  (\|\textbf{y}\|_{\mathbb{W}^{1,\infty}}) J_\epsilon B(J_\epsilon\textbf{y},J_\epsilon\textbf{y}) + \varpi_R  (\|\textbf{y}\|_{\mathbb{W}^{1,\infty}})F(\textbf{y}),\\
& \mathcal {G}_R(t,\textbf{y}) = \varpi_R  (\|\textbf{y}\|_{\mathbb{W}^{1,\infty}})G(t,\textbf{y}) ,\\
& J_\epsilon B(J_\epsilon\textbf{y},J_\epsilon\textbf{y})= (J_\epsilon[
   ( J_\epsilon u\cdot \nabla) J_\epsilon u],J_\epsilon[(
     J_\epsilon u\cdot \nabla) J_\epsilon\gamma])^T,\\
& \textbf{y}(\omega,0)=\textbf{y}_0(\omega),
\end{aligned}
\right.
\end{equation}
where $J_\epsilon$ is the classical Fredriches mollifier. By using  Assumption \ref{assume}, one can verify that the coefficients in \eqref{2.16}$_1$ are actually locally bounded:
\begin{equation*}
\begin{split}
\|\mathcal {G}_R(t,\textbf{y})\|_{\mathcal {L}_2(\mathfrak{U}_1,\mathbb{H}^{s})} &\leq \chi(R)\mu (t)(1+\|\textbf{y}\|_{\mathbb{H}^{s}}),\\
\|\mathcal {F}_{R,\epsilon}(\textbf{y})\|_{\mathcal {L}_2(\mathfrak{U}_1,\mathbb{H}^{s})} &\leq  (\frac{C}{\epsilon}+1)R\|\textbf{y}\|_{\mathbb{H}^{s}},
\end{split}
\end{equation*}
for any $t\leq T$ and some constant $C>0$, which indicates that \eqref{2.16} can be regarded as a system of SDEs in Hilbert spaces $\mathbb{H}^s(\mathbb{T}^d)$. Moreover, by Assumption \ref{assume}, it is also not difficult to verify that
$$
\mathcal {F}_{R,\epsilon}(\textbf{y}):\mathbb{H}^{s}(\mathbb{T}^d)\mapsto \mathbb{H}^{s}(\mathbb{T}^d),\quad \mathcal {G}_R(t,\textbf{y}): \mathbb{H}^{s}(\mathbb{T}^d)\mapsto L_2(\mathfrak{U}_1,\mathbb{H}^{s}(\mathbb{T}^d))
$$
both are locally Lipchitz continuous functionals. Therefore, by classcial theory for SDEs in Hilbert spaces, there exits a time $T_{R,\epsilon}>0$ such that the SDEs \eqref{2.16} admits a local strong solution $\textbf{y}_{R,\epsilon}\in \mathcal {C}([0,T_{R,\epsilon});\mathbb{H}^{s}(\mathbb{T}^d))$ $\mathbb{P}$-almost surely. In a similar manner as we did in the proof of Theorem \ref{th1}(2), one can show that if $T_{R,\epsilon}<\infty$, then $
\limsup _{t\rightarrow T_{R,\epsilon}}\| \textbf{y}_{R,\epsilon}(t)\|_{\mathbb{W}^{1,\infty}}=\infty$, $\mathbb{P}$-a.s. However,  due to the appearance of the cut-off function $\varpi_R  (\|\textbf{y}_{R,\epsilon}\|_{\mathbb{W}^{1,\infty}})$, the solutions $\textbf{y}_{R,\epsilon}$ exists globally.

We summarize the above discussion into the following lemma:
\begin{lemma} \label{lem:2.1}
Let $s> 4+\frac{d}{2}$, $d\geq 1$, $R\geq1$  and $\epsilon>0$. Suppose that $\textbf{y}_0$ is a $\mathcal {F}_0$-measurable random variable in $L^2(\Omega;\mathbb{H}^s(\mathbb{T}^d))$, and the conditions (A1)-(A2) hold. Then for any given $T>0$, the SDEs \eqref{2.16} admits a unique strong solution $\textbf{y}_{R,\epsilon} \in \mathcal {C}([0,T];\mathbb{H}^s(\mathbb{T}^d))$, $\mathbb{P}$-almost surely.
\end{lemma}

\subsection{Momentum estimates}
In order to taking the limit $R\rightarrow \infty$ and $\epsilon\rightarrow 0$ in suitable sense for  approximation solutions $\{\textbf{y}_{R,\epsilon}\}_{R\geq1,0<\epsilon <1}$, we shall first establish some a priori uniform estimates for approximate solutions for any given $R\geq1$.

\begin{lemma} \label{lem:2.2}
Let $s> 4+\frac{d}{2} $, $d\geq1$ and $R>1$. Suppose that the conditions in Assumption \ref{assume} hold, and $\textbf{y}_0\in L^r(\Omega; \mathbb{H}^s(\mathbb{T}^d))$ is a $\mathcal {F}_0$-measurable random variable. For any $T>0$, let $ \textbf{y}_{R,\epsilon}$ be the unique strong solution to \eqref{2.16}. Then for any $p\geq 3$, $\alpha\in (0,\frac{1}{2}) $ and $\beta\in (0,\frac{1}{2}-\frac{1}{p})$, we have
\begin{equation}\label{2.18}
\begin{split}
\textbf{y}_{R,\epsilon}\in & L^p(\Omega;\mathcal {C}([0,T];\mathbb{H}^s(\mathbb{T}^d))) \bigcap L^p(\Omega;\mathbb{W}^{\alpha,p}(0,T;\mathbb{H}^{s-1}(\mathbb{T}^d) )) \\
& \bigcap L^p(\Omega;\mathcal {C}^\beta([0,T];\mathbb{H}^{s-1}(\mathbb{T}^d) )),\quad \forall \epsilon \in (0,1).
\end{split}
\end{equation}
Moreover, there exists some positive constant $C$ independent of $\epsilon$ such that
\begin{equation}\label{2.19}
\begin{split}
\sup_{0<\epsilon<1}\mathbb{E}\left\|\int_0^\cdot\mathcal {G}_R(r,\textbf{y}_{R,\epsilon})\mathrm{d}\mathcal {W}\right\|_{\mathbb{W}^{\alpha,p}(0,T;\mathbb{H}^{s-1} )} ^p &\leq C,\\
\sup_{0<\epsilon<1}\mathbb{E}\left\|\textbf{y}_{R,\epsilon}-\int_0^\cdot\mathcal {G}_R(r,\textbf{y}_{R,\epsilon})\mathrm{d}\mathcal {W}\right\|_{\mathbb{W}^{1,p}(0,T;\mathbb{H}^{s-1} )} ^p &\leq C.
\end{split}
\end{equation}
\end{lemma}

\begin{proof}[\emph{\textbf{Proof.}}] The proof consists of two steps.

{\textsf{Step 1}}: By applying the Bessel potential $\Lambda^s$ to  \eqref{2.16}$_1$ and then the It\^{o} formula to $\|\textbf{y} _{R,\epsilon}\|_{\mathbb{H}^s}^2=(\Lambda^s\textbf{y}_{R,\epsilon} ,\Lambda^s\textbf{y} _{R,\epsilon}) _{\mathbb{L}^2}$, we get
\begin{equation}\label{2.20}
\begin{split}
\| \textbf{y} (t)\|_{\mathbb{H}^s}^2=& \| \textbf{y} (0)\|_{\mathbb{H}^s}^2-2\int_0^t \int_{\mathbb{T}^d}\varpi_R  (\|\textbf{y}\|_{\mathbb{W}^{1,\infty}}) \Lambda^s \textbf{y} \cdot\Lambda^s    J_\epsilon B(J_\epsilon\textbf{y},J_\epsilon\textbf{y}) \mathrm{d}x\mathrm{d}r\\
&+2\int_0^t \int_{\mathbb{T}^d}\varpi_R  (\|\textbf{y}\|_{\mathbb{W}^{1,\infty}})\Lambda^s \textbf{y} \cdot  F(\textbf{y}) \mathrm{d}x\mathrm{d}r \\
&+\int_0^t\varpi_R^2  (\|\textbf{y}\|_{\mathbb{W}^{1,\infty}}) \| G (r,\textbf{y} )\|_{L_2(\mathfrak{U}_1,\mathbb{H}^{s})}^2\mathrm{d}r \\
   &+ 2\int_0^t\int_{\mathbb{T}^d} \varpi_R  (\|\textbf{y}\|_{\mathbb{W}^{1,\infty}}) \Lambda^s\textbf{y}\cdot\Lambda^sG (r,\textbf{y} )\mathrm{d}x \mathrm{d} \mathcal {W} \\
   =&\| \textbf{y} (0)\|_{\mathbb{H}^s}^2+I_1(t)+I_2(t)+I_3(t)+I_4(t).
\end{split}
\end{equation}
Here and in the proof of Lemma \ref{lem:2.2}, we shall omit the subscripts $R$ and $\epsilon$ of $\textbf{y} _{R,\epsilon}$ for simplicity. For $I_1(t)$, by commutating the operator $\Lambda^s$ with $J_\epsilon$ (cf. \cite[idendity (2.3)]{tang2020noise}) and then integrating by parts, we obtain
\begin{equation}\label{2.21}
\begin{split}
 &\int_{\mathbb{T}^d} \Lambda^s \textbf{y} \cdot\Lambda^s    J_\epsilon B(J_\epsilon\textbf{y},J_\epsilon\textbf{y}) \mathrm{d}x\\
 &\quad = \int_{\mathbb{T}^d}\Lambda^s J_\epsilon u \cdot ([\Lambda^s, J_\epsilon u ] \cdot \nabla J_\epsilon u ) \mathrm{d}x-\frac{1}{2}\int_{\mathbb{T}^d}|\Lambda^s J_\epsilon u|^2\textrm{ div} (J_\epsilon u)   \mathrm{d}x\\
     &\quad\quad +\int_{\mathbb{T}^d}\Lambda^s J_\epsilon \gamma \cdot ([\Lambda^s, J_\epsilon u ] \cdot \nabla J_\epsilon \gamma ) \mathrm{d}x-\frac{1}{2}\int_{\mathbb{T}^d}|\Lambda^s J_\epsilon \gamma|^2\textrm{ div} (J_\epsilon u)\mathrm{d}x.
\end{split}
\end{equation}
By using Lemma \ref{lem:commutator} and the fact that $\|J_\epsilon u \|_{L^\infty}\leq \| u \|_{L^\infty}$ for any $u\in L^\infty(\mathbb{T}^d) $, we have
\begin{equation*}
\begin{split}
&\int_{\mathbb{T}^d}\Lambda^s J_\epsilon u \cdot ([\Lambda^s, J_\epsilon u ] \cdot \nabla J_\epsilon u ) \mathrm{d}x\\
 &\quad\leq C\| u \|_{H^s}(\|\Lambda^s  J_\epsilon u\|_{L^2}\|\nabla J_\epsilon u\|_{L^\infty}+ \|\nabla  J_\epsilon u\|_{L^\infty}\|\Lambda^{s-1} \nabla J_\epsilon u \|_{L^2}) \\
 &\quad\leq C\|\nabla u\|_{L^\infty}\| u \|_{H^s}^2,\\
&\int_{\mathbb{T}^d}|\Lambda^s J_\epsilon u|^2\textrm{ div} (J_\epsilon u)   \mathrm{d}x  \leq \|\textrm{div} (J_\epsilon u) \|_{L^\infty}\|\Lambda^{s} J_\epsilon u \|_{L^2}^2 \\
&\quad\leq C\|\nabla u\|_{L^\infty}\| u \|_{H^s}^2 .
\end{split}
\end{equation*}
Similarly, the third and the forth terms on the R.H.S. of \eqref{2.21} can be estimated as
\begin{equation*}
\begin{split}
&\int_{\mathbb{T}^d}\Lambda^s J_\epsilon \gamma \cdot ([\Lambda^s, J_\epsilon u ] \cdot \nabla J_\epsilon \gamma ) \mathrm{d}x-\frac{1}{2}\int_{\mathbb{T}^d}|\Lambda^s J_\epsilon \gamma|^2\textrm{ div} (J_\epsilon u)\mathrm{d}x \\
 & \quad \leq C(\|\nabla \gamma\|_{L^\infty}+\|\nabla u\|_{L^\infty})(\| u \|_{H^s}^2+\| \gamma \|_{H^s}^2 ).
\end{split}
\end{equation*}
Plugging the last three estimates into \eqref{2.21} yields that
\begin{equation}\label{2.22}
\begin{split}
\mathbb{E}\sup_{r\in [0,t]}|I_1(r)|&\leq C\mathbb{E}\int_0^t \varpi_R  (\|\textbf{y}\|_{\mathbb{W}^{1,\infty}})\left|\int_{\mathbb{T}^d}\Lambda^s \textbf{y} \cdot\Lambda^s    J_\epsilon B(J_\epsilon\textbf{y},J_\epsilon\textbf{y}) \mathrm{d}x\right| \mathrm{d}r \\
&\leq C\mathbb{E}\int_0^t \varpi_R  (\|\textbf{y}\|_{\mathbb{W}^{1,\infty}}) (\|\nabla \gamma\|_{L^\infty}+\|\nabla u\|_{L^\infty})  (\| u \|_{H^s}^2+\| \gamma \|_{H^s}^2 )  \mathrm{d}r\\
&\leq CR\int_0^t\mathbb{E} \|\textbf{y} (r)\|_{\mathbb{H}^{s}}^2 \mathrm{d}r.
\end{split}
\end{equation}
For $I_2(t)$, by using Lemma \ref{lem:moser} and the property $\|\Lambda^2f\|_{H^{s-2}}\approx\|f\|_{H^{s}}$ for any $f\in \mathscr{S}(\mathbb{T}^d)$, one can estimate $\mathscr{L}_1 (u)$, $\mathscr{L}_2 (\gamma)$ and $\mathscr{L}_3 (u,\gamma) $ as follows:
\begin{equation*}
\begin{split}
\|\mathscr{L}_1 (u) \|_{H^s} \leq& C(\| u\|_{L^\infty}+\|\nabla u\|_{L^\infty})\| u \|_{H^s} ,\\
\|\mathscr{L}_2 (\gamma) \|_{H^s}\leq&  C(\| \gamma\|_{L^\infty}+\|\nabla \gamma\|_{L^\infty})\| \gamma \|_{H^s},\\
\|\mathscr{L}_3 (u,\gamma) \|_{H^s} \leq& C (\| \gamma\|_{L^\infty}+\|\nabla \gamma\|_{L^\infty}+\|\nabla u\|_{L^\infty}) (  \|  u  \|_{H^{s }} +\|  \gamma\|_{H^{s }} ),
\end{split}
\end{equation*}
which lead to
\begin{equation}\label{2.23}
\begin{split}
\mathbb{E}\sup_{r\in [0,t]}|I_2(r)|\leq& C\mathbb{E}\int_0^t\varpi_R  (\|\textbf{y}\|_{\mathbb{W}^{1,\infty}})  \|\textbf{y} \|_{\mathbb{H}^{s}}\|F(\textbf{y}) \|_{\mathbb{H}^{s}} \mathrm{d}r \\
\leq&   C\mathbb{E}\int_0^t\varpi_R  (\|\textbf{y}\|_{\mathbb{W}^{1,\infty}})  \|\textbf{y}\|_{\mathbb{W}^{1,\infty}} \|\textbf{y} \|_{\mathbb{H}^{s}}^2 \mathrm{d}r
\leq  CR\int_0^t\mathbb{E} \|\textbf{y} (r)\|_{\mathbb{H}^{s}}^2 \mathrm{d}r.
\end{split}
\end{equation}
By assumption (1), one can estimate $I_3(t)$ as
\begin{equation}\label{2.24}
\begin{split}
\mathbb{E}\sup_{r\in [0,t]}|I_3(r)| &\leq C\mathbb{E}  \int_0^t\varpi_R^2  (\|\textbf{y}\|_{\mathbb{W}^{1,\infty}}) \mu^2 (t) \chi^2(\|\textbf{y}\|_{\mathbb{W}^{1,\infty}})(1+\|\textbf{y}\|_{\mathbb{H}^{s}}^2)\mathrm{d}r\\
&\leq    C \chi ^2(2R)\int_0^t\mu ^2 (r)(1+\mathbb{E}\|\textbf{y}(r)\|_{\mathbb{H}^{s}} ^2) \mathrm{d}r.
\end{split}
\end{equation}
For $I_4(t)$, one can use the BDG inequality and \eqref{2.24} to obtain
\begin{equation}\label{2.25}
\begin{split}
\mathbb{E}\sup_{r\in [0,t]}|I_4(r)|&\leq   C\mathbb{E}\left( \int_0^t\sum_{k\geq1}\left(\int_{\mathbb{T}^d} \varpi_R  (\|\textbf{y}\|_{\mathbb{W}^{1,\infty}}) \Lambda^s\textbf{y}\cdot\Lambda^sG _k(r,\textbf{y} )\mathrm{d}x\right)^2 \mathrm{d}r\right)^{\frac{1}{2}}\\
 &\leq   C\mathbb{E}\left( \int_0^t\mu ^2 (r) \varpi_R ^2 (\|\textbf{y}\|_{\mathbb{W}^{1,\infty}})\chi ^2(\|\textbf{y}\|_{\mathbb{W}^{1,\infty}})\| \textbf{y}\|_{\mathbb{H}^{s}}^2(1+\|\textbf{y}\|_{\mathbb{H}^{s}})  ^2 \mathrm{d}r\right)^{\frac{1}{2}}\\
 &\leq   \frac{1}{2} \mathbb{E} \sup_{r\in[0,t]}\| \textbf{y}(r)\|_{\mathbb{H}^{s}}^2+ C \chi ^2(2R)\int_0^t\mu ^2 (r)(1+\mathbb{E}\|\textbf{y}(r)\|_{\mathbb{H}^{s}} ^2) \mathrm{d}r.
\end{split}
\end{equation}
By taking the supremum on both sides of \eqref{2.20}, we deduce from the estimates \eqref{2.22}-\eqref{2.25} that
\begin{equation*}
\begin{split}
 \mathbb{E}\sup_{r\in [0,t]}\| \textbf{y} (r)\|_{\mathbb{H}^s}^2 \leq & 1+2\mathbb{E}\| \textbf{y} (0)\|_{\mathbb{H}^s}^2+ C (R+\chi ^2(2R))\\
&\times\mathbb{E}\int_0^t(1+\mu ^2 (r))\left(1+\mathbb{E}\sup_{\varsigma\in [0,r]}\|\textbf{y}(\varsigma)\|_{\mathbb{H}^{s}} ^2\right) \mathrm{d}r,
\end{split}
\end{equation*}
Thanks to the Gronwall inequality, we get
\begin{equation*}
\begin{split}
 \mathbb{E}\sup_{r\in [0,T]}\| \textbf{y} (r)\|_{\mathbb{H}^s}^2 \leq Ce^{(R+\chi ^2(2R))\int_0^T(1+\mu^2(r)) \mathrm{d}r} (1+\mathbb{E}\| \textbf{y} (0)\|_{\mathbb{H}^s}^2),
\end{split}
\end{equation*}
for any $T>0$, which combined with the continuity of $\mu(\cdot)$ yield  that the approximations are uniformly bounded in $L^2(\Omega;\mathcal {C}([0,T];\mathbb{H}^s(\mathbb{T}^d)))$.

Now we apply the It\^{o} formula to $\| \textbf{y} (r)\|_{\mathbb{H}^s}^p=(\| \textbf{y} (r)\|_{\mathbb{H}^s}^2)^{\frac{p}{2}}$ with $p>2$, and then use the identity \eqref{2.20}, one find
\begin{equation}\label{2.26}
\begin{split}
\| \textbf{y} (t)\|_{\mathbb{H}^s}^p &=\| \textbf{y} (0)\|_{\mathbb{H}^s}^p -p\int_0^t\varpi_R  (\|\textbf{y}\|_{\mathbb{W}^{1,\infty}})\| \textbf{y} (r)\|_{\mathbb{H}^s}^{p-2}(\textbf{y} , J_\epsilon B(J_\epsilon\textbf{y},J_\epsilon\textbf{y})- F(\textbf{y}))_{\mathbb{H}^s}\mathrm{d} r\\
&\quad +\frac{p}{2}\int_0^t\varpi_R  (\|\textbf{y}\|_{\mathbb{W}^{1,\infty}})\| \textbf{y} (r)\|_{\mathbb{H}^s}^{p-2}  \|G(r,\textbf{y})\|_{\mathcal {L}_2(\mathfrak{U}_1,\mathbb{H}^{s})}^2\mathrm{d} r   \\
&\quad+\frac{p(p-2)}{2} \sum_{k\geq 1}\int_0^t\varpi_R  (\|\textbf{y}\|_{\mathbb{W}^{1,\infty}})\| \textbf{y} (r)\|_{\mathbb{H}^s}^{p-4}(\textbf{y} , G_{k}( \textbf{y}))_{\mathbb{H}^s}^2\mathrm{d}r\\
&\quad+p\sum_{k\geq 1}\int_0^t\varpi_R  (\|\textbf{y}\|_{\mathbb{W}^{1,\infty}})\| \textbf{y} (r)\|_{\mathbb{H}^s}^{p-2}(\textbf{y} , G_{k}( \textbf{y}))_{\mathbb{H}^s}\mathrm{d}\beta_k\\
&= \| \textbf{y} (0)\|_{\mathbb{H}^s}^p + H_1(t)+ H_2(t)+ H_3(t)+ H_4(t).
\end{split}
\end{equation}
The term $H_1(t)$ can be treated as
\begin{equation}\label{2.27}
\begin{split}
\mathbb{E}\sup_{r\in [0,t]}|H_1(r)| &\leq C\mathbb{E}\int_0^t\varpi_R  (\|\textbf{y}\|_{\mathbb{W}^{1,\infty}})\| \textbf{y} (r)\|_{\mathbb{H}^s}^{p-1}\| J_\epsilon B(J_\epsilon\textbf{y},J_\epsilon\textbf{y})- F(\textbf{y})\|_{\mathbb{H}^s}\mathrm{d} r\\
&\leq C\mathbb{E}\int_0^t\varpi_R  (\|\textbf{y}\|_{\mathbb{W}^{1,\infty}})\| \textbf{y} (r)\|_{\mathbb{H}^s}^{p-1}\|\textbf{y}\|_{\mathbb{W}^{1,\infty}} \|\textbf{y} \|_{\mathbb{H}^{s}}\mathrm{d} r \\
&\leq CR\int_0^t \mathbb{E}\| \textbf{y} (r)\|_{\mathbb{H}^s}^{p} \mathrm{d} r.
\end{split}
\end{equation}
For $H_2(t)$, it follows from the Assumption \ref{assume} and Young inequality that
\begin{equation}\label{2.28}
\begin{split}
\mathbb{E}\sup_{r\in [0,t]}|H_2(r)| &\leq C\mathbb{E} \int_0^t\varpi_R  (\|\textbf{y}\|_{\mathbb{W}^{1,\infty}})\| \textbf{y} (r)\|_{\mathbb{H}^s}^{p-2}  \|G(r,\textbf{y})\|_{\mathcal {L}_2(\mathfrak{U}_1,\mathbb{H}^{s})}^2\mathrm{d} r\\
&\leq C\mathbb{E} \int_0^t\varpi_R  (\|\textbf{y}\|_{\mathbb{W}^{1,\infty}})  \mu^2 (t) \chi^2(\|\textbf{y}\|_{\mathbb{W}^{1,\infty}})(1+\|\textbf{y}\|_{\mathbb{H}^{s}}^2)\| \textbf{y} (r)\|_{\mathbb{H}^s}^{p-2}\mathrm{d} r\\
&\leq C\chi^2(2R) \int_0^t  \mu^2 (r)(1+\mathbb{E}\|\textbf{y}(r)\|_{\mathbb{H}^{s}}^p)\mathrm{d} r.
\end{split}
\end{equation}
In a similar manner,
\begin{equation}\label{2.29}
\begin{split}
\mathbb{E}\sup_{r\in [0,t]}|H_3(r)|   \leq C\chi^2(2R) \int_0^t  \mu^2 (r)(1+\mathbb{E}\|\textbf{y}(r)\|_{\mathbb{H}^{s}}^p)\mathrm{d} r.
\end{split}
\end{equation}
For the stochastic integral term $H_4(t)$, we get from \eqref{a1} and the BDG inequality that
\begin{equation}\label{2.30}
\begin{split}
\mathbb{E}\sup_{r\in [0,t]}|H_4(r)|   &\leq C\mathbb{E}\left(\int_0^t\mu^2 (t) \varpi_R  ^2(\|\textbf{y}\|_{\mathbb{W}^{1,\infty}})\chi^2(\|\textbf{y}\|_{\mathbb{W}^{1,\infty}})\| \textbf{y} (r)\|_{\mathbb{H}^s}^{2p-2} (1+\|\textbf{y}\|_{\mathbb{H}^{s}}^2)\mathrm{d} t\right)^{\frac{1}{2}}\\
&\leq C\chi(2R)\mathbb{E}\left[\sup_{r\in [0,t]}\| \textbf{y} (r)\|_{\mathbb{H}^s}^{\frac{p}{2}}\left( \int_0^t\mu^2 (t) \| \textbf{y} (r)\|_{\mathbb{H}^s}^{p-2} (1+\|\textbf{y}\|_{\mathbb{H}^{s}}^2)\mathrm{d} t\right)^{\frac{1}{2}}\right]\\
&\leq  \frac{1}{2} \mathbb{E}\sup_{r\in [0,t]} \| \textbf{y} (r)\|_{\mathbb{H}^s}^{p}+C\chi^2(2R)\int_0^t\mu^2 (t)  (1+\mathbb{E}\|\textbf{y}(r)\|_{\mathbb{H}^{s}}^p)\mathrm{d} r.
\end{split}
\end{equation}
Therefore, after taking supremum to \eqref{2.26} over the interval $[0, t]$,
we deduce from the estimates \eqref{2.27}-\eqref{2.30} that
\begin{equation*}
\begin{split}
\mathbb{E}\sup_{r\in [0,t]}\| \textbf{y} (t)\|_{\mathbb{H}^s}^p\leq &2\| \textbf{y} (0)\|_{\mathbb{H}^s}^p +C(R+\chi^2(2R))\int_0^t(1+\mu^2 (t))  (1+\mathbb{E}\|\textbf{y}(r)\|_{\mathbb{H}^{s}}^p)\mathrm{d} r\\
\leq& 2\| \textbf{y} (0)\|_{\mathbb{H}^s}^p +C(R+\chi^2(2R))\int_0^t(1+\mu^2 (r))  \mathrm{d} r\\
&+C(R+\chi^2(2R))\int_0^t(1+\mu^2 (r))\mathbb{E}\sup_{\varsigma\in [0,r]}\| \textbf{y} (\varsigma)\|_{\mathbb{H}^s}^p \mathrm{d} r.
\end{split}
\end{equation*}
An application of the  Gronwall inequality to above inequality yields that
\begin{equation*}
\begin{split}
\mathbb{E}\sup_{r\in [0,T]}\| \textbf{y} (r)\|_{\mathbb{H}^s}^p\leq& e^{C(R+\chi^2(2R))\int^T_0 (1+\mu^2 (r))\mathrm{d} r}\bigg(2\| \textbf{y} (0)\|_{\mathbb{H}^s}^p\\
&+ C\left(R+\chi^2(2R)\right)\int_0^T (1+\mu^2 (\varsigma)) e^{-C(R+\chi^2(2R))\int_0^\varsigma (1+\mu^2 (r))\mathrm{d} r}\mathrm{d}\varsigma\bigg),
\end{split}
\end{equation*}
for any $T>0$. As the function $\mu^2(t)$ is continuous and hence integrable on any finite interval $[0,T]$, there is a constant $C>0$ independent of $\epsilon$ such that
\begin{equation*}
\begin{split}
\mathbb{E}\sup_{r\in [0,T]}\| \textbf{y}_{R,\epsilon} (r)\|_{\mathbb{H}^s}^p\leq C,\quad \forall \epsilon\in (0,1),
\end{split}
\end{equation*}
which implies that $\{\textbf{y}_{R,\epsilon}\}_{0<\epsilon<1}$  is uniformly bounded in $L^p(\Omega;\mathcal {C}([0,T];\mathbb{H}^s(\mathbb{T}^d)))$.

{\textsf{Step 2 (H\"{o}lder regularity)}}: Since we do not
expect $\textbf{y}_{R,\epsilon}$ to be differentiable in time in the stochastic setting, we are inspired to consider the estimates on fractional time derivatives of order strictly less than $ \frac{1}{2}$. Notice that for  any $\alpha\in (0,1)$, we have
\begin{equation}\label{2.31}
\begin{split}
 \mathbb{E}\|\textbf{y}\|_{\mathbb{W}^{\alpha,p}(0,T;\mathbb{H}^{s-1} )} ^p\leq& C\left(\mathbb{E}\|\textbf{y}(0)\|_{\mathbb{H}^s}+\mathbb{E}\left\|\int_0^\cdot\mathcal {F}_{R,\epsilon}(\textbf{y})\mathrm{d} r\right\|_{\mathbb{W}^{1,p}(0,T;\mathbb{H}^{s-1} )} ^p\right.\\
 &+\left.\mathbb{E}\left\|\int_0^\cdot\mathcal {G}_R(r,\textbf{y})\mathrm{d}\mathcal {W}\right\|_{\mathbb{W}^{\alpha,p}(0,T;\mathbb{H}^{s-1} )} ^p\right),
\end{split}
\end{equation}
where we used the Sobolev embedding $W^{1,p}(0,T;\mathbb{H}^{s-1}(\mathbb{T}^d) )\hookrightarrow W^{\alpha,p}(0,T;\mathbb{H}^{s-1} (\mathbb{T}^d))$ for all $0<\alpha<1$. Let us estimate the second and third terms on the R.H.S. of \eqref{2.31}. First, we get by using the Minkowski inequality that
\begin{equation*}
\begin{split}
 \mathbb{E}\left\|\int_0^\cdot\mathcal {F}_{R,\epsilon}(\textbf{y})\mathrm{d} r\right\|_{\mathbb{W}^{1,p}(0,T;\mathbb{H}^{s-1} )} ^p &= \mathbb{E} \int_0^T\|\mathcal {F}_{R,\epsilon}(\textbf{y})\|_{\mathbb{H}^{s-1}}^p\mathrm{d} t +\mathbb{E} \int_0^T\left\|\int_0^t\mathcal {F}_{R,\epsilon}(\textbf{y})\mathrm{d} r\right\|_{\mathbb{H}^{s-1}}^p\mathrm{d} t\\
 &\leq C \mathbb{E} \int_0^T\|\mathcal {F}_{R,\epsilon}(\textbf{y})\|_{\mathbb{H}^{s-1}}^p\mathrm{d} t.
\end{split}
\end{equation*}
Due to the boundedness of the mollifier $J_\epsilon$ (cf. \eqref{1.7}) and the Moser-type estimate (cf. Lemma \ref{lem:moser}), we have
\begin{equation*}
\begin{split}
 &\| J_\epsilon[
   ( J_\epsilon u\cdot \nabla) J_\epsilon u]\|_{H^{s-1}}+\|J_\epsilon[(
     J_\epsilon u\cdot \nabla) J_\epsilon\gamma] \|_{H^{s-1}} \\
     &\quad \leq C\Big(\|J_\epsilon u\|_{L^\infty} \|\nabla J_\epsilon u\|_{H^{s-1}} + \|J_\epsilon u\| _{H^{s-1}}\|\nabla J_\epsilon u\| _{L^\infty} \\
     &\quad\quad \quad +\|J_\epsilon u\|_{L^\infty} \|\nabla J_\epsilon \gamma\|_{H^{s-1}}+ \|\nabla J_\epsilon \gamma\| _{L^\infty}\| J_\epsilon u\| _{H^{s-1}}\Big)\\
     &\quad \leq C (\|  u\|_{L^\infty} \|u\|_{H^{s}} + \|\nabla u\| _{L^\infty}\|u\| _{H^{s}}+\|u\|_{L^\infty} \|\gamma\|_{H^{s}}+ \|\nabla \gamma\| _{L^\infty}\| u\| _{H^{s}} )\\
     &\quad \leq C (\|  u\|_{W^{1,\infty}} + \|\nabla \gamma\| _{L^\infty} )(\|u\|_{H^{s}}+ \|\gamma\|_{H^{s}}),
\end{split}
\end{equation*}
and
\begin{equation*}
\begin{split}
  \|F(\textbf{y})\|_{\mathbb{H}^{s-1}}  \leq& C \Big( \|\nabla u\| _{L^\infty} \|\nabla u\| _{H^{s-1}} + \| u\| _{L^\infty} \|\nabla u\| _{H^{s-1}} + \| u\| _{H^s} \|\nabla u\| _{L^\infty} \\
 &  +\|\nabla \gamma\| _{L^\infty} \|\nabla \gamma\| _{H^{s-1}}+\|\gamma\| _{L^\infty} \|\gamma\| _{H^s}+  \|\nabla \gamma\| _{L^\infty} \|\nabla u\| _{H^{s-1}}\\
 &+  \|\nabla u\| _{L^\infty} \|\nabla \gamma\| _{H^{s-1}} + \|\nabla u\| _{L^\infty} \| \gamma\| _{H^s}+ \|\gamma\| _{L^\infty} \|\nabla u\| _{H^{s-1}}\Big)\\
 \leq& C (\|  u\|_{W^{1,\infty}} + \|\gamma\| _{W^{1,\infty}} )(\|u\|_{H^{s}}+ \|\gamma\|_{H^{s}}).
\end{split}
\end{equation*}
From the last two estimates, the definition of $\mathcal {F}_{R,\epsilon}(\textbf{y})$ and the uniform bound obtained in \textsf{Step 1}, we deduce that
\begin{equation}\label{2.32}
\begin{split}
 \mathbb{E}\left\|\int_0^\cdot\mathcal {F}_{R,\epsilon}(\textbf{y})\mathrm{d} r\right\|_{\mathbb{W}^{1,p}(0,T;\mathbb{H}^{s-1} )} ^p  &\leq C \mathbb{E} \int_0^T\varpi_R ^p (\|\textbf{y}\|_{\mathbb{W}^{1,\infty}})\|\textbf{y}\|_{W^{1,\infty}}^p\|\textbf{y}\|_{\mathbb{H}^s}^p\mathrm{d} t\\
 &\leq CR ^p  \mathbb{E} \sup_{t\in[0,T]}\|\textbf{y}(t)\|_{\mathbb{H}^s}^p <\infty.
\end{split}
\end{equation}
For the stochastic term in \eqref{2.31}, we have
\begin{equation}\label{2.33}
\begin{split}
 &\mathbb{E}\left\|\int_0^\cdot\mathcal {G}_R(r,\textbf{y})\mathrm{d}\mathcal {W}\right\|_{\mathbb{W}^{\alpha,p}(0,T;\mathbb{H}^{s-1} )} ^p \\
 &\quad\leq  C  \sum_{i=1,2}\mathbb{E}\int_0^T\mu_i  ^p(t)\varpi_R^p  (\|\textbf{y}\|_{\mathbb{W}^{1,\infty}}) \chi_i ^p(\|\textbf{y}\|_{\mathbb{W}^{1,\infty}})(1+\|\textbf{y}\|_{\mathbb{H}^{s}} ^p)   \mathrm{d}t\\
&\quad\leq C\sum_{i=1,2}\chi_i ^p(2R) \int_0^T\mu_i  ^p(t)\mathrm{d}t\left(1+\mathbb{E}\sup_{t\in[0,T]}\|\textbf{y}(t)\|_{\mathbb{H}^{s}} ^p\right) \leq  C.
\end{split}
\end{equation}
As a result, it follows from \eqref{2.31}-\eqref{2.33} that
\begin{equation*}
\begin{split}
\sup_{0<\epsilon<1}\mathbb{E}\left(\|\textbf{y}_{R,\epsilon}\|_{W^{\alpha,p}(0,T;\mathbb{H}^{s-1} )} ^p\right)\leq C,
\end{split}
\end{equation*}
for some constant $C>0$ independent of $\epsilon$, which implies that the approximations $\{\textbf{y}_{R,\epsilon}\}_{0<\epsilon<1}$ is uniformly bounded in $L^p(\Omega;W^{\alpha,p}(0,T;\mathbb{H}^{s-1}(\mathbb{T}^d) ))$, for any $T>0$.

In addition, similar to the proof in \textsf{Step 1}, one can also derive  $\mathbb{E}\sup_{\varsigma\in [r,t]}\| \textbf{y} (\varsigma)\|_{\mathbb{H}^s}^p\leq C$ for any $0<r<t<T$, which together with the H\"{o}lder inequality and BDG inequality lead  to
\begin{equation}\label{2.34}
\begin{split}
 &\mathbb{E} \left(\|\textbf{y}_{R,\epsilon}(t )-\textbf{y}_{R,\epsilon}(r)\|_{\mathbb{H}^{s-1} }^p\right)\\
 &\quad \leq C \mathbb{E}\left(\int_r^t\|\mathcal {F}_{R,\epsilon}(\textbf{y})\|_{\mathbb{H}^{s-1} }^2\mathrm{d}\varsigma\right)^{\frac{p}{2}}+C\mathbb{E}\left(\int_r^t\|\mathcal {G}_R(\varsigma,\textbf{y})\|_{L_2(\mathfrak{U}_1;\mathbb{H}^{s-1} )}^2\mathrm{d}\varsigma\right)^{\frac{p}{2}}\\
 &\quad \leq C \mathbb{E}\sup_{\varsigma\in [r,t]}\left( \|\mathcal {F}_{R,\epsilon}(\textbf{y})\|_{\mathbb{H}^{s-1} }^2+\|\mathcal {G}_R(\varsigma,\textbf{y})\|_{L_2(\mathfrak{U}_1;\mathbb{H}^{s-1} )}^2\right)|t-r|^{\frac{p}{2}}\\
  &\quad\leq C |t-r|^{\frac{p}{2}},
\end{split}
\end{equation}
where the constant $C>0$ is independent of $\epsilon$. Thanks to the Kolmogorov Continuity Theorem (cf. Theorem 3.3 in \cite{43}), the uniform estimate \eqref{2.34} implies that the approximation  $\textbf{y}_{R,\epsilon}$ has a continuous modification in $\mathcal {C}^\beta([0,T];\mathbb{H}^{s-1} )$, and
$$
\mathbb{E} \|\textbf{y}_{R,\epsilon}\|_{\mathcal {C}^\beta([0,T];\mathbb{H}^{s-1} )}^p \leq C,\quad \beta\in (0,\frac{1}{2}-\frac{1}{p}),
$$
for some positive constant $C $ independent of $\epsilon$. This completes the proof of Lemma \ref{lem:2.2}.
\end{proof}

\subsection{Smooth strong solutions}
\subsubsection{Martingale solutions} With the uniform bounds in Lemma \ref{lem:2.2}, we shall prove by the stochastic compactness method that the probability measures $\{\mathbb{P}\circ \textbf{y}_{R,\epsilon}^{-1}\}_{0<\epsilon<1}$ induced by the approximations  $\{\textbf{y}_{R,\epsilon}\}_{0<\epsilon<1}$ is weakly compact.

For any $s>1+\frac{d}{2} $, $d\geq 1$, $R>1$ and $0<\epsilon<1$, we consider  the phase space
\begin{equation*}
\begin{split}
\mathcal {X} ^s = \mathcal {X} _{\textbf{y}}^s \times \mathcal {X} _\mathcal {W},\quad \textrm{where}~~\mathcal {X} ^s _{\textbf{y}} = \mathcal {C}([0,T];\mathbb{H}^s(\mathbb{T}^d)) ,~~\mathcal {X} _\mathcal {W}= \mathcal {C}([0,T];\mathfrak{U}_1\times \mathfrak{U}_1),
\end{split}
\end{equation*}
On the given probability space $(\Omega,\mathcal {F},\mathbb{P})$, we define
\begin{equation*}
\begin{split}
\mu ^{R,\epsilon}=\mu _{\textbf{y}}^{R,\epsilon}\times \mu _ \mathcal {W} \in \textrm{Pr}(\mathcal {X}^s),\quad \textrm{where}~~ \mu _{\textbf{y}}^{R,\epsilon} =\mathbb{P}\circ (\textbf{y}_{R,\epsilon})^{-1} ,~~\mu _ \mathcal {W}  =\mathbb{P}\circ\mathcal {W}^{-1},
\end{split}
\end{equation*}
where $\textrm{Pr}(\mathcal {X}^s)$ is the collection of Borel probability measures on $\mathcal {X}^s$.

Recalling that a collection $\mathscr{O}\subset \textrm{Pr}(\mathcal {X}^s)$  is \textsf{tight} on $\mathcal {X}^s$ if, for every $\gamma>0$, there exists a compact set $K_\gamma\subset \mathcal {X}$ such that, $\nu(K_\gamma)\geq 1-\gamma$ for all $\nu\in \mathscr{O}$.

We have the following weakly compact result.

\begin{lemma}\label{lem:2.3}
Let $s>4+\frac{d}{2} $, $r>2$, $R>1$ and $0<\epsilon<1$, assume that the conditions \eqref{a1}-\eqref{a2} hold and consider any $\mu_0\in Pr(\mathcal {X} ^{s-1})$ with $\int_{\mathcal {X} ^{s-1}}|\textbf{y}|^r \mu_0(\mathrm{d}\textbf{y})$, for some $r>2$. Suppose that $\{ {\textbf{y}}_{R,\epsilon}\}_{0<\epsilon < 1}$ are solutions to SDEs \eqref{2.16} with respect to the initial data $\textbf{y}_0$ satisfying  $\mu_0=\mathbb{P}\circ \textbf{y}_0^{-1}$. Then the sequence of probability measures $\{\mu ^{R,\epsilon}\}_{0<\epsilon < 1}$ is tight on $\mathcal {X}^{s-1}$, and hence has a weakly convergent subsequence in $\textrm{Pr}(\mathcal {X} ^{s-1})$.
\end{lemma}

\begin{proof}[\emph{\textbf{Proof.}}]
It suffices to prove that, for every $\eta>0$, there exists a relatively compact set $K_\eta\in \mathcal {X}^{s-1} $ such that $\mu ^{R,\epsilon} (\overline{K_\eta})\geq 1-\eta$, for all $\epsilon\in(0,1)$.

Indeed, choosing   $\alpha\in (0,\frac{1}{2}-\frac{1}{p})$ such that $\alpha p >1$. Due to the Theorem 2.1 in \cite{flandoli1995martingale}, one find that both $\mathbb{W}^{1,p}(0,T;\mathbb{H}^{s-1}(\mathbb{T}^d) )$ and $\mathbb{W}^{\alpha,p}(0,T;\mathbb{H}^{s}(\mathbb{T}^d)  )$ are compactly embedded in $\mathcal {C}([0,T];\mathbb{H}^{s-1}(\mathbb{T}^d) )$. For any $L>0$, the set
\begin{equation}\label{2.35}
\begin{split}
B(L)\triangleq \{\textbf{y}: \|\textbf{y}\|_{\mathbb{W}^{1,p}(0,T;\mathbb{H}^{s-1} )}< L \}\bigcap \{\textbf{y}: \|\textbf{y}\|_{\mathbb{W}^{\alpha,p}(0,T;\mathbb{H}^{s} )}< L \}
\end{split}
\end{equation}
is pre-compact in $\mathcal {X} _{\textbf{y}}^{s-1} $. Define the balls
\begin{equation*}
\begin{split}
B_1(L)&\triangleq  \left\{\textbf{y}: \|\int_0^\cdot\mathcal {G}_R(r,\textbf{y} )\mathrm{d}\mathcal {W} \|_{\mathbb{W}^{\alpha,p}(0,T;\mathbb{H}^{s-1} )}  < L \right\},\\
B_2(L) &\triangleq  \left\{\textbf{y}: \|\textbf{y} -\int_0^\cdot\mathcal {G}_R(r,\textbf{y} )\mathrm{d}\mathcal {W} \|_{\mathbb{W}^{1,p}(0,T;\mathbb{H}^{s-1} )} < L \right\}.
\end{split}
\end{equation*}
Simple calculation shows that $B_1(L)\bigcap B_2(L)\subseteq  B(L)$. By \eqref{2.35}, the uniform momentum estimates \eqref{2.19} and the Chebyshev inequality, we have
\begin{equation}\label{2.36}
\begin{split}
\mu _{\textbf{y}}^{R,\epsilon}\left(\overline{B(L)}^c \right) \leq& \mathbb{P}  \left( \left\|\textbf{y} -\int_0^\cdot\mathcal {G}_R(r,\textbf{y} )\mathrm{d}\mathcal {W} \right\|_{\mathbb{W}^{1,p}(0,T;\mathbb{H}^{s-1} )} \geq L  \right)\\
 & + \mathbb{P} \left( \left\|\int_0^\cdot\mathcal {G}_R(r,\textbf{y} )\mathrm{d}\mathcal {W} \right\|_{\mathbb{W}^{\alpha,p}(0,T;\mathbb{H}^{s-1} )}  \geq L  \right)  \\
 \leq& \mathbb{E} \left\|\textbf{y} -\int_0^\cdot\mathcal {G}_R(r,\textbf{y} )\mathrm{d}\mathcal {W} \right\|_{\mathbb{W}^{1,p}(0,T;\mathbb{H}^{s-1} )}^p +\frac{1}{L^p}  \mathbb{E}\left\|\int_0^\cdot\mathcal {G}_R(r,\textbf{y} )\mathrm{d}\mathcal {W} \right\|_{\mathbb{W}^{\alpha,p}(0,T;\mathbb{H}^{s-1} )}^p
 \\
\leq&  \frac{C}{L^p} ,
\end{split}
\end{equation}
for some positive constant $C$ independent of $\epsilon$. By choosing $L=L(\eta) >(C/\eta)^{\frac{1}{p}}$, one can derive from the estimate \eqref{2.36} that
$$
\mu _{\textbf{y}}^{R,\epsilon}\left(\overline{B(L)}\right)=1-\mu _{\textbf{y}}^{R,\epsilon}\left(\overline{B(L)}^c \right) > 1- \frac{C}{L^p}> 1-\eta,\quad \textrm{for all} ~~0<\epsilon<1,
$$
which implies that the collection of probability measures
$\{\mu _{\textbf{y}}^{R,\epsilon}\}_{0<\epsilon <1}$ is  tight on $\mathcal {X} _{\textbf{y}}^{s-1} $. Moreover, since the sequence $\{\mu_\mathcal {W}\}$ is constant, it is trivially weakly compact and hence tight. As a result, one  may finally infer that the sequence $\{\mu ^{R,\epsilon}\}_{0<\epsilon<1}$ is tight on $\mathcal {X} ^{s-1} $. This finishes the proof of Lemma \ref{lem:2.3}.
\end{proof}

Based on the weak compactness result in Lemma \ref{lem:2.3}, one can now start to prove the existence of global martingale solutions to the truncated SMEP2 \eqref{2.15}.
\begin{lemma} (\cite{yan2022initial})\label{lem:2.4}
Let $s>1+\frac{d}{2}$, $d\geq 1$. The functional $F(\textbf{y})$ defined as in \eqref{1.8} satisfies
\begin{equation}\label{2.37}
\begin{split}
\|F(\textbf{y})\|_{\mathbb{H}^s}\leq \|\textbf{y}\|_{\mathbb{W}^{1,\infty}}\|\textbf{y}\|_{\mathbb{H}^s},\quad \forall \textbf{y}\in \mathbb{H}^s(\mathbb{T}^d),
\end{split}
\end{equation}
\begin{equation}\label{2.38}
\begin{split}
\|F(\textbf{y}_1)-F(\textbf{y}_2)\|_{\mathbb{H}^s}\leq   C( \|\textbf{y}_1\|_{\mathbb{H}^s}+ \|\textbf{y}_2\|_{\mathbb{H}^s}) \|\textbf{y}_1-\textbf{y}_2\|_{\mathbb{H}^s} ,
\end{split}
\end{equation}
for any $\textbf{y}_1=(u_1,\gamma_1)$, $\textbf{y}_2=(u_2,\gamma_2) \in \mathbb{H}^s(\mathbb{T}^d)$.
\end{lemma}

The following theorem ensures the existence of global martingale solutions to  SDEs \eqref{2.16} in a new probability space.

\begin{lemma} \label{lem:2.5}
Fix any $s>4+\frac{d}{2}$, $d\geq 1$ and $R\geq1$. Suppose that the conditions   \eqref{a1}-\eqref{a2} hold, and  $\mu_0\in Pr(\mathcal {X} ^{s-1})$ is a  given initial distribution satisfying $\int_{\mathcal {X} ^{s-1}}|\textbf{y}|^r \mu_0(\mathrm{d}\textbf{y})$, for some $r>2$. Then there exists a new stochastic basis $\widetilde{\mathcal {S}}\triangleq( \widetilde{\Omega} , \widetilde{\mathcal {F}},\{\widetilde{\mathcal {F}}_t\}_{t\geq0},\widetilde{\mathbb{P}},\widetilde{\mathcal {W}})$ and a $\widetilde{\mathcal {F}}_t$-predictable process
$$
\widetilde{\textbf{y}}_{R}(\cdot):\Omega\mapsto  \mathcal {C}([0,T];\mathbb{H}^{s-1}(\mathbb{T}^d) ) ,\quad \textrm{for any}~~ T>0,
$$
such that $\widetilde{\mathbb{P}}\circ \widetilde{\textbf{y}}_{R}(0)^{-1}=\mathbb{P}\circ \textbf{y}_0^{-1}$, and the following equation
\begin{equation}\label{2.38}
\begin{split}
& \widetilde{\textbf{y}}_{R}(t)+ \int_0^t\varpi_R  (\|\widetilde{\textbf{y}}_{R}\|_{\mathbb{W}^{1,\infty}}) B(\widetilde{\textbf{y}}_{R},\widetilde{\textbf{y}}_{R})\mathrm{d} \tau\\
&\quad =\widetilde{\textbf{y}}_{R}(0)+\int_0^t\varpi_R (\|\widetilde{\textbf{y}}_{R}\|_{\mathbb{W}^{1,\infty}})F(\widetilde{\textbf{y}}_{R})\mathrm{d} \tau+\int_0^t \varpi _R (\|\widetilde{\textbf{y}}_{R}\|_{\mathbb{W}^{1,\infty}})G(\tau,\widetilde{\textbf{y}}_{R}) \mathrm{d}\widetilde{\mathcal {W}}
\end{split}
\end{equation}
holds $\widetilde{\mathbb{P}}$-almost surely, for all $t\in [0,T]$.
\end{lemma}
\begin{proof}[\emph{\textbf{Proof.}}]
{\textsf{Step 1 (Existence):}} Due to the weakly compactness result in Lemma \ref{lem:2.3} and $\mathcal {X} ^{s-1}$ is a separable complete metric space, one infer from the Prokhorov Theorem (cf. Theorem 2.3 in \cite{43}) that there exists a probability measure $\mu  ^{R}\in \textrm{Pr}(\mathcal {X} ^{s-1})$ and a subsequence of $\{\mu  ^{R,\epsilon}\}_{0<\epsilon <1}$, still denoted by itself, such that
$
\mu ^{R,\epsilon}\rightharpoonup \mu  ^{R}$, as $\epsilon\rightarrow 0$.
It then follows from the Skorokhod Representation Theorem (cf. Theorem 2.6.1 in \cite{breit2018local}) that there exist a new  probability space $( \widetilde{\Omega} , \widetilde{\mathcal {F}},\widetilde{\mathbb{P}})$, on which defined a sequence of $\mathcal {X} ^{s-1}$-valued random elements $\{(\widetilde{\textbf{y}}_{R,\epsilon},\widetilde{\mathcal {W}}_{\epsilon})\}_{0<\epsilon <1}=\{(\widetilde{u}_{R,\epsilon},\widetilde{\gamma}_{R,\epsilon},\widetilde{\mathcal {W}}_{1,\epsilon},\widetilde{\mathcal {W}}_{2,\epsilon})\}_{0<\epsilon <1}$ converging almost surely in $\mathcal {X}^{s-1}$ to an element $(\widetilde{\textbf{y}}_{R},\widetilde{\mathcal {W}})=(\widetilde{u}_{R},\widetilde{\gamma}_{R},\widetilde{\mathcal {W}}_1,\widetilde{\mathcal {W}}_2)$, that is,
\begin{equation}\label{2.39}
\begin{split}
 &\widetilde{u}_{R,\epsilon}\rightarrow \widetilde{u}_{R} , \quad \widetilde{\gamma}_{R,\epsilon}\rightarrow \widetilde{\gamma}_{R},  \quad   \quad \textrm{in}  ~~~ C([0,T];\mathbb{H}^{s-1}(\mathbb{T}^d)), \quad  \widetilde{\mathbb{P}}\textrm{-a.s.,}
\end{split}
\end{equation}
and
\begin{equation}\label{2.40}
\begin{split}
 &\widetilde{\mathcal {W}}_{1,\epsilon}\rightarrow \widetilde{\mathcal {W}}_{1}, \quad \widetilde{\mathcal {W}}_{2,\epsilon}\rightarrow \widetilde{\mathcal {W}}_{2},  \quad \textrm{in}~~~ C([0,T];\mathfrak{U}_1), \quad \widetilde{\mathbb{P}}\textrm{-a.s.}
\end{split}
\end{equation}
Notice that by Theorem 2.1.35 and Corollary 2.1.36 in \cite{breit2018stochastically}, the random elements $\widetilde{\mathcal {W}}_{\epsilon}$ and $\widetilde{\mathcal {W}} $ are both cylindrical Wiener processes relative to the filters $\mathcal {F}_{\epsilon}^t\triangleq \sigma\{(\widetilde{\textbf{y}}_{R,\epsilon}(\tau),\widetilde{\mathcal {W}}_{\epsilon}(\tau))\}_{\tau\in [0,t]}$ and $\mathcal {F} ^t\triangleq \sigma\{(\widetilde{\textbf{y}}_{R}(\tau),\widetilde{\mathcal {W}}(\tau))\}_{\tau\in [0,t]}$, respectively.

In order to verify that the limitation $\widetilde{\textbf{y}}_{R}$ is a martingale solution to the truncated SMEP2 \eqref{2.15}, we first observe that, the uniform bounds in Lemma \ref{lem:2.2} hold true in the new probability $( \widetilde{\Omega}, \widetilde{\mathcal {F}},\widetilde{\mathbb{P}})$. So we get by using the Fatou Lemma that
\begin{equation}\label{2.41}
\begin{split}
 \widetilde{\mathbb{E}}\left(\|\widetilde{\textbf{y}}_{R}\|_{L^\infty([0,T];\mathbb{H}^{s})}^p\right)
 +\widetilde{\mathbb{E}}\left(\|\widetilde{\textbf{y}}_{R}\|_{\mathbb{W}^{\alpha,p}([0,T];\mathbb{H}^{s-1})}^p\right)
 \leq C,
\end{split}
\end{equation}
for any $T>0$ and  some positive constant $C$ independent of $\epsilon$.

Define the stochastic process
\begin{equation}\label{2.42}
\begin{split}
\mathcal {E}_\epsilon(t)\triangleq& \widetilde{\textbf{y}}_{R,\epsilon}(t)-\widetilde{\textbf{y}}_{R,\epsilon}(0) +\int_0^t\varpi_R  (\|\widetilde{\textbf{y}}_{R,\epsilon}\|_{\mathbb{W}^{1,\infty}}) B(\widetilde{\textbf{y}}_{R,\epsilon},\widetilde{\textbf{y}}_{R,\epsilon})\mathrm{d} r\\
&-\int_0^t\varpi_R  (\|\widetilde{\textbf{y}}_{R,\epsilon}\|_{\mathbb{W}^{1,\infty}})F(\widetilde{\textbf{y}}_{R,\epsilon})\mathrm{d} r .
\end{split}
\end{equation}
By Lemma \ref{lem:2.2} (under the probability space $( \widetilde{\Omega}, \widetilde{\mathcal {F}},\widetilde{\mathbb{P}})$) and \eqref{a1}-\eqref{a2}, it is easy to verify that $\mathcal {E}_\epsilon(t)$ is a $\mathbb{H}^{s-2}(\mathbb{T}^d)$-valued square integrable $\widetilde{\mathbb{P}}$-martingale, and the associated quadratic variation process is given by
\begin{equation*}
\begin{split}
[\mathcal {E}_\epsilon](t)=\int_0^t\varpi _R ^2 (\|\widetilde{\textbf{y}}_{R,\epsilon}\|_{\mathbb{W}^{1,\infty}})G(t,\widetilde{\textbf{y}}_{R,\epsilon})
G(t,\widetilde{\textbf{y}}_{R,\epsilon})^* \mathrm{d}r.
\end{split}
\end{equation*}
For any vector valued function $\varphi\in \mathscr{S}(\mathbb{T}^d)\times \mathscr{S}(\mathbb{T}^d)$ and bounded continuous function $\phi$ on $\mathcal {C}([0,s];\mathbb{H}^{s-2}(\mathbb{T}^d))\times \mathcal {C}([0,s];\mathfrak{U}_1\times\mathfrak{U}_1 )$ with $0\leq s < t \leq T$, there hold
\begin{equation*}
\begin{split}
\widetilde{\mathbb{E}}\left[(\mathcal {E}_\epsilon(t)-\mathcal {E}_\epsilon(s),\varphi)_{\mathbb{H}^{s-2}} \cdot \phi(\widetilde{\textbf{y}}_{R,\epsilon},\widetilde{\mathcal {W}})|_{[0,s]}  \right]=0,
\end{split}
\end{equation*}
and
\begin{equation*}
\begin{split}
&\widetilde{\mathbb{E}}\bigg[\bigg((\mathcal {E}_\epsilon(t),\varphi)_{\mathbb{H}^{s-2}} ^2 -(\mathcal {E}_\epsilon(t),\varphi)_{\mathbb{H}^{s-2}} ^2\\
  &\quad \quad- \int_0^t\varpi _R ^2 (\|\widetilde{\textbf{y}}_{R,\epsilon}\|_{\mathbb{W}^{1,\infty}})\|
G(t,\widetilde{\textbf{y}}_{R,\epsilon})^*\varphi\|_{\mathfrak{A}}^2 \mathrm{d}r \bigg)           \cdot \phi(\widetilde{\textbf{y}}_{R,\epsilon},\widetilde{\mathcal {W}})|_{[0,s]} \bigg]=0.
\end{split}
\end{equation*}
By applying the It\^{o} product rule, we have
\begin{equation*}
\begin{split}
 \mathrm{d}\Big(\widetilde{\textbf{b}}_k(t)(\mathcal {E}_\epsilon(t),\varphi)_{\mathbb{H}^{s-2}}\Big)=(\mathcal {E}_\epsilon(t),\varphi)_{\mathbb{H}^{s-2}}\mathrm{d} \widetilde{\textbf{b}}_k(t) + \widetilde{\textbf{b}}_k(t)\mathrm{d}(\mathcal {E}_\epsilon(t),\varphi)_{\mathbb{H}^{s-2}}+\mathrm{d} \widetilde{\textbf{b}}_k(t)\mathrm{d} (\mathcal {E}_\epsilon(t),\varphi)_{\mathbb{H}^{s-2}},
\end{split}
\end{equation*}
where $\widetilde{\mathcal {W}}(t)=\sum_{k\geq 1} e_k  \widetilde{\textbf{b}}_k(t)$ is a cylindrical Wiener process on $\mathfrak{A}\times \mathfrak{A}$, and $\widetilde{\textbf{b}}_k=(\beta_k^1,\beta_k^2)^T$ denotes the two dimensional Brownian motion. Integrating the last identity on $[s,t]$ and then taking the expectation leads to
\begin{equation*}
\begin{split}
 &\widetilde{\mathbb{E}}\bigg[\bigg(\widetilde{\textbf{b}}_k(t)(\mathcal {E}_\epsilon(t),\varphi)_{\mathbb{H}^{s-2}} -\widetilde{\textbf{b}}_k(s)(\mathcal {E}_\epsilon(s),\varphi)_{\mathbb{H}^{s-2}}\\
 &\quad \quad -\int_s^t\varpi _R (\|\widetilde{\textbf{y}}_{R,\epsilon}\|_{\mathbb{W}^{1,\infty}})
(G(t,\widetilde{\textbf{y}}_{R,\epsilon})^*\varphi,e_j)_\mathfrak{A} \mathrm{d}r\bigg)        \cdot \phi(\widetilde{\textbf{y}}_{R,\epsilon},\widetilde{\mathcal {W}})|_{[0,s]}\bigg]=0.
\end{split}
\end{equation*}
Thanks to Lemma \ref{lem:2.4}, we find
\begin{equation}\label{2.43}
\begin{split}
 &\left\|\int_0^t\varpi_R  (\|\widetilde{\textbf{y}}_{R,\epsilon}\|_{\mathbb{W}^{1,\infty}})F(\widetilde{\textbf{y}}_{R,\epsilon}) - \varpi_R  (\|\widetilde{\textbf{y}}_{R}\|_{\mathbb{W}^{1,\infty}})F(\widetilde{\textbf{y}}_{R})\mathrm{d} r\right\|_{\mathbb{H}^s}\\
  &\quad \quad \leq C\int_0^t \varpi_R  (c\|\widetilde{\textbf{y}}_{R}\|_{\mathbb{H}^{s}}) ( \|\widetilde{\textbf{y}}_{R}\|_{\mathbb{H}^s}+ \|\widetilde{\textbf{y}}_{R,\epsilon}\|_{\mathbb{H}^s}) \|\widetilde{\textbf{y}}_{R,\epsilon}-\widetilde{\textbf{y}}_{R}\|_{\mathbb{H}^s}\mathrm{d} r\\
 &\quad \quad\quad \quad+C\int_0^t\|\varpi_R '\|_{L^\infty}\|
 \widetilde{\textbf{y}}_{R}\|_{\mathbb{W}^{1,\infty}}\|\widetilde{\textbf{y}}_{R}\|_{\mathbb{H}^s}
 \|\widetilde{\textbf{y}}_{R,\epsilon}-\widetilde{\textbf{y}}_{R}\|_{\mathbb{W}^{1,\infty}}\mathrm{d} r\\
 &\quad \quad \leq   C\int_0^t \|\widetilde{\textbf{y}}_{R,\epsilon}(r)-\widetilde{\textbf{y}}_{R}(r)\|_{\mathbb{H}^s}\mathrm{d} r \rightarrow 0,\quad \textrm{as} ~~\epsilon\rightarrow0,
\end{split}
\end{equation}
where the second inequality used the properties for the cut-off functions and the uniform bounds for $\widetilde{\textbf{y}}_{R,\epsilon}$ and  $\widetilde{\textbf{y}}_{R}$ (cf. Lemma \ref{lem:2.2}, \eqref{2.43} and \eqref{2.41}). Moreover, by using the Moser-type estimates (see Lemma \ref{lem:moser}) for $s-1>1+\frac{d}{2}$, one get
\begin{equation}\label{2.44}
\begin{split}
&\|\varpi_R  (\|\widetilde{\textbf{y}}_{R,\epsilon}\|_{\mathbb{W}^{1,\infty}}) B(\widetilde{\textbf{y}}_{R,\epsilon},\widetilde{\textbf{y}}_{R,\epsilon})-\varpi_R  (\|\widetilde{\textbf{y}}_{R}\|_{\mathbb{W}^{1,\infty}}) B(\widetilde{\textbf{y}}_{R},\widetilde{\textbf{y}}_{R})\|_{\mathbb{H}^s}\\
 &\quad \leq\|\varpi_R '\|_{L^\infty} \|\widetilde{\textbf{y}}_{R,\epsilon}-\widetilde{\textbf{y}}_{R}\|_{\mathbb{W}^{1,\infty}}
\|\widetilde{\textbf{y}}_{R,\epsilon}\|_{\mathbb{H}^s}^2\\
&\quad\quad+\varpi_R  (\|\widetilde{\textbf{y}}_{R }\|_{\mathbb{W}^{1,\infty}}) (\|\widetilde{\textbf{y}}_{R,\epsilon}-\widetilde{\textbf{y}}_{R}\|_{\mathbb{H}^s}\|\widetilde{\textbf{y}}_{R,\epsilon}\|_{\mathbb{H}^s}+ \|\widetilde{\textbf{y}}_{R}\|_{\mathbb{H}^s}\|\widetilde{\textbf{y}}_{R,\epsilon}-\widetilde{\textbf{y}}_{R}\|_{\mathbb{H}^s})\\
&\quad \leq C\|\widetilde{\textbf{y}}_{R,\epsilon}-\widetilde{\textbf{y}}_{R}\|_{\mathbb{H}^s} \rightarrow 0, \quad \textrm{as} ~~\epsilon\rightarrow0,
\end{split}
\end{equation}
where the last inequality used the $\mathbb{P}$-almost surely convergence result in \eqref{2.39}. From \eqref{2.43} and \eqref{2.44}, we deduce that
\begin{equation*}
\begin{split}
 &\lim_{\epsilon\rightarrow 0}\widetilde{\mathbb{E}}\sup_{t\in[0,T]}\left|(\int_0^t\varpi_R  (\|\widetilde{\textbf{y}}_{R,\epsilon}\|_{\mathbb{W}^{1,\infty}}) B(\widetilde{\textbf{y}}_{R,\epsilon},\widetilde{\textbf{y}}_{R,\epsilon})\mathrm{d} r-\int_0^t\varpi_R  (\|\widetilde{\textbf{y}}_{R}\|_{\mathbb{W}^{1,\infty}}) B(\widetilde{\textbf{y}}_{R},\widetilde{\textbf{y}}_{R})\mathrm{d} r,\varphi)_{\mathbb{H}^{s-2}}\right| \\
 &\quad=0,\\
&\lim_{\epsilon\rightarrow 0}\widetilde{\mathbb{E}}\sup_{t\in[0,T]}\left|(\int_0^t\varpi_R  (\|\widetilde{\textbf{y}}_{R,\epsilon}\|_{\mathbb{W}^{1,\infty}})F(\widetilde{\textbf{y}}_{R,\epsilon})-\varpi_R  (\|\widetilde{\textbf{y}}_{R}\|_{\mathbb{W}^{1,\infty}})F(\widetilde{\textbf{y}}_{R})\mathrm{d} r,\varphi)_{\mathbb{H}^{s-2}}\right|=0.
\end{split}
\end{equation*}
Thereby, one can deduce from the definition of $\mathcal {E}_\epsilon(t)$ that
\begin{equation}\label{2.45}
\begin{split}
\lim_{\epsilon\rightarrow 0}\widetilde{\mathbb{E}}\sup_{t\in[0,T]}\left|(\mathcal {E}_\epsilon(t)-\mathcal {E}(t),\varphi)_{\mathbb{H}^{s-2}}\right|=0,
\end{split}
\end{equation}
where $
\mathcal {E}(t)= \widetilde{\textbf{y}}_{R}(t)-\widetilde{\textbf{y}}_{R}(0)+\int_0^t\varpi_R  (\|\widetilde{\textbf{y}}_{R}\|_{\mathbb{W}^{1,\infty}}) B(\widetilde{\textbf{y}}_{R},\widetilde{\textbf{y}}_{R})\mathrm{d} r-\int_0^t\varpi_R  (\|\widetilde{\textbf{y}}_{R}\|_{\mathbb{W}^{1,\infty}})F(\widetilde{\textbf{y}}_{R})\mathrm{d} r $.

Now let us consider the convergence result related to high order momentum in \eqref{2.45}. Indeed, since for any $\varphi\in \mathscr{S}(\mathbb{T}^d)\times \mathscr{S}(\mathbb{T}^d)$, the process $(\mathcal {E}_\epsilon(t),\varphi)_{\mathbb{H}^{s-2}}$ is a real valued martingale, so by using the assumption on $F(\cdot)$ and the BDG inequality, we obtain
\begin{equation*}
\begin{split}
&\sup_{0<\epsilon< 1} \widetilde{\mathbb{E}}|(\mathcal {E}_\epsilon(t),\varphi)_{\mathbb{H}^{s-2}}|^p\\
&\quad \leq C\sup_{0<\epsilon< 1} \widetilde{\mathbb{E}}\left(\int_0^t\mu^2 (t) \chi^2(\|\widetilde{\textbf{y}}_{R,\epsilon}\|_{\mathbb{W}^{1,\infty}})\varpi_R  (\|\widetilde{\textbf{y}}_{R,\epsilon}\|_{\mathbb{W}^{1,\infty}})
(1+\|\widetilde{\textbf{y}}_{R,\epsilon}\|_{\mathbb{H}^{s}}^2)\right)^{\frac{p}{2}}\\
&\quad \leq C\chi^p(2R)\|\mu\|_{L^p}^p\left(1+\sup_{0<\epsilon< 1}\widetilde{\mathbb{E}}\sup_{t\in [0,T]}\|\widetilde{\textbf{y}}_{R,\epsilon}(t)\|_{\mathbb{H}^{s}}^p\right)\\
&\quad \leq C\chi^p(2R)\|\mu\|_{L^p}^p,
\end{split}
\end{equation*}
for some positive constant $C$ independent of $\epsilon$. This implies that the process $|(\mathcal {E}_\epsilon(t),\varphi)_{\mathbb{H}^{s-2}}|^2$ is uniformly integrable. It then follows from the Vitali Convergence Theorem (cf. pp.187 in [19]) that
\begin{equation}\label{2.46}
\begin{split}
\lim_{\epsilon\rightarrow 0}\widetilde{\mathbb{E}}\sup_{t\in[0,T]}\left|(\mathcal {E}_\epsilon(t)-\mathcal {E}(t),\varphi)_{\mathbb{H}^{s-2}}\right|^2=0.
\end{split}
\end{equation}
Thanks to \eqref{2.46}, one can take the limit as $\epsilon\rightarrow0$ to derive that
\begin{equation*}
\begin{split}
&\widetilde{\mathbb{E}}\left[(\mathcal {E}(t)-\mathcal {E}(s),\varphi)_{\mathbb{H}^{s-2}} \cdot \phi(\widetilde{\textbf{y}}_{R},\widetilde{\mathcal {W}})|_{[0,s]}  \right]=0,\\
&\widetilde{\mathbb{E}}\left[\left((\mathcal {E}(t),\varphi)_{\mathbb{H}^{s-2}} ^2 -(\mathcal {E}(t),\varphi)_{\mathbb{H}^{s-2}} ^2- \int_0^t\varpi _R ^2 (\|\widetilde{\textbf{y}}_{R}\|_{\mathbb{W}^{1,\infty}})\|
G(t,\widetilde{\textbf{y}}_{R})^*\varphi\|_{\mathfrak{A}}^2 \mathrm{d}r \right) \cdot \phi(\widetilde{\textbf{y}}_{R},\widetilde{\mathcal {W}})|_{[0,s]} \right]\\
&=0,
\end{split}
\end{equation*}
and
\begin{equation*}
\begin{split}
 &\widetilde{\mathbb{E}}\left[\Big(\widetilde{\textbf{b}}_k(t)(\mathcal {E}(t),\varphi)_{\mathbb{H}^{s-2}} -\widetilde{\textbf{b}}_k(t)(\mathcal {E}(t),\varphi)_{\mathbb{H}^{s-2}}-\int_s^t\varpi _R (\|\widetilde{\textbf{y}}_{R}\|_{\mathbb{W}^{1,\infty}})
(G(t,\widetilde{\textbf{y}}_{R})^*\varphi,e_j)_\mathfrak{A} \mathrm{d}r\right.\Big)  \\
 &\left.\quad \cdot \phi(\widetilde{\textbf{y}}_{R},\widetilde{\mathcal {W}})|_{[0,s]}\right]=0,
\end{split}
\end{equation*}
which indicate that the limit process $\mathcal {E}(t)$ is in fact an $\widetilde{\mathcal {F}}_t$-adapted square integrable martingale taking values in $\mathbb{H}^s(\mathbb{T}^d)$. As a result, one can apply the generalized Martingale Presentation Theorem (cf. Proposition A.1 in  \cite{hofmanova2013degenerate}) to obtain
\begin{equation*}
\begin{split}
\mathcal {E}(t)= \int_0^t\varpi_R  (\|\widetilde{\textbf{y}}_{R}\|_{\mathbb{W}^{1,\infty}})G(r,\widetilde{\textbf{y}}_{R})\mathrm{d} \widetilde{\mathcal {W}},\quad \forall t\in [0,T].
\end{split}
\end{equation*}
Meanwhile, there exists a new  filtered probability space $(\widetilde{\Omega}, \widetilde{\mathcal {F}}, \widetilde{\mathcal {F}}_t ,\widetilde{\mathbb{P}})$, a cylindrical Wiener process $\widetilde{\mathcal {W}}$ and an $\widetilde{\mathcal {F}}_t$-adapted process $
\widetilde{\textbf{y}}_{R}\in L^\infty([0, T ]; \mathbb{H}^s (\mathbb{T}^d)) \cap \mathcal {C}([0, T ]; \mathbb{H}^{s-1}(\mathbb{T}^d))$,
such
that $\widetilde{\textbf{y}}_{R}$ satisfies Eq.\eqref{2.38} $\widetilde{\mathbb{P}}$-almost surely, and $\widetilde{\textbf{y}}_{R}(0)$ has the same distribution with $\widetilde{\textbf{y}}(0)= \widetilde{\textbf{y}}_0$. In addition, the solution $\widetilde{\textbf{y}}_{R}(\cdot)$ exists globally as $T > 0 $ is arbitrary.

{\textsf{Step 2 (Regularity):}} We are going to show that $
\widetilde{\textbf{y}}_{R}\in \mathcal {C}([0, T ]; \mathbb{H}^s (\mathbb{T}^d))$, $ \widetilde{\mathbb{P}}$-a.s. Indeed, it is shown in \textsf{Step 1} that $\widetilde{\textbf{y}}_{R}$ belongs to $ L^\infty([0, T ]; \mathbb{H}^s (\mathbb{T}^d)) \cap \mathcal {C}([0, T ]; \mathbb{H}^{s-1}(\mathbb{T}^d))$, so by the continuous embedding $\mathbb{H}^{s}(\mathbb{T}^d)\subset \mathbb{H}^{s-1}(\mathbb{T}^d)$, we infer that the solution is weakly continuous in $\mathbb{H}^{s}(\mathbb{T}^d)$ (cf. Lemma II.5.9 in \cite{boyer2012mathematical}), that is, $
\widetilde{\textbf{y}}_{R}\in \mathcal {C}([0, T ]; \mathbb{H}^{s}(\mathbb{T}^d)_{\textrm{weak}})$,
which implies that
\begin{equation}\label{2.47}
\begin{split}
\widetilde{\mathbb{E}}\|\widetilde{\textbf{y}}_{R}(\varsigma)\|_{\mathbb{H}^{s}}^2
=&\widetilde{\mathbb{E}}\left|\sup_{\varphi\in \mathbb{H}^{-s},~~\|\varphi\|_{\mathbb{H}^{-s}}=1}\lim_{t\rightarrow \varsigma^+}(\widetilde{\textbf{y}}_{R}(t),\varphi)_{\mathbb{H}^{s},\mathbb{H}^{-s}}\right|^2\\
\leq &\widetilde{\mathbb{E}}\liminf_{t\rightarrow \varsigma^+} \left(\sup_{\varphi\in \mathbb{H}^{-s},~~\|\varphi\|_{\mathbb{H}^{-s}}=1}
|(\widetilde{\textbf{y}}_{R}(t),\varphi)_{\mathbb{H}^{s},\mathbb{H}^{-s}}|^2\right)
\\
\leq & \liminf_{t\rightarrow \varsigma^+}\widetilde{\mathbb{E}} \|\widetilde{\textbf{y}}_{R}(t)\|_{\mathbb{H}^{s}}^2.
\end{split}
\end{equation}
As $\widetilde{\textbf{y}}_{R}(\cdot)$ solves the truncated SMEP2 \eqref{2.14} $\widetilde{\mathbb{P}}$-almost surely, one can apply the It\^{o} formula to $ \|\widetilde{\textbf{y}}_{R}(r)\|_{\mathbb{H}^{s}}^2$.  After utilizing the BDG inequality and the assumption \eqref{a1}, we obtain
\begin{equation*}
\begin{split}
\widetilde{\mathbb{E}}\sup_{r\in [\varsigma,t]}\| \widetilde{\textbf{y}}_{R} (r)\|_{\mathbb{H}^s}^2\leq& \| \widetilde{\textbf{y}}_{R} (\varsigma)\|_{\mathbb{H}^s}^2 +C\int_\varsigma^t  \varpi_R  (\|\widetilde{\textbf{y}}_{R}\|_{\mathbb{W}^{1,\infty}}) \|\widetilde{\textbf{y}}_{R}\|_{\mathbb{W}^{1,\infty}}
\|\widetilde{\textbf{y}}_{R}\|_{\mathbb{H}^s}^2\mathrm{d}r \\
&+ C\widetilde{\mathbb{E}}\left(\int_\varsigma^t \mu^2 (t) \varpi_R ^2 (\|\widetilde{\textbf{y}}_{R}\|_{\mathbb{W}^{1,\infty}}) \chi^2(\|\textbf{y}\|_{\mathbb{W}^{1,\infty}})\| \widetilde{\textbf{y}}_{R}\|_{\mathbb{H}^{s}} ^2 (1+\|\widetilde{\textbf{y}}_{R}\|_{\mathbb{H}^{s}}^2)\mathrm{d} r\right)^{\frac{1}{2}}\\
&+C\widetilde{\mathbb{E}}\int_\varsigma^t\mu^2 (t) \varpi_R ^2 (\|\widetilde{\textbf{y}}_{R}\|_{\mathbb{W}^{1,\infty}}) \chi^2(\|\textbf{y}\|_{\mathbb{W}^{1,\infty}}) (1+\|\widetilde{\textbf{y}}_{R}\|_{\mathbb{H}^{s}}^2)\mathrm{d}r.
\end{split}
\end{equation*}
 Thereby, it follows from \eqref{2.48} and the uniform bound for $\widetilde{\textbf{y}}_{R}$ in $\mathbb{H}^{s}(\mathbb{T}^d)$ (cf. \eqref{2.41}) that
\begin{equation*}
\begin{split}
\widetilde{\mathbb{E}} \| \widetilde{\textbf{y}}_{R} (t)\|_{\mathbb{H}^s}^2\leq\widetilde{\mathbb{E}}\sup_{r\in [\varsigma,t]}\| \widetilde{\textbf{y}}_{R} (r)\|_{\mathbb{H}^s}^2\leq \widetilde{\mathbb{E}}\| \widetilde{\textbf{y}}_{R} (\varsigma)\|_{\mathbb{H}^s}^2 +C(|t-\varsigma|+|t-\varsigma|^{\frac{1}{2}}).
\end{split}
\end{equation*}
By taking the superior limit $t\rightarrow \varsigma^+$ in last inequality leads to
\begin{equation} \label{2.48}
\begin{split}
\limsup_{t\rightarrow \varsigma^+}\widetilde{\mathbb{E}}\|\widetilde{\textbf{y}}_{R}(t)\|_{\mathbb{H}^{s}}^2 \leq\limsup_{t\rightarrow \varsigma^+}\Big(\widetilde{\mathbb{E}}\| \widetilde{\textbf{y}}_{R} (\varsigma)\|_{\mathbb{H}^s}^2 +C(|t-\varsigma|+|t-\varsigma|^{\frac{1}{2}}\Big) = \widetilde{\mathbb{E}} \|\widetilde{\textbf{y}}_{R}(\varsigma)\|_{\mathbb{H}^{s}}^2 .
\end{split}
\end{equation}
Combining \eqref{2.47} and \eqref{2.48} yields that
\begin{equation*}
\begin{split}
\widetilde{\mathbb{E}} \|\widetilde{\textbf{y}}_{R}(\varsigma)\|_{\mathbb{H}^{s}}^2\leq \liminf_{t\rightarrow \varsigma^+}\widetilde{\mathbb{E}} \|\widetilde{\textbf{y}}_{R}(t)\|_{\mathbb{H}^{s}}^2\leq  \limsup_{t\rightarrow \varsigma^+}\widetilde{\mathbb{E}} \|\widetilde{\textbf{y}}_{R}(t)\|_{\mathbb{H}^{s}}^2\leq \widetilde{\mathbb{E}} \|\widetilde{\textbf{y}}_{R}(\varsigma)\|_{\mathbb{H}^{s}}^2,
\end{split}
\end{equation*}
which implies that
$$
\widetilde{\mathbb{E}} \|\widetilde{\textbf{y}}_{R}(\varsigma)\|_{\mathbb{H}^{s}}^2= \lim_{t\rightarrow \varsigma^+}\widetilde{\mathbb{E}} \|\widetilde{\textbf{y}}_{R}(t)\|_{\mathbb{H}^{s}}^2.
$$
In a similar manner, one can also prove that $\widetilde{\mathbb{E}} \|\widetilde{\textbf{y}}_{R}(\varsigma)\|_{\mathbb{H}^{s}}^2= \lim_{t\rightarrow \varsigma^-}\widetilde{\mathbb{E}} \|\widetilde{\textbf{y}}_{R}(t)\|_{\mathbb{H}^{s}}^2$. Hence, the solution $\widetilde{\textbf{y}}_{R}$ is actually strongly continuous in $\mathbb{H}^{s}(\mathbb{T}^d)$.
 This completes the proof of Lemma \ref{lem:2.5}.
\end{proof}
\subsubsection{Pathwise uniqueness}

\begin{lemma} \label{lem:2.6}
Let $s>4+\frac{d}{2}$, $R\geq 1$, and the conditions \eqref{a1}-\eqref{a2} hold.  Suppose that $(\textbf{y}_{R}^{(i)},\mathcal {S})$,  $\textbf{y}_{R}^{(i)}=(u_{R}^{(i)},\gamma_{R}^{(i)})$, $i=1,2$, are two global solutions of \eqref{2.38} in the sense of Lemma \ref{lem:2.5}, relative to the same stochastic basis $\mathcal {S}=( \Omega , \mathcal {F},\{{\mathcal {F}}_t\}_{t\geq0},{\mathbb{P}})$. If ${\textbf{y}}_{R}^{(1)}(0)={\textbf{y}}_{R}^{(2)}(0)=\textbf{y} _0$ $\mathbb{P}$-almost surely with $\mathbb{E}\|\textbf{y} _0\|_{\mathbb{H}^s}^r<\infty$ for some $r>2$, then we have
$$
\mathbb{P}\left\{\textbf{y}_{R}^{(1)}(t)=\textbf{y}_{R}^{(2)}(t),~~~\forall t\geq 0\right\}=1.
$$
\end{lemma}
\begin{proof}[\emph{\textbf{Proof.}}]
Notice that the continuity in time is  ensured for $\mathbb{H}^{s-1}$-norm, so by  $s>4+\frac{d}{2}$ one can define the stopping times $
\mathbbm{t}_K\triangleq \inf \{t\geq 0;~~ \|\textbf{y}_{R}^{(1)}(t)\|_{\mathbb{H}^{s-1}}^2+ \|\textbf{y}_{R}^{(2)}(t)\|_{\mathbb{H}^{s-1}}^2> K \}$, $  K>0$.
It is sufficient to verify the Lipschitz continuous of the coefficients in $\mathbb{H}^{s-2}(\mathbb{T}^d)$, and then uniqueness follows by classical results (cf. \cite{43}). Here we just deal with the term $\mathscr{A}(\textbf{y})\triangleq\varpi_R  (\|y\|_{\mathbb{W}^{1,\infty}}) B(y,y)$, since the other terms can be treated in a similar manner.

Indeed, we note that
\begin{equation}\label{ww}
\begin{split}
&\mathscr{A}(\textbf{y}_{R}^{(1)})-\mathscr{A}(\textbf{y}_{R}^{(2)})\\
&\quad =\varpi_{1,2}(t)B(\textbf{y}_{R}^{(2)},\textbf{y}_{R}^{(2)})+\varpi_R  (\|\textbf{y}_{R}^{(1)}\|_{\mathbb{W}^{1,\infty}}) (B(\textbf{z}^{12},\textbf{y}_{R}^{(1)})+ B(\textbf{y}_{R}^{(2)},\textbf{z}^{12}) ),
\end{split}
\end{equation}
where $\textbf{z}^{12}=\textbf{y}_{R}^{(1)} - \textbf{y}_{R}^{(2)}$ and $\varpi_{1,2}(t)=\varpi_R  (\|\textbf{y}_{R}^{(1)}\|_{\mathbb{W}^{1,\infty}})-\varpi_R  (\|\textbf{y}_{R}^{(2)}\|_{\mathbb{W}^{1,\infty}})$.

For the first term on the R.H.S. of \eqref{ww}, we first get by   Mean-Value Theorem that
\begin{equation}\label{zz}
\begin{split}
|\varpi_{1,2} | &=|\varpi_R'  (\theta\|\textbf{y}_{R}^{(1)}\|_{\mathbb{W}^{1,\infty}} +(1-\theta) \|\textbf{y}_{R}^{(2)}\|_{\mathbb{W}^{1,\infty}}  )(\|\textbf{y}_{R}^{(1)}\|_{\mathbb{W}^{1,\infty}}-\|\textbf{y}_{R}^{(2)}\|_{\mathbb{W}^{1,\infty}})|\\
 &\leq C \|\textbf{y}_{R}^{(1)}-\textbf{y}_{R}^{(2)}\|_{\mathbb{W}^{1,\infty}}.
\end{split}
\end{equation}
Then by using the Moser-type estimate and   $\mathbb{H}^{s-1}(\mathbb{T}^d)\hookrightarrow\mathbb{H}^{s-2}(\mathbb{T}^d)\hookrightarrow \mathbb{W}^{1,\infty}(\mathbb{T}^d)$, we get
\begin{equation*}
\begin{split}
\|B(\textbf{y}_{R}^{(2)},\textbf{y}_{R}^{(2)})\|_{\mathbb{H}^{s-2}}
&\leq C\left(\|u_{R}^{(2)}\|_{L^\infty}\|\nabla \textbf{y}_{R}^{(2)} \|_{\mathbb{H}^{s-2}}+\|u_{R}^{(2)}\|_{H^{s-2}}\| \nabla \textbf{y}_{R}^{(2)} \|_{\mathbb{L}^{\infty}}  \right) \\
&\leq C \|\textbf{y}_{R}^{(2)} \|_{\mathbb{H}^{s-1}}^2,
\end{split}
\end{equation*}
which combined with \eqref{zz} lead to
\begin{equation}\label{cc}
\begin{split}
\|\varpi_{1,2}B(\textbf{y}_{R}^{(2)},\textbf{y}_{R}^{(2)})\|_{\mathbb{H}^{s-2}}
 &\leq C \|\textbf{y}_{R}^{(2)} \|_{\mathbb{H}^{s-1}}^2 \|\textbf{y}_{R}^{(1)}-\textbf{y}_{R}^{(2)}\|_{\mathbb{H}^{s-2}}\\
 &\leq C  \|\textbf{y}_{R}^{(1)}-\textbf{y}_{R}^{(2)}\|_{\mathbb{H}^{s-2}},\quad \textrm{on the interval}~~[0,t\wedge\mathbbm{t}_K],
\end{split}
\end{equation}
for some constant $C>0$ depending on $R$ and $K$.

Similarly,  we also have
\begin{equation} \label{xx}
\begin{split}
&|\varpi_R  (\|\textbf{y}_{R}^{(1)}\|_{\mathbb{W}^{1,\infty}}) (B(\textbf{z}^{12},\textbf{y}_{R}^{(1)})+ B(\textbf{y}_{R}^{(2)},\textbf{z}^{12}) )|\leq C  \|\textbf{y}_{R}^{(1)}-\textbf{y}_{R}^{(2)}\|_{\mathbb{H}^{s-2}}.
\end{split}
\end{equation}
It follows from \eqref{ww}, \eqref{cc} and \eqref{xx} that
\begin{equation*}
\begin{split}
&\|\mathscr{A}(\textbf{y}_{R}^{(1)})-\mathscr{A}(\textbf{y}_{R}^{(2)}) \|_{\mathbb{H}^{s-2}}\leq C  \|\textbf{y}_{R}^{(1)}-\textbf{y}_{R}^{(2)}\|_{\mathbb{H}^{s-2}},\quad \textrm{on the interval}~~[0,t\wedge\mathbbm{t}_K],
\end{split}
\end{equation*}
for arbitrarily  fixed $R$ and $K$. This proves the Lipschitz continuity of $\mathscr{A}(\cdot)$ in $ \mathbb{H}^{s-2}(\mathbb{T}^d)$.  The proof of Lemma \ref{lem:2.6} is now completed.
\end{proof}

\subsubsection{Regularized pathwise solution}
With the pathwise uniqueness and the existence of martingale solutions at hand, we could now establish the existence of local pathwise solutions in sufficient regular spaces, whose proof is based on the following result.

\begin{lemma} [Gy\"{o}ngy-Krylov Lemma \cite{47}] \label{lem:2.7}
Let $X$ be a Polish space, and $\{Y_j \}_{j\geq0}$ be a
sequence of $X$-valued random elements. We define the collection of joint
laws $\{\nu_{j,l}\}_{j,l\geq1}$ of $\{Y_j \}_{j\geq1}$ by
$$
\nu_{j,l}(E)=\mathbb{P}\{(Y_j,Y_l)\in E\},~~~E\in \mathcal {B}(X\times X).
$$
Then $\{Y_j \}_{j\geq0}$ converges in probability if and only if for every subsequence $\{\nu_{j_k,l_k}\}_{k\geq0}$, there exists a further subsequence which converges
weakly to a probability measure $\nu$ such that
$$
\nu(\{(u,v)\in X \times X:u=v\})=1.
$$
\end{lemma}

The main result in this subsection can be stated as follows.
\begin{lemma}  \label{lem:2.8}
Fix $s>4+\frac{d}{2}$ and $d\geq 1$.  Assume that the conditions in Assumption \ref{assume} hold, and $\textbf{y}_0=(u_0,\gamma_0)\in L^r(\Omega; \mathbb{H}^s(\mathbb{T}^d))$ is a $\mathcal {F}_0$-measurable initial random variable for some $r>2$. Then the Cauchy problem \eqref{2.15} has a unique local smooth maximal pathwise solution $ \textbf{y}_{R}=(u_{R},\gamma _{R})$ in the sense of Definition \ref{def1}.
\end{lemma}

\begin{proof}[\emph{\textbf{Proof.}}]
The proof of Lemma \ref{lem:2.8} will be divided into two steps.

{\textsf{Step 1 (Existence of global pathwise solution)}}: Let $\{\textbf{y}_{R,\epsilon}\}_{0<\epsilon < 1}$ be a sequence of strong solutions for the SDEs \eqref{2.16} in Hilbert space $\mathbb{H}^s(\mathbb{T}^d)$ with respect to the stochastic basis $\mathbb{S}$ fixed in advance. Set
$$
\nu_{R,\epsilon_1,\epsilon_2}(E)=\mathbb{P}\{(\textbf{y}_{R,\epsilon_1},\textbf{y}_{R,\epsilon_2})\in E\}, ~~~ \pi_{R,\epsilon_1,\epsilon_2}(E')=\mathbb{P}\{(\textbf{y}_{R,\epsilon_1},\textbf{y}_{R,\epsilon_2},\mathcal {W})\in E'\},
$$
for any $E\in \mathcal {B}(\mathcal {X} _{\textbf{y}}^s\times\mathcal {X} _{\textbf{y}}^s)$, $E'\in \mathcal {B}(\mathcal {X} _{\textbf{y}}^s\times\mathcal {X} _{\textbf{y}}^s\times \mathcal {X} _\mathcal {W} )$, where $\mathcal {X} ^s _{\textbf{y}} = \mathcal {C}([0,T];\mathbb{H}^s(\mathbb{T}^d)) $ and $\mathcal {X} _\mathcal {W}= \mathcal {C}([0,T];\mathfrak{U}_1\times \mathfrak{U}_1)$.

With only minor modifications to the arguments in Lemma \ref{lem:2.3}, one can easily prove that  the collection $\{\pi_{R,\epsilon_1,\epsilon_2}\}_{0<\epsilon_1,\epsilon_2<1}$ is tight and hence weakly compact. By using the Prokhorov Theorem, there exist two   subsequences of $\{\epsilon_{1}\},\{\epsilon_{2}\}$, denoted by $\{\epsilon_{1,k}\},\{\epsilon_{2,k}\}$ respectively, converging to $0$ as  $k\rightarrow\infty$, such that $
\pi_{R,\epsilon_{1,k},\epsilon_{2,k}}\rightarrow \pi_R$ weakly in $\mathcal {P}_r(\mathcal {X} _{\textbf{y}}^s\times\mathcal {X} _{\textbf{y}}^s\times \mathcal {X} _\mathcal {W})$, as $ k\rightarrow\infty $.
Furthermore, it follows from the Skorokhod Representing Theorem that one can choose a new probability space $(\widetilde{\Omega},\widetilde{\mathcal {F}},\widetilde{\mathbb{P}})$, on which a sequence of random elements $(\widetilde{\textbf{y}}_{R,\epsilon_{1,k}},\widetilde{\textbf{y}}_{R,\epsilon_{2,k}},\widetilde{W}_k)$ are defined such that
$$
\widetilde{\mathbb{P}}\circ (\widetilde{\textbf{y}}_{R,\epsilon_{1,k}},\widetilde{\textbf{y}}_{R,\epsilon_{2,k}},\widetilde{W}_k)^{-1}
=\pi_{R,\epsilon_{1,k},\epsilon_{2,k}} \quad \mbox{$\widetilde{\mathbb{P}}$-a.s.,} \quad \textrm{as}~~ k\rightarrow\infty,
$$
and
$$
(\widetilde{\textbf{y}}_{R,\epsilon_{1,k}},\widetilde{\textbf{y}}_{R,\epsilon_{2,k}},\widetilde{W}_k)\rightarrow (\widetilde{\textbf{y}}_R,\widetilde{\textbf{y}}_R^\sharp,\widetilde{W})\quad \textrm{in}\quad\mathcal {X} _{\textbf{y}}^s\times\mathcal {X} _{\textbf{y}}^s\times \mathcal {X} _\mathcal {W}\quad \mbox{$\widetilde{\mathbb{P}}$-a.s.,} \quad \textrm{as}~~ k\rightarrow\infty,
$$
with $\widetilde{\mathbb{P}}\circ (\widetilde{\textbf{y}}_R,\widetilde{\textbf{y}}_R^\sharp,\widetilde{W})^{-1}
=\pi_R$.  Especially, we have obtained
$$
\nu_{R,\epsilon_{1,k},\epsilon_{2,k}}=\widetilde{\mathbb{P}} \circ (\widetilde{\textbf{y}}_{R,\epsilon_{1,k}},\widetilde{\textbf{y}}_{R,\epsilon_{1,k}})^{-1}\rightarrow \nu_R\quad \textrm{weakly in} \quad \mathcal {P}_r(\mathcal {X} _{\textbf{y}}^s\times\mathcal {X} _{\textbf{y}}^s ),\quad k\rightarrow\infty,
$$
with
$
\widetilde{\mathbb{P}}\circ (\widetilde{\textbf{y}}_R,\widetilde{\textbf{y}}_R^\sharp)^{-1}
=\nu_R$.
Following the argument at the beginning of Subsection 2.3, we infer that $\widetilde{\textbf{y}}_R$ and $\widetilde{\textbf{y}}_R^\sharp$ are both global martingale solutions to SDEs \eqref{2.14} related to the stochastic basis $\widetilde{\mathbb{S}} =(\widetilde{\Omega} ,\widetilde{\mathcal {F}}, \widetilde{\mathcal {F}}_{t},\widetilde{\mathbb{P}})$, where $\widetilde{\mathcal {F}}_{t}=\sigma\{\widetilde{\textbf{y}}_R(r),\widetilde{\textbf{y}}_R^\sharp(r),\widetilde{W} (r)\}_{r \leq t}$. Since $\widetilde{\textbf{y}}_R(0)=\widetilde{\textbf{y}}_R^\sharp(0)$, it follows from the pathwise uniqueness (cf. Lemma \ref{lem:2.6}) again that
$$
\nu_R \{(\widetilde{\textbf{y}}_R,\widetilde{\textbf{y}}_R^\sharp)\in \mathcal {X} ^s _{\textbf{y}}\times\mathcal {X} ^s _{\textbf{y}};~\widetilde{\textbf{y}}_R=\widetilde{\textbf{y}}_R^\sharp\}=1. $$
Therefore, one can conclude from the Gy\"{o}ngy-Krylov Lemma that, the sequence $\{\textbf{y}_{R,\epsilon}\}_{0<\epsilon<1}$ defined on the original probability space $(\Omega,\mathcal {F},\mathbb{P})$ converges to an element $\textbf{y}_{R}$ almost surely, that is,
$\textbf{y}_{R,\epsilon}\rightarrow\textbf{y}_{R}$  as $\epsilon\rightarrow 0$, $\mathbb{P}$-a.s.,
in the strong topology of $\mathcal {X} ^s _{\textbf{y}}$. This convergence combined with \eqref{2.38} imply that the limit $\textbf{y}_{R}$ is a global pathwise solution to the SMEP2 with cut-off functions.

{\textsf{Step 2 (Construction of local pathwise solution)}}: The goal will be achieved by decomposing the random initial data $\textbf{y}_0=(u_0,\gamma_0)$ and using the result in \textsf{Step 1}.

\underline{Case 1}. Assume that the random initial data $\textbf{y}_0 (\omega) \in \mathbb{H}^s(\mathbb{T}^d)$ is bounded by some positive deterministic constant, i.e., there exists a real number $l>0$ such that
\begin{equation}\label{2.59}
\begin{split}
\|\textbf{y}_0(\omega)\|_{\mathbb{H}^s}\leq l,\quad \forall\omega\in \Omega.
\end{split}
\end{equation}
Let $c>0$ be the embedding constant from $\mathbb{H}^s(\mathbb{T}^d)$ into $\mathbb{W}^{1,\infty}(\mathbb{T}^d)$, for all $s>1+\frac{d}{2}$. Then we get from \eqref{2.59} that
\begin{equation}\label{2.60}
\begin{split}
\|\textbf{y}_0\|_{\mathbb{W}^{1,\infty}}\leq c\|\textbf{y}_0\|_{\mathbb{H}^s}\leq cl.
\end{split}
\end{equation}
Considering the stopping times $
\mathbbm{t}_{m}\triangleq  \{t>0; ~~ \|\textbf{y}(t) \|_{ \mathbb{H}^s }\geq m  \}$, $\forall m \in \mathbb{R}^+$.
If $m>cl$, then it follows from \eqref{2.60} that $\mathbbm{t}_{m}>0$ $\mathbb{P}$-almost surely, and hence $
\sup_{t\in [0,\mathbbm{t}_{m}]}\|\textbf{y}(t) \|_{ \mathbb{W}^{1,\infty} }\leq m$,
which implies that $\varpi_R(\|\textbf{y}\|_{\mathbb{W}^{1,\infty}} )=1$ for any $R\geq m$. As a result, the pair $(\textbf{y}|_{[0,\mathbbm{t}_{m}]},\mathbbm{t}_{m})$ is a unique local pathwise solution to the SMEP2 \eqref{2.5}.

\underline{Case 2}. For general data $\textbf{y}_0\in L^2(\Omega; \mathbb{H}^s(\mathbb{T}^d))$, there holds $\|\textbf{y}_0\|_{\mathbb{H}^s}<\infty$ $\mathbb{P}$-almost surely. In this case one can make the decomposition
$$
\textbf{y}_0(\omega,x)=\sum_{n\in \mathbb{N}^+}  \textbf{y}_0(\omega,x) \textbf{1}_{\{n-1\leq \|\textbf{y}_0\|_{\mathbb{H}^s} <n\}}(\omega)\triangleq\sum_{n\in \mathbb{N}^+} \textbf{y}_0(\omega,x) \textbf{1}_{\Omega_n}(\omega),\quad \mathbb{P}\textrm{-a.s.}
$$
For each $n\geq 1$, setting $
\textbf{y}_{0,n}(\omega,x)\triangleq \textbf{y}_0(\omega,x) \textbf{1}_{\{n-1\leq \|\textbf{y}_0\|_{\mathbb{H}^s} <n\}}(\omega)$.
It follows that $\Omega_n\cap \Omega_{n'}=\emptyset$ when $n\neq n'$, $\cup _{n\geq 1}\Omega_n$ is a set of full
measure, and the sequence $\{\textbf{y}_{0,n}(\omega)\}\in \mathbb{H}^s(\mathbb{T}^d)$ is uniformly bounded. By replacing the initial data $\textbf{y}_{0}$ with $\textbf{y}_{0,n}$ in the truncated system \eqref{2.14}, one can  concludes from  \textsf{Step 1} that the system admits a unique global pathwise solution $
\textbf{y}_{n}(\cdot) \in C([0,\infty);\mathbb{H}^s(\mathbb{T}^d))$, for any fixed $R>0 $.
Thereby, by considering the stopping time $\mathbbm{t}_{m}$ with $m>cn$, one get that the solution limited on $[0,\mathbbm{t}_{m}]$, denoted by $(\widetilde{\textbf{y}}_{n},\tilde{\mathbbm{t}}_{n})$, provides a local pathwise solution to system \eqref{2.5} with initial data $\textbf{y}_{0,n}$.

Now we define a stochastic process $\textbf{y}(t)$ by piecing together these solutions, i.e.,
$$
\textbf{y}(\omega,t,x)\triangleq \sum_{n\in \mathbb{N}^+}\widetilde{\textbf{y}}_{n}(\omega,t,x)\textbf{1}_{\Omega_n}(\omega), ~~~\mathbbm{t}\triangleq \sum_{n\in \mathbb{N}^+}\tilde{\mathbbm{t}}_{n}\textbf{1}_{\Omega_n}(\omega) .
$$
As $\mathbb{E}(\|\textbf{y}_0\|^2_{\mathbb{H}^s})<\infty$, one infer from the uniform bound for $\widetilde{\textbf{y}}_{n}$ (cf. Lemma \ref{2.2}) that $\textbf{y}(\cdot \wedge \mathbbm{t})\in L^2(\Omega;C([0,\infty); \mathbb{H}^s(\mathbb{T}^d)))$. Moreover, since $\widetilde{\textbf{y}}_{n}$ is a local solution to \eqref{2.5} with initial data $\textbf{y}_{0,n}$, we deduce that
\begin{equation*}
\begin{split}
  \textbf{y}(r\wedge \mathbbm{t})  =&\sum_{n\geq 1}\left(\textbf{y}_{0,n}-\int_0^{r\wedge \textbf{1}_{\Omega_n} \tilde{\mathbbm{t}}_n}B(\textbf{1}_{\Omega_n}\widetilde{\textbf{y}}_{n},\textbf{1}_{\Omega_n}\widetilde{\textbf{y}}_{n})\mathrm{d} r -\int_0^{r\wedge \textbf{1}_{\Omega_n}\tilde{\mathbbm{t}}_n}F(\textbf{1}_{\Omega_n}\widetilde{\textbf{y}}_{n})\mathrm{d} r\right.\\
  &\left.+\int_0^{r\wedge \textbf{1}_{\Omega_n}\tilde{\mathbbm{t}}_n}G(r,1_{\Omega_n}\widetilde{\textbf{y}}_{n}) \mathrm{d}\mathcal {W}\right)\\
  =&\textbf{y}_{0}-\int_0^{r\wedge \mathbbm{t}}B(\textbf{y},\textbf{y})\mathrm{d} r-\int_0^{r\wedge \mathbbm{t}}F(\textbf{y})\mathrm{d} r+\int_0^{r\wedge \mathbbm{t}}G(r,\textbf{y}) \mathrm{d}\mathcal {W}.
\end{split}
\end{equation*}
Therefore, the pair $(\textbf{y},\mathbbm{t})$ is a unique local pathwise solution to (1.5). By using a standard argument (cf. \cite{hofmanova2013degenerate,41}), we can extend the solution  $(\textbf{y},\mathbbm{t})$  to a maximal time of existence.

The proof of Lemma \ref{lem:2.8} is now completed.
\end{proof}

\subsection{Proof of Theorem \ref{th1}(1)}
In this subsection, we are going to prove the existence of local pathwise solutions to the SMEP2 \eqref{1.8} in the sharp case of $s>1+\frac{d}{2}$, $d\geq 1$. To this end, we shall apply a stability density argument. Precisely, we consider the following SMEP2 with regularized initial data:
\begin{equation}
\left\{
\begin{aligned}\label{2.61}
&\mathrm{d}\textbf{y}_{j}+ B(\textbf{y}_{j},\textbf{y}_{j})\mathrm{d} t+F(\textbf{y}_{j})\mathrm{d} t=G(t,\textbf{y}_{j}) \mathrm{d}\mathcal {W},\\
&\textbf{y}_{j}(0)=J_{1/j}\textbf{y}_0,\quad j\in \mathbb{N}^+,
\end{aligned}
\right.
\end{equation}
where $\textbf{y}_{j}=(u_j,\gamma_j)^T$ and  $J_{1/j}$ stands for the standard Friedrichs mollifier.

By Lemma \ref{2.8}, for each $j\geq1$, the Eq.\eqref{2.61} admits a unique local smooth strong solution $(\textbf{y}_{j},\mathbbm{t}_j)$ in the sense of Definition \ref{def1}. In the following, we shall show that $\{\textbf{y}_{j}\}_{j\geq 1}$ is a Cauchy sequence in the strong topology of $ \mathcal {C}([0,\mathbbm{t}^*];\mathbb{H}^s(\mathbb{T}^d))$
for some $\mathbbm{t}^*>0$ $\mathbb{P}$-almost surely. In view of the decomposition method as that in Lemma \ref{2.8}, one can first consider uniformly bounded data, i.e., $
\|\textbf{y}_0(\omega)\|_{\mathbb{H}^s}\leq M$,
for some deterministic $M>0$ independent of $j$.

\begin{lemma}\label{lem:2.9}
Let $T>0$, $s>1+\frac{d}{2}$, $d\geq 1$, and $ (\textbf{y}_j,\mathbbm{t}_j) _{j\geq1}$ be a sequence of local strong pathwise solutions to the system \eqref{2.61} related to the random variables $(J_{1/j}\textbf{y}_0)_{j\geq1}$. For each $j,k\in \mathbb{N}^+$, define the existing times $
\mathbbm{t}_{j,k}^T\triangleq \mathbbm{t}_{j}^T\wedge \mathbbm{t}_{k}^T$, where
$$
\mathbbm{t}_{j}^T \triangleq T\wedge \inf\left\{t>0;~~\|\textbf{y}_{j}(t)\|_{\mathbb{H}^s}^2\geq \|J_{1/j}\textbf{y}_0 \|_{\mathbb{H}^s}^2+3\right\}.
$$
Then we have
\begin{equation}\label{2.62}
\begin{split}
 \mathbb{E}\sup_{r\in [0, \mathbbm{t}_{j,k}^T ]}\|\textbf{y}_{j} -\textbf{y}_{k}\|_{\mathbb{H}^s}^2 \leq C \left(\mathbb{E} \|J_{1/j}\textbf{y}_0-J_{1/k}\textbf{y}_0\|_{\mathbb{H}^s}^2+  \mathbb{E} \sup_{r\in [0, \mathbbm{t}_{j,k}^T ]}(\|u_{j} -u_{k} \|_{H^{s-1}}^2\|\textbf{y}_{j}\|_{\mathbb{H}^{s+1}}^2 ) \right).
\end{split}
\end{equation}
\end{lemma}

\begin{proof}[\emph{\textbf{Proof.}}]
For each $j,k\in \mathbb{N}^+$, denoting $\textbf{y}_{j,k}=\textbf{y}_{j}-\textbf{y}_{k}$ and $\textbf{y}_{j,k}=(u_{j,k},\gamma _{j,k})$, then it follows from \eqref{2.61} that $\Lambda^s \textbf{y}_{j,k}$ satisfies the following system:
\begin{equation}\label{2.63}
\begin{split}
&\mathrm{d}\Lambda^s\textbf{y}_{j,k}+ \Lambda^s(B(\textbf{y}_{j},\textbf{y}_{j,k}) + B(\textbf{y}_{j,k},\textbf{y}_{j}))\mathrm{d} t+\Lambda^s(F(\textbf{y}_{j})-F(\textbf{y}_{k}))\mathrm{d} t\\
&\quad =\sum_{l\geq 1}\Lambda^s(G_l(t,\textbf{y}_{j})-G_l(t,\textbf{y}_{k})) \mathrm{d}\beta ^l_t ,
\end{split}
\end{equation}
 where $
\textbf{y}_{j,k}(0)= J_{1/j}\textbf{y}_0-J_{1/k}\textbf{y}_0$, $G_l(r,\cdot)=G(r,\cdot)e_l$ and $\{e_l\}$ is an orthogonal basis in $\mathfrak{A}$.   Applying the Ito formula to $\|\textbf{y}_{j,k}\|_{\mathbb{H}^s}^2$, one find
\begin{equation}\label{2.64}
\begin{split}
\|\textbf{y}_{j,k}(t\wedge\mathbbm{t}_{j,k}^T )\|_{\mathbb{H}^s}^2 =& \|\textbf{y}_{j,k}(0)\|_{\mathbb{H}^s}^2+ 2 \int_0^{t\wedge\mathbbm{t}_{j,k}^T }\left(\Lambda^s\textbf{y}_{j,k}(r), \Lambda^s(B(\textbf{y}_{j},\textbf{y}_{j,k}) + B(\textbf{y}_{j,k},\textbf{y}_{j}))\right) _{\mathbb{L}^2}\mathrm{d} r\\
&+ 2 \int_0^{t\wedge\mathbbm{t}_{j,k}^T }\left(\Lambda^s\textbf{y}_{j,k}(r), \Lambda^s(F(\textbf{y}_{j})-F(\textbf{y}_{k})\right) _{\mathbb{L}^2}\mathrm{d} r\\
&+\sum_{l\geq 1}\int_0^{t\wedge\mathbbm{t}_{j,k}^T }\|\Lambda^s(G_l(r,\textbf{y}_{j})-G_l(r,\textbf{y}_{k})) \|_{\mathbb{L}^2 }^2\mathrm{d} r\\
&+ 2 \sum_{l\geq 1}\int_0^{t\wedge\mathbbm{t}_{j,k}^T }\left(\Lambda^s\textbf{y}_{j,k}(r), \Lambda^s(G_l(t,\textbf{y}_{j})-G_l(t,\textbf{y}_{k})) \right) _{\mathbb{L}^2}\mathrm{d}\beta ^l_t\\
=& \|\textbf{y}_{j,k}(0)\|_{\mathbb{H}^s}^2+ U_1(t)+ U_2(t)+ U_3(t)+ U_4(t).
\end{split}
\end{equation}
Let us estimate the terms $U_i(t)$, $i=1,2,3,4$ in \eqref{2.64} one by one. For the term $B(\textbf{y}_{j},\textbf{y}_{j,k})$, by using the Moser-type estimate and the fact of $H^{s-1}(\mathbb{T}^d)\hookrightarrow L^\infty (\mathbb{T}^d)$, for all $s>\frac{d}{2}+1$, one can get
\begin{equation}\label{2.65}
\begin{split}
(\Lambda^s\textbf{y}_{j,k}, \Lambda^sB(\textbf{y}_{j,k},\textbf{y}_{j})) _{\mathbb{L}^2} \leq& C\Big( \| \nabla u_{j}\|_{L^\infty}\|u_{j,k}\|_{H^s}^2+\|u_{j,k} \|_{L^\infty}\|u_{j,k}\|_{H^s}\| \nabla u_{j}\|_{H^s}\\
&+   \|\gamma_{j,k}\|_{H^s} \|u_{j,k}\|_{H^s}\| \nabla \gamma_{j}\|_{L^\infty}+\|\gamma_{j,k}\|_{H^s} \|u_{j,k} \|_{L^\infty}\| \nabla \gamma_{j}\|_{H^s} \Big)        \\
 \leq& C\|u_{j,k} \|_{H^{s-1}}^2(\| u_{j}\|_{H^{s+1}}^2+\| \gamma_{j}\|_{H^{s+1}}^2)\\
&+ C \Big((\| u_{j}\|_{H^{s}}+1)\|u_{j,k}\|_{H^s}^2+ \|\gamma_{j,k}\|_{H^s}^2 (\|  \gamma_{j}\|_{H^{s}}^2+1)\Big).
\end{split}
\end{equation}
For the term $B(\textbf{y}_{j,k},\textbf{y}_{j})$, by commutating the operator $\Lambda^s$ with $u_j$ and integrating by parts, we find
\begin{equation*}
\begin{split}
(\Lambda^s\textbf{y}_{j,k}, \Lambda^s B(\textbf{y}_{j},\textbf{y}_{j,k})_{\mathbb{L}^2} =& (\Lambda^su_{j,k}, [\Lambda^s, u_j\cdot \nabla] u_{j,k})_{L^2}-\frac{1}{2}(|\Lambda^su_{j,k}|^2, \textrm{div} u_j)_{L^2} \\
&+(\Lambda^s\gamma_{j,k}, [\Lambda^s ,u_j\cdot \nabla] \gamma_{j,k})_{L^2}-\frac{1}{2}(|\Lambda^s\gamma_{j,k}|^2, \textrm{div} u_j)_{L^2}\\
\triangleq& D_1+D_2+D_3+D_4.
\end{split}
\end{equation*}
By applying the Cauchy-Schwartz inequality and the commutator estimate Lemma \ref{lem:commutator} for the first and third terms on the R.H.S., one get
\begin{equation*}
\begin{split}
|D_1+D_3|\leq & C\|u_{j,k}\|_{H^s}(\|\Lambda^s u_j\|_{L^2}\|\nabla u_{j,k}\|_{L^\infty}+\|\nabla u_j\|_{L^\infty}\|\Lambda^{s-1} \nabla u_{j,k} \|_{L^2})\\
&+C\|\gamma_{j,k}\|_{H^s}(\|\Lambda^s u_j\|_{L^2}\|\nabla \gamma_{j,k}\|_{L^\infty}+\|\nabla u_j\|_{L^\infty}\|\Lambda^{s-1} \nabla \gamma_{j,k} \|_{L^2})\\
\leq & C \|u_j\|_{H^s}(\|u_{j,k}\|_{H^s}^2+\|\gamma_{j,k}\|_{H^s}^2).
\end{split}
\end{equation*}
The other terms  can be estimated as
\begin{equation*}
\begin{split}
|D_2+D_4|\leq   C \|u_j\|_{H^s}(\|u_{j,k}\|_{H^s}^2+\|\gamma_{j,k}\|_{H^s}^2).
\end{split}
\end{equation*}
Putting the estimates for $\{D_i\}_{i=1}^4$ together, we get
\begin{equation*}
\begin{split}
(\Lambda^s\textbf{y}_{j,k}, \Lambda^s B(\textbf{y}_{j},\textbf{y}_{j,k})_{\mathbb{L}^2} \leq   C \|u_j\|_{H^s}(\|u_{j,k}\|_{H^s}^2+\|\gamma_{j,k}\|_{H^s}^2),
\end{split}
\end{equation*}
which combined with \eqref{2.65} and the definition of the stopping time $\mathbbm{t}_{j,k}^T$ yield that
\begin{equation}
\begin{split}
U_1(t)\leq& C\int_0^{t\wedge\mathbbm{t}_{j,k}^T }    \Big(\|u_{j,k} \|_{H^{s-1}}^2(\| u_{j}\|_{H^{s+1}}^2+\| \gamma_{j}\|_{H^{s+1}}^2)\\
&+  \|u_{j,k}\|_{H^s}^2(\| u_{j}\|_{H^{s}}+1)+ \|\gamma_{j,k}\|_{H^s}^2 (\|  \gamma_{j}\|_{H^{s}}^2+\| u_{j}\|_{H^{s}}+1) \Big)\mathrm{d} r\\
\leq& C\int_0^{t\wedge\mathbbm{t}_{j,k}^T }    \Big(\|u_{j,k} \|_{H^{s-1}}^2\| \textbf{y}_{j}\|_{\mathbb{H}^{s+1}}^2 +  \|\textbf{y}_{j,k}\|_{\mathbb{H}^s}^2  \Big)\mathrm{d} r,
\end{split}
\end{equation}
where the second inequality used the boundedness of $J_\epsilon$ in $H^s(\mathbb{T}^d)$. For $U_2(t)$, it follows from the Lemma \ref{2.6} and Young inequality that
\begin{equation}
\begin{split}
U_2(t) &\leq C\int_0^{t\wedge\mathbbm{t}_{j,k}^T }\|\textbf{y}_{j,k}(r)\|_{\mathbb{H}^s}( \|\textbf{y}_j\|_{\mathbb{H}^s}^2+ \|\textbf{y}_k\|_{\mathbb{H}^s}^2+1) \|\textbf{y}_j-\textbf{y}_k\|_{\mathbb{H}^s}\mathrm{d} r \\
&\leq C\int_0^{t\wedge\mathbbm{t}_{j,k}^T }\|\textbf{y}_{j,k}(r)\|_{\mathbb{H}^s}^2\mathrm{d} r.
\end{split}
\end{equation}
For $U_4(t)$, by using the BDG inequality and condition \eqref{a2}, we get
\begin{equation}
\begin{split}
&\mathbb{E}\sup_{r\in [0,t\wedge\mathbbm{t}_{j,k}^T ]}|U_4(r)|  \leq \mathbb{E}\left(\int_0^{t\wedge\mathbbm{t}_{j,k}^T }\tilde{\mu}^2 (r) \| \textbf{y}_{j,k}(r)\|_{\mathbb{H}^s}^2 \tilde{\chi}^2(\|\textbf{y}_{j}\|_{\mathbb{W}^{1,\infty}}
+\|\textbf{y}_{k}\|_{\mathbb{W}^{1,\infty}}) \|\textbf{y}_{j}-\textbf{y}_{k}\|_{\mathbb{H}^{s}}^2\mathrm{d}r\right)^{\frac{1}{2}}\\
&\quad \leq \mathbb{E}\left[\sup_{r\in [0,t\wedge\mathbbm{t}_{j,k}^T ]}\| \textbf{y}_{j,k}(r)\|_{\mathbb{H}^s} \left(\int_0^{t\wedge\mathbbm{t}_{j,k}^T }\tilde{\mu}^2 (r)\tilde{\chi}^2(\|\textbf{y}_{j}\|_{\mathbb{W}^{1,\infty}}
+\|\textbf{y}_{k}\|_{\mathbb{W}^{1,\infty}}) \| \textbf{y}_{j,k}(r)\|_{\mathbb{H}^s}^2 \mathrm{d}r\right)^{\frac{1}{2}}\right]\\
&\quad \leq \frac{1}{2}\mathbb{E} \sup_{r\in [0,t\wedge\mathbbm{t}_{j,k}^T ]}\| \textbf{y}_{j,k}(r)\|_{\mathbb{H}^s}^2+ C\tilde{\chi}^2(C(1+M^2))\mathbb{E} \int_0^{t\wedge\mathbbm{t}_{j,k}^T }\tilde{\mu}^2 (r) \| \textbf{y}_{j,k}(r)\|_{\mathbb{H}^s}^2 \mathrm{d}r,
\end{split}
\end{equation}
where the last inequality used the facts that $\tilde{\chi}(\cdot)$ is nondecreasing, the Sobolev embedding $H^s(\mathbb{T}^d)\subset W^{1,\infty}(\mathbb{T}^d)$ for $s>\frac{d}{2}+1$ as well as the uniform bound
$$
\sup_{j\geq 1}\|\textbf{y}_{j}(r)\|_{\mathbb{H}^s}
\leq \frac{1}{2}+\frac{1}{2}\sup_{j \geq 1}\|\textbf{y}_{j}(r)\|_{\mathbb{H}^s}^2\leq \frac{1}{2}(\|J_{1/j}\textbf{y}_0 \|_{\mathbb{H}^s}^2+2)\leq C(M^2+1),
$$
for all $r\in [0,t\wedge\mathbbm{t}_{j,k}^T ]$ with some constant $C>0$ independent of $j$.

For $U_3(t)$, one can use the condition \eqref{a2} to obtain
\begin{equation}
\begin{split}
|U_3(t)| \leq& \mathbb{E} \int_0^{t\wedge\mathbbm{t}_{j,k}^T }\tilde{\mu}^2 (r)  \tilde{\chi}^2(\|\textbf{y}_{j}\|_{\mathbb{W}^{1,\infty}}
+\|\textbf{y}_{k}\|_{\mathbb{W}^{1,\infty}}) \|\textbf{y}_{j}-\textbf{y}_{k}\|_{\mathbb{H}^{s}}^2\mathrm{d}r \\
 \leq& C\tilde{\chi}^2(C(1+M^2))\mathbb{E} \int_0^{t\wedge\mathbbm{t}_{j,k}^T }\tilde{\mu}^2 (r) \| \textbf{y}_{j,k}(r)\|_{\mathbb{H}^s}^2 \mathrm{d}r.
\end{split}
\end{equation}

Set $
G(t)\triangleq\mathbb{E}\sup_{r\in [0,t\wedge\mathbbm{t}_{j,k}^T ]}\|\textbf{y}_{j,k}(r)\|_{\mathbb{H}^s}^2.$
After plugging the estimates for $U_i$ into \eqref{2.64}, we get
\begin{equation*}
\begin{split}
G(t)\leq &\mathbb{E} \|\textbf{y}_{j,k}(0)\|_{\mathbb{H}^s}^2+C\int_0^{t}\mathbb{E} \sup_{\varsigma\in [0,r\wedge\mathbbm{t}_{j,k}^T ]}(\|u_{j,k} \|_{H^{s-1}}^2\| \textbf{y}_{j}\|_{\mathbb{H}^{s+1}}^2 ) \mathrm{d} r +C\int_0^{t}(1+\tilde{\mu}^2 (r)) G(r)\mathrm{d} r.
\end{split}
\end{equation*}
An application of the Gronwall Lemma to above inequality leads to
\begin{equation*}
\begin{split}
\mathbb{E}\sup_{r\in [0, \mathbbm{t}_{j,k}^T ]}\|\textbf{y}_{j,k}(r)\|_{\mathbb{H}^s}^2  \leq e^{C\int_0^{T}(1+\tilde{\mu}^2 (r))\mathrm{d} r} \left(\mathbb{E} \|\textbf{y}_{j,k}(0)\|_{\mathbb{H}^s}^2+C \mathbb{E} \sup_{\varsigma\in [0, \mathbbm{t}_{j,k}^T ]}(\|u_{j,k} \|_{H^{s-1}}^2\| \textbf{y}_{j}\|_{\mathbb{H}^{s+1}}^2 ) \right).
\end{split}
\end{equation*}
This completes the proof of Lemma \ref{lem:2.9}.
\end{proof}

\begin{lemma}\label{lem:2.10}
Under the assumptions of Lemma \ref{lem:2.9}, we have
\begin{equation}\label{2.70}
\begin{split}
 &\mathbb{E}\sup_{\varsigma\in [0, \mathbbm{t}_{j,k}^T]}\left(\|u_{j,k}  \|_{H^{s-1}}^2\|\textbf{y}_{j}\|_{\mathbb{H}^{s+1}}^2\right)\\
  &\quad \leq C\left(\mathbb{E}(\|J_{1/j}u_0-  J_{1/k}u_0 \|_{H^{s-1}}^2\|J_{1/j}\textbf{y}_0\|_{\mathbb{H}^{s+1}}^2) + \mathbb{E}\sup_{\varsigma\in [0, \mathbbm{t}_{j,k}^T]}\|u_{j,k}(\varsigma)  \|_{H^{s-1}}^2\right).
\end{split}
\end{equation}
\end{lemma}

\begin{proof}[\emph{\textbf{Proof.}}]
The proof of Lemma \ref{lem:2.10} is based on  a priori estimate  for $\|u_{j,k}  \|_{H^{s-1}}^2\|\textbf{y}_{j}\|_{\mathbb{H}^{s+1}}^2 $.  To this end, we get by applying the It\^{o} product rule that
\begin{equation}\label{2.71}
\begin{split}
&\mathrm{d}\left( \|u_{j,k}  \|_{H^{s-1}}^2\|\textbf{y}_{j}\|_{\mathbb{H}^{s+1}}^2\right)\\
&\quad =\|\textbf{y}_{j}\|_{\mathbb{H}^{s+1}}^2\mathrm{d}\|u_{j,k}  \|_{H^{s-1}}^2+\|u_{j,k}  \|_{H^{s-1}}^2\mathrm{d}\|\textbf{y}_{j}\|_{\mathbb{H}^{s+1}}^2+\mathrm{d} \|u_{j,k}  \|_{H^{s-1}}^2\mathrm{d}\|\textbf{y}_{j}\|_{\mathbb{H}^{s+1}}^2.
\end{split}
\end{equation}
From the first component of the system \eqref{2.61}, we have
\begin{equation*}
\begin{split}
\mathrm{d}\Lambda^{s-1}u_{j,k}=&- \Lambda^{s-1}( u_{j,k}\cdot \nabla u_j +u_k\cdot \nabla u_{j,k} )\mathrm{d} t- \Lambda^{s-1}(  \mathscr{L}_1(u_j)-\mathscr{L}_1(u_k) )\mathrm{d} t \\
& - \Lambda^{s-1} ( \mathscr{L}_2 (\gamma_j)-\mathscr{L}_2 (\gamma_k)  )\mathrm{d} t+\Lambda^{s-3}(g_1(t,m_j,\rho_j)-g_1(t,m_k,\rho_k))\mathrm{d}W_1.
\end{split}
\end{equation*}
Using the It\^{o} formula again, one infer that
\begin{equation}\label{2.72}
\begin{split}
 \mathrm{d}\|u_{j,k}\|_{H^{s-1}}^2&= - 2(\Lambda^{s-1}u_{j,k},\Lambda^{s-1}( u_{j,k}\cdot \nabla u_j +u_k\cdot \nabla u_{j,k} ))_{L^2} \mathrm{d} t\\
 &\quad- 2(\Lambda^{s-1}u_{j,k},\Lambda^{s-1}(  \mathscr{L}_1(u_j)-\mathscr{L}_1(u_k) ))_{L^2} \mathrm{d} t \\
&\quad - 2(\Lambda^{s-1}u_{j,k},\Lambda^{s-1} ( \mathscr{L}_2 (\gamma_j)-\mathscr{L}_2 (\gamma_k)  ))_{L^2} \mathrm{d} t\\
&\quad + \|\Lambda^{s-3}(g_1(t,m_j,\rho_j)-g_1(t,m_k,\rho_k))\|_{L_2(\mathfrak{A},L^2)}^2\mathrm{d} t\\
&\quad+2(\Lambda^{s-1}u_{j,k},\Lambda^{s-3}(g_1(t,m_j,\rho_j)-g_1(t,m_k,\rho_k))\mathrm{d}\mathcal {W}_1)_{L^2} \\
&=\sum_{i=1}^4 I_i(t) \mathrm{d} t+I_5(t)\mathrm{d}\mathcal {W}_1.
\end{split}
\end{equation}
For $\|\textbf{y}_{j}\|_{\mathbb{H}^{s+1}}^2$, we get by replacing $s$ with $s+1$ in \eqref{2.63} throughout that
\begin{equation}\label{2.73}
\begin{split}
 \mathrm{d}\|\textbf{y}_{j}\|_{H^{s+1}}^2=& -2(\Lambda^{s+1}\textbf{y}_{j}, \Lambda^{s+1}B(\textbf{y}_{j},\textbf{y}_{j}))_{\mathbb{L}^2}\mathrm{d} t-2(\Lambda^{s+1}\textbf{y}_{j}, \Lambda^{s+1}F(\textbf{y}_{j}))_{\mathbb{L}^2} \mathrm{d} t\\
 &+  \|G(t,\textbf{y}_{j})\|_{L_2(\mathfrak{A},\mathbb{H}^{s+1})}^2\mathrm{d} t+2(\Lambda^{s+1}\textbf{y}_{j}, \Lambda^{s+1} G(t,\textbf{y}_{j})\mathrm{d}\mathcal {W})_{\mathbb{L}^2}\\
 =&\sum_{i=1}^3 J_i(t) \mathrm{d} t+J_4(t)\mathrm{d}\mathcal {W}.
\end{split}
\end{equation}
After plugging the identities \eqref{2.72}-\eqref{2.73}  into \eqref{2.71}, we get
\begin{equation}\label{2.74}
\begin{split}
\mathrm{d}( \|u_{j,k}  \|_{H^{s-1}}^2\|\textbf{y}_{j}\|_{\mathbb{H}^{s+1}}^2)=& \left (\sum_{i=1}^4 \|\textbf{y}_{j}\|_{\mathbb{H}^{s+1}}^2I_i(t)+\sum_{i=1}^3 \|u_{j,k}  \|_{H^{s-1}}^2J_i(t)+\mathcal {K} \right) \mathrm{d} t\\
& +\|\textbf{y}_{j}\|_{\mathbb{H}^{s+1}}^2I_5(t)\mathrm{d}W_1 +\|u_{j,k}  \|_{H^{s-1}}^2J_4(t)\mathrm{d}\mathcal {W} .
\end{split}
\end{equation}
Here $\mathcal {K}$ is the term  arising from $\mathrm{d} \|u_{j,k}  \|_{H^{s-1}}^2\mathrm{d}\|\textbf{y}_{j}\|_{\mathbb{H}^{s+1}}^2$, and is given by
$$
\mathcal {K}=4\sum_{q\geq 1} (\Lambda^{s+1}u_{j}, \Lambda^{s-1} g_1(t,,m_j,\rho_j)e_q)_{L^2} (\Lambda^{s-1}u_{j,k},\Lambda^{s-3}(g_1(t,m_j,\rho_j)-g_1(t,m_k,\rho_k)) e_q)_{L^2}.
$$
It remains to estimates the terms on the R.H.S. of \eqref{2.74}.

\textsf{Estimate for $I_1$.} The estimate for $I_1$ will be classified according to the dimension $d\geq 1$.

\underline{Case of $d=1$.}
In this case, the unknown is a scalar quantity, we observe from the identity $u_j\cdot \nabla u_j=u_j\partial_x u_j=\frac{1}{2}\partial_x(u_j^2)$ that $4
I_1(t) =\frac{1}{2}(\Lambda^{s-1}u_{j,k},\Lambda^{s-1}\partial_x (u_{j,k} ( u_j +u_k )))_{L^2}$.
By commutating $\Lambda^{s-1}\partial_x$ with $u_j +u_k$, and using the Cauchy-Schwartz inequality, we get
\begin{equation*}
\begin{split}
 |I_1(t) |&\leq\frac{1}{2} \left|(\Lambda^{s-1}u_{j,k},[\Lambda^{s-1}\partial_x, u_j +u_k ]u_{j,k})_{L^2} \right|+\frac{1}{2}\left|(\partial_x(u_j +u_k) ,(\Lambda^{s-1}u_{j,k})^2)_{L^2}\right|\\
&\leq C\| u_{j,k}\|_{H^{s-1}}\left(\|[\Lambda^{s-1}\partial_x, u_j +u_k ]u_{j,k}\|_{L^2}  + \| u_j +u_k \|_{W^{1,\infty}}\| u_{j,k}\|_{H^{s-1}}\right)\\
 &\leq C(\| u_j\|_{H^s} +\|u_k  \|_{H^s}  )\| u_{j,k}\|_{H^{s-1}}^2,
\end{split}
\end{equation*}
where the third inequality used the commutator estimates (cf. \cite[Proposition 4.2]{taylor2003commutator})
and the last inequality used the Sobolev embedding $H^s(\mathbb{T})\subset W^{1,\infty}(\mathbb{T})$ for $s>\frac{3}{2}$.

The discussion for the cases of $d=2$ and $d\geq 3$ is more involved, we first get from the Moser-type estimates that
\begin{equation}\label{2.75}
\begin{split}
\|\textbf{y}_{j}\|_{\mathbb{H}^{s+1}}^2|I_1(t)|\leq& C \|\textbf{y}_{j}\|_{\mathbb{H}^{s+1}}^2\Big(\| u_{j,k}\|_{H^{s-1}}(\|u_{j,k}\|_{L^\infty}\|\nabla u_j\|_{H^{s-1}}+\|u_{j,k}\|_{H^{s-1}}\|\nabla u_j\|_{L^\infty}) \\
& + |(\Lambda^{s-1}u_{j,k},\Lambda^{s-1}( u_k\cdot \nabla u_{j,k} ))_{L^2}|\Big)\\
\leq& C \|\textbf{y}_{j}\|_{\mathbb{H}^{s+1}}^2\| u_{j,k}\|_{H^{s-1}}^2\| u_j\|_{H^{s}} \\
&+ C\|\textbf{y}_{j}\|_{\mathbb{H}^{s+1}}^2|(\Lambda^{s-1}u_{j,k},\Lambda^{s-1}( u_k\cdot \nabla u_{j,k} ))_{L^2}| .
\end{split}
\end{equation}
The main difficulty comes from the second term on the R.H.S. of \eqref{2.75}.

\underline{Case of $d=2$}.   It is clear that $H^{s}(\mathbb{T}^2)\subset W^{1,\infty}(\mathbb{T}^2)$ for $s>2$. If $2< s < 3$, then $H^{s-1}(\mathbb{T}^2)\subset W^{1,q}(\mathbb{T}^2)$, for some $q > 2$ such that $s-2=1-\frac{2}{q}$. Choosing $p>2$ satisfying $1=\frac{2}{p}+\frac{2}{q}$, then the following embedding  holds: $
 H^s(\mathbb{T}^2)\subset W^{s-1+\frac{2}{p},p}(\mathbb{T}^2)\subset W^{s-1,p}(\mathbb{T}^2)$.
By using Lemma \ref{lem:commutator} and integrating by parts, we have
\begin{equation*}
\begin{split}
&|(\Lambda^{s-1}u_{j,k},\Lambda^{s-1}( u_k\cdot \nabla u_{j,k} ))_{L^2}|\\
 &\quad \leq C \| u_{j,k}\|_{H^{s-1}} (\|[\Lambda^{s-1}, u_k\cdot \nabla] u_{j,k} \|_{L^2}+ \|\textrm{div} u_k  \|_{L^\infty} \| u_{j,k}\|_{H^{s-1}})\\
 &\quad \leq C \| u_{j,k}\|_{H^{s-1}} ( \| u_k\|_{W^{s-1,p}} \|u_{j,k}  \|_{W^{1,q}} +\| u_k\|_{W^{1,\infty}} \|u_{j,k} \|_{H^{s-1}} ) \\
 &\quad\leq C\|u_k \|_{H^s} \| u_{j,k}\|_{H^{s-1}} ^2.
\end{split}
\end{equation*}
If $s\geq 3$, then $H^{s-1}(\mathbb{T}^2)\subset \mathcal {C}^1(\mathbb{T}^2)\subset W^{1,q}(\mathbb{T}^2)$, for all $q\in (2,\infty)$. The result also holds ture.

\underline{Case of $d\geq 3$}.  By the Sobolev embeddings $
H^{s}(\mathbb{T}^d)\subset W^{1,\infty}(\mathbb{T}^d)$, $H^s(\mathbb{T}^d)\subset W^{s-1,\frac{2d}{d-2}}(\mathbb{T}^d)$, $H^{s-1}(\mathbb{T}^d)\subset W^{1,d}(\mathbb{T}^d)$,
and the commutator estimates, we have
\begin{equation*}
\begin{split}
&|(\Lambda^{s-1}u_{j,k},\Lambda^{s-1}( u_k\cdot \nabla u_{j,k} ))_{L^2}|\\
&\quad \leq |(\Lambda^{s-1}u_{j,k},[\Lambda^{s-1}, u_k\cdot \nabla] u_{j,k} ))_{L^2}|+\frac{1}{2}|(|\Lambda^{s-1}u_{j,k}|^2,\textrm{div} u_k  )_{L^2}|\\
&\quad \leq C\Big(\|u_k \|_{W^{s-1,\frac{2d}{d-2}}} \|  u_{j,k}\|_{W^{1,d}}\| u_{j,k}\|_{H^{s-1}} +\|  u_k\|_{W^{1,\infty}} \| u_{j,k}\|_{H^{s-1}}^2 + \|\textrm{div} u_k  \|_{L^\infty} \| u_{j,k}\|_{H^{s-1}}^2\Big)\\
&\quad \leq C\|u_k \|_{H^s} \| u_{j,k}\|_{H^{s-1}} ^2.
\end{split}
\end{equation*}
In summary, the term involving $I_1$ can be estimated by
\begin{equation}
\begin{split}
\|\textbf{y}_{j}\|_{\mathbb{H}^{s+1}}^2|I_1(t)| \leq  C \| u_{j,k}\|_{H^{s-1}} ^2\|\textbf{y}_{j}\|_{\mathbb{H}^{s+1}}^2 ( \| u_j\|_{H^{s}} +\|u_k \|_{H^s} ).
\end{split}
\end{equation}

\textsf{Estimate for $I_2$.} We get by using the Cauchy-Schwartz inequality that
\begin{equation*}
\begin{split}
|I_2(t)| \leq C \| u_{j,k}\|_{H^{s-1}} \|\mathscr{L}_1(u_j)-\mathscr{L}_1(u_k)\|_{H^{s-1}}.
\end{split}
\end{equation*}
To estimate $\|\mathscr{L}_1(u_j)-\mathscr{L}_1(u_k)\|_{H^{s-1}}$, we need the following Moser-type estimates (cf. Proposition 2.82 in \cite{bahouri2011fourier}): For any $s_1\leq \frac{d}{2}< s_2$ ($s_2\geq \frac{d}{2}$ if $r=1$), $s_1+s_2>0$, then $
\|fg\|_{B_{2,r}^{s_1}}\leq C \|f\|_{B_{2,r}^{s_1}}\|g\|_{B_{2,r}^{s_2}}$.
We divide the discussion into two cases.

\underline{Case of $\frac{d}{2}+1<s\leq\frac{d}{2}+2$.}  Since $s-2\leq \frac{d}{2}< s-1$, and $(s-2)+(s-1)=2s-3>0$, it follows from the last Moser estimates that
\begin{equation*}
\begin{split}
 &\|\mathscr{L}_1(u_j)-\mathscr{L}_1(u_k)\|_{H^{s-1}}\\
 &\quad \leq C \left\|  (|\nabla u_1|+|\nabla u_2|)|\nabla u_{j,k}|\textbf{I}_d+\nabla u_{j,k}\nabla u_j+\nabla u_k\nabla u_{j,k}+\nabla u_{j,k} \nabla u_j ^T \right.\\
&\quad\quad\left.+\nabla u_k\nabla u_{j,k}^T+\nabla u_{j,k}^T  \nabla u_k+\nabla u_j ^T  \nabla u_{j,k}-\textrm{div}u_{j,k} \nabla u_k- \textrm{div}u_j  \nabla u_{j,k} \right\|_{H^{s-2}} \\
&\quad\quad+\left\|u_j \textrm{div} u_{j,k}  +u_{j,k} \textrm{div}u_k  +u_{j,k}\cdot   \nabla u_j ^T+u_k\cdot \nabla u_{j,k}^T \right\|_{H^{s-3}}\\
 &\quad \leq C  \| \nabla u_{j,k}\|_{H^{s-2}} (\|\nabla u_j\|_{H^{s-1}}+\|\nabla u_k\|_{H^{s-1}})+C  \| \nabla  u_{j,k}\|_{H^{s-2}} \\
 &\quad\quad\times  (\|  u_j\|_{H^{s-1}}+\|  u_k\|_{H^{s-1}})+C\| u_{j,k}\|_{H^{s-1}} (\|\nabla u_j\|_{H^{s-2}}+\|\nabla u_k\|_{H^{s-2}})\\
 &\quad \leq C   \|  u_{j,k}\|_{H^{s-1}} (\| u_j\|_{H^{s}}+\|u_k\|_{H^{s}}),
\end{split}
\end{equation*}
which implies
\begin{equation*}
\begin{split}
\|\textbf{y}_{j}\|_{\mathbb{H}^{s+1}}^2|I_2(t)| \leq C \|\textbf{y}_{j}\|_{\mathbb{H}^{s+1}}^2\| u_{j,k}\|_{H^{s-1}}^2 (\| u_j\|_{H^{s}}+\|u_k\|_{H^{s}}).
\end{split}
\end{equation*}

\underline{Case of $ s >\frac{d}{2}+2$.} In this case, the Sobolev spaces  $H^{s-2}(\mathbb{T}^d)$ are  Banach algebras, and hence it follows from the embedding $H^{s}(\mathbb{T}^d)\subset H^{t}(\mathbb{T}^d)$ for $s>t$   that
\begin{equation*}
\begin{split}
 &\|\mathscr{L}_1(u_j)-\mathscr{L}_1(u_k)\|_{H^{s-1}}\\
 &\quad \leq C  \| \nabla u_{j,k}\|_{H^{s-2}} (\|\nabla u_j\|_{H^{s-2}}+\|\nabla u_k\|_{H^{s-2}})+C  \| \nabla  u_{j,k}\|_{H^{s-2}}   (\|  u_j\|_{H^{s-2}}\\
 &\quad\quad+\|  u_k\|_{H^{s-2}})+C\| u_{j,k}\|_{H^{s-2}} (\|\nabla u_j\|_{H^{s-2}}+\|\nabla u_k\|_{H^{s-2}})\\
 &\quad \leq C   \|  u_{j,k}\|_{H^{s-1}} (\| u_j\|_{H^{s}}+\|u_k\|_{H^{s}}),
\end{split}
\end{equation*}
which lead to the similar estimate.

\textsf{Estimate for $I_3$.} There holds
\begin{equation*}
\begin{split}
\|\textbf{y}_{j}\|_{\mathbb{H}^{s+1}}^2I_3 \leq&   C \|\textbf{y}_{j}\|_{\mathbb{H}^{s+1}}^2\| u_{j,k}\|_{H^{s-1}}^2      \Big(\|\gamma_{j,k}(\gamma_j+\gamma_k)\|_{H^{s-2}}+ \| \nabla\gamma_{j,k}(|\nabla\gamma_j|+|\nabla\gamma_k|)\|_{H^{s-2}}\\
&+\|(\nabla\gamma_{j,k})^T\nabla\gamma_j \|_{H^{s-2}}+\|(\nabla\gamma_k)^T\nabla\gamma_{j,k} \|_{H^{s-2}}\Big)\\
\leq&   C \|\textbf{y}_{j}\|_{\mathbb{H}^{s+1}}^2\| u_{j,k}\|_{H^{s-1}}^2 \| \gamma_{j,k}\|_{H^{s}}(\| \gamma_j\|_{H^{s}}+\|\gamma_k\|_{H^{s}})\\
\leq&   C \|\textbf{y}_{j}\|_{\mathbb{H}^{s+1}}^2\| u_{j,k}\|_{H^{s-1}}^2  (\| \gamma_j\|_{H^{s}}^2+\|\gamma_k\|_{H^{s}}^2).
\end{split}
\end{equation*}

\textsf{Estimate for $I_4$ and $I_5$.} In terms of the Assumption \ref{assume}, we have
\begin{equation*}
\begin{split}
\|\textbf{y}_{j}\|_{\mathbb{H}^{s+1}}^2I_4\leq& C \|\textbf{y}_{j}\|_{\mathbb{H}^{s+1}}^2 \|g_1(t,m_j)-g_1(t,m_k)\|_{L_2(\mathfrak{A},H^{s-3})}^2\\
\leq& C \tilde{\mu}_1^2 (t) \tilde{\chi}_1^2(\|u_j\|_{\mathbb{W}^{1,\infty}}+\|u_k\|_{\mathbb{W}^{1,\infty}})
\|\textbf{y}_{j}\|_{\mathbb{H}^{s+1}}^2
\|u_{j,k}\|_{\mathbb{H}^{s-1}}^2.
\end{split}
\end{equation*}
Using the BDG inequality and   Assumption \ref{assume}, we have for any stopping time $\mathbbm{t}$ that
\begin{equation}
\begin{split}
&\mathbb{E}\sup_{t\in [0,\mathbbm{t}]}\left|\int_0^t\|\textbf{y}_{j}\|_{\mathbb{H}^{s+1}}^2I_5(\varsigma)\mathrm{d}\mathcal {W}_1\right|\\
  &\quad \leq C\mathbb{E} \left [\sup_{\varsigma\in [0,\mathbbm{t}]}(\|\textbf{y}_{j}\|_{\mathbb{H}^{s+1}} \|u_{j,k}\| _{H^{s-1}})\right.\\
&\quad\quad \left.\times \left(\int_0^\mathbbm{t} \tilde{\mu}_1^2 (\varsigma) \tilde{\chi}_1^2(\|u_j\|_{\mathbb{W}^{1,\infty}}+\|u_k\|_{\mathbb{W}^{1,\infty}})
\|\textbf{y}_{j}\|_{\mathbb{H}^{s+1}}^2
\|u_{j,k}\|_{\mathbb{H}^{s-1}}^2
\mathrm{d}\varsigma\right)^{\frac{1}{2}} \right]\\
&\quad \leq  \frac{1}{4}\mathbb{E}  \sup_{\varsigma\in [0,\mathbbm{t}]}(\|\textbf{y}_{j}\|_{\mathbb{H}^{s+1}}^2 \|u_{j,k}\| _{H^{s-1}}^2)\\
&\quad\quad +C\mathbb{E} \int_0^\mathbbm{t} \tilde{\mu}_1^2 (\varsigma) \tilde{\chi}_1^2(\|u_j\|_{\mathbb{W}^{1,\infty}}+\|u_k\|_{\mathbb{W}^{1,\infty}})
\|\textbf{y}_{j}\|_{\mathbb{H}^{s+1}}^2
\|u_{j,k}\|_{\mathbb{H}^{s-1}}^2
\mathrm{d}\varsigma  .
\end{split}
\end{equation}

\textsf{Estimate for $J_1$.} Using the definition of $\mathbb{H}^s(\mathbb{T}^d)$,  commutating the operator $\Lambda^{s+1}$ with $u_j$ and then integrating by parts, we have
\begin{equation}
\begin{split}
\|u_{j,k}  \|_{H^{s-1}}^2J_1 \leq& C\|u_{j,k}  \|_{H^{s-1}}^2\Big(\| u_{j}\|_{H^{s+1}}(\|\Lambda^{s+1} u_{j}\|_{L^2}\|\nabla u_{j}\|_{L^{\infty}}+\|\nabla u_j\|_{L^{\infty}}\|\Lambda^{s} \nabla u_{j}\|_{L^2}) \\
&+\|\gamma_{j} \|_{H^{s+1}}
(\|\Lambda^{s+1} u_{j}\|_{L^2}\|\nabla\gamma_{j}\|_{L^{\infty}}+\|\nabla u_j\|_{L^{\infty}}\|\Lambda^{s} \nabla\gamma_{j}\|_{L^2})\\
&+\|u_{j}\|_{W^{1,\infty}} \|u_{j} \|_{H^{s+1}}^2+\|u_{j}\|_{W^{1,\infty}}\|\gamma_{j} \|_{H^{s+1}}^2\Big)\\
\leq& C\| u_j\|_{H^{s}}\|\gamma_{j} \|_{H^{s}} \|u_{j,k}  \|_{H^{s-1}}^2  \| \textbf{y}_{j}\|_{\mathbb{H}^{s+1}}^2  .
\end{split}
\end{equation}

\textsf{Estimate for $J_2$ and $J_3$.} By using Lemma \ref{2.6},  the embedding $\mathbb{H}^s(\mathbb{T}^d)\subset\mathbb{W}^{1,\infty}(\mathbb{T}^d)$ ($s>\frac{d}{2}+1$) as well as  assumption \eqref{a1}, we have
\begin{equation*}
\begin{split}
 \|u_{j,k}  \|_{H^{s-1}}^2J_2&\leq C\|\textbf{y}_{j}\|_{\mathbb{H}^s}^2\|u_{j,k}  \|_{H^{s-1}}^2 \|\textbf{y}_{j} \|_{\mathbb{H}^{s+1}}^2,\\
 \|u_{j,k}  \|_{H^{s-1}}^2J_3&\leq  \mu (t) \chi(\|\textbf{y}_{j}\|_{\mathbb{W}^{1,\infty}})\|u_{j,k}  \|_{H^{s}}^2 +\mu (t) \chi(\|\textbf{y}_{j}\|_{\mathbb{W}^{1,\infty}})\|u_{j,k}  \|_{H^{s-1}}^2 \|\textbf{y}_{j} \|_{\mathbb{H}^{s+1}}^2.
\end{split}
\end{equation*}

\textsf{Estimate for $J_4$.}  By applying the BDG inequality and the assumption \eqref{a1}, we have
\begin{equation}
\begin{split}
&\mathbb{E}\sup_{t\in [0,\mathbbm{t}]}\left|\int_0^t\|u_{j,k}  \|_{H^{s-1}}^2J_4(t)\mathrm{d}\mathcal {W}\right|\\
 &\quad \leq  \frac{1}{4}\mathbb{E} \sup_{t\in [0,\mathbbm{t}]}\left(\|u_{j,k}  \|_{H^{s-1}}^2\|\textbf{y}_{j}\|_{\mathbb{H}^{s+1}}^2\right)+\mathbb{E}\int_0^\mathbbm{t}\mu ^2 (\varsigma) \chi^2(\|\textbf{y}\|_{\mathbb{W}^{1,\infty}})\|u_{j,k}  \|_{H^{s-1}}^2 \mathrm{d}\varsigma \\
 &\quad \quad+\mathbb{E}\int_0^\mathbbm{t}\mu ^2 (\varsigma) \chi^2(\|\textbf{y}\|_{\mathbb{W}^{1,\infty}})\|u_{j,k}  \|_{H^{s-1}}^2 \|\textbf{y}\|_{\mathbb{H}^{s+1}}^2\mathrm{d}\varsigma.
\end{split}
\end{equation}

\textsf{Estimate for $\mathcal {K}$.} By using the H\"{o}lder inequality, we deduce that
\begin{equation*}
\begin{split}
|\mathcal {K}| 
\leq& C \mu_1 (t)\tilde{\mu}_1 (t) \chi_1(\|\textbf{y}_j\|_{\mathbb{W}^{1,\infty}})\tilde{\chi}_1(\|\textbf{y}_j\|_{\mathbb{W}^{1,\infty}}+\|\textbf{y}_k\|_{\mathbb{W}^{1,\infty}})
  (\| u_{j,k}\|_{H^{s-1}}^2+\| u_{j,k}\|_{H^{s-1}}^2\|\textbf{y}_j\|_{\mathbb{H}^{s+1}}^2).
\end{split}
\end{equation*}
Noting that from the definition of $\mathbbm{t}_{j,k}^T$, there holds $
\|\textbf{y}_\kappa\|_{\mathbb{W}^{1,\infty}}\leq C\|\textbf{y}_i\|_{\mathbb{H}^{s}}\leq C(M^2+1)$, for any $\kappa\in \{j,k\}$ and $ t\in [0,\mathbbm{t}_{j,k}^T] $,
for some $C>0$ independent of $j$ and $k$. Using the nondecreasing property of  $\chi_1(\cdot),\tilde{\chi}_1(\cdot)$, we get by taking the supremum  over $[0,\mathbbm{t}_{j,k}^T]$ that
\begin{equation*}
\begin{split}
&\mathbb{E}\sup_{\varsigma\in [0,t\wedge\mathbbm{t}_{j,k}^T]}(\|u_{j,k}  \|_{H^{s-1}}^2\|\textbf{y}_{j}\|_{\mathbb{H}^{s+1}}^2)\\
&\quad \leq C\mathbb{E}(\|u_{j,k}(0)  \|_{H^{s-1}}^2\|\textbf{y}_{j}(0)\|_{\mathbb{H}^{s+1}}^2)+C\int_0^t(1+\mu^2 (t)+\tilde{\mu}_1^2(t)) \mathbb{E}\sup_{\varsigma\in [0,r\wedge\mathbbm{t}_{j,k}^T]}\|u_{j,k}  \|_{H^{s-1}}^2\mathrm{d}r\\
&\quad\quad+C\int_0^t (1+\mu^2 (r)+\tilde{\mu}_1^2(r))\mathbb{E}\sup_{\varsigma\in [0,r\wedge\mathbbm{t}_{j,k}^T]}\| u_{j,k}\|_{H^{s-1}} ^2\|\textbf{y}_{j}\|_{\mathbb{H}^{s+1}}^2\mathrm{d}r,
\end{split}
\end{equation*}
for some positive constant $C$ independent of $j$ and $k$. An application of the Gronwall inequality to above integral inequality leads to
\begin{equation*}
\begin{split}
 &\mathbb{E}\sup_{\varsigma\in [0,t\wedge\mathbbm{t}_{j,k}^T]}(\|u_{j,k}  \|_{H^{s-1}}^2\|\textbf{y}_{j}\|_{\mathbb{H}^{s+1}}^2)\leq Ce^{\int_0^T(1+\mu^2 (r)+\tilde{\mu}_1^2(r))\mathrm{d}r }\left(1+\int_0^T (1+\mu^2 (r)+\tilde{\mu}_1^2(r))\mathrm{d}r\right)\\
 &\quad \quad  \times \left(\mathbb{E}(\|J_{1/j}u_0-  J_{1/k}u_0 \|_{H^{s-1}}^2\|J_{1/j}\textbf{y}_0\|_{\mathbb{H}^{s+1}}^2) + \mathbb{E}\sup_{\varsigma\in [0,t\wedge\mathbbm{t}_{j,k}^T]}\|u_{j,k}(\varsigma)  \|_{H^{s-1}}^2\right).
\end{split}
\end{equation*}
which implies the desired inequality \eqref{2.70}, and this completes the proof of Lemma \ref{2.10}.
\end{proof}

Based on the last two lemmas, one can prove the following convergence results.
\begin{lemma}\label{lem:2.11}
Under the same conditions of Lemma 2.9, we have
\begin{equation}\label{2.81}
\begin{split}
\lim_{j\rightarrow\infty}\sup_{k\geq j}\mathbb{E}\sup_{r\in [0, \mathbbm{t}_{j,k}^T ]}\|\textbf{y}_{j} -\textbf{y}_{k}\|_{\mathbb{H}^s}^2=0,
\end{split}
\end{equation}
\begin{equation}\label{2.82}
\begin{split}
\lim_{\omega\rightarrow 0}\sup_{j\geq 0}\mathbb{P}\left\{\sup_{r\in [0, \omega\wedge \mathbbm{t}_{j}^T ]}\|\textbf{y}_{j}(r)\|_{\mathbb{H}^s}\geq \|J_{1/j}\textbf{y}_0\|_{\mathbb{H}^s}+3\right\}=0.
\end{split}
\end{equation}
\end{lemma}

\begin{proof}[\emph{\textbf{Proof.}}]
Since $\{J_\epsilon u\}_{j\geq 1}$ is a Cauchy sequence in $H^s(\mathbb{T}^d)$, we have $
\lim_{j\rightarrow\infty}\sup_{k\geq j}\mathbb{E} \|J_{1/j}\textbf{y}_0-J_{1/k}\textbf{y}_0\|_{\mathbb{H}^s}^2=0$,
which implies that the first term on the R.H.S. of \eqref{2.62} converges to $0$ as $j,k\rightarrow \infty$. By Lemma \ref{lem:2.9}, it suffices to prove that
$$\lim_{j\rightarrow\infty}\sup_{k\geq j}\mathbb{E}\sup_{\varsigma\in [0, \mathbbm{t}_{j,k}^T]}(\|u_{j,k}  \|_{H^{s-1}}^2\|\textbf{y}_{j}\|_{\mathbb{H}^{s+1}}^2)=0.$$ First,  we get from the property of the mollifier $J_\epsilon$ that
 \begin{equation*}
\begin{split}
 \mathbb{E}\left( \|J_{1/j}\textbf{y}_0\|_{\mathbb{H}^{s+1}}^4\right) &\leq C j^4 \mathbb{E}\left( \|\textbf{y}_0\|_{\mathbb{H}^{s}}^4\right)\leq C j^4,\\
 \mathbb{E}\left( \|J_{1/j}u_0-  J_{1/k}u_0 \|_{H^{s-1}}^4\right) &\leq  8 \mathbb{E}\left( \|J_{1/j}u_0- u_0 \|_{H^{s-1}}^4\right)+ 8\mathbb{E}\left( \|u_0- J_{1/k}u_0 \|_{H^{s-1}}^4\right) \\
 &= o(\frac{1}{j^4})+o(\frac{1}{k^4}).
\end{split}
\end{equation*}
From the last two estimates, we have
  \begin{equation*}
\begin{split}
 &\lim_{j\rightarrow\infty}\sup_{k\geq j}\mathbb{E}\left(\|J_{1/j}u_0-  J_{1/k}u_0 \|_{H^{s-1}}^2\|J_{1/j}\textbf{y}_0\|_{\mathbb{H}^{s+1}}^2\right) \\
 &\quad \leq \lim_{j\rightarrow\infty}\sup_{k\geq j}\Big(\mathbb{E}\|J_{1/j}u_0-  J_{1/k}u_0 \|_{H^{s-1}}^4\Big)^{\frac{1}{2}}\Big(\mathbb{E} \|J_{1/j}\textbf{y}_0\|_{\mathbb{H}^{s+1}}^4\Big)^{\frac{1}{2}}\\
 &\quad \leq  C\lim_{j\rightarrow\infty}\sup_{k\geq j}j^2\Big(o(1/j^4)+o(1/k^4)\Big)^{\frac{1}{2}} \\
 &\quad=  C\lim_{j\rightarrow\infty}\sup_{k\geq j} \bigg( \frac{o(1/j^4)}{1/j^4}+\frac{j^4}{k^4}
 \frac{o(1/k^4)}{1/k^4}\bigg)^{\frac{1}{2}}=0.
\end{split}
\end{equation*}
Second, the standard $L^2$-estimate shows that $\|(\textbf{y}_{j} -\textbf{y}_{k})(t)  \|_{H^{s-1}}$ can be
bounded by $C\|\textbf{y}_{j}(0)-\textbf{y}_{k}(0)\|_{\mathbb{H}^{s-1}}$, where $C$ is a positive constant independent of $j$ and $k$, and hence we get
\begin{equation*}
\begin{split}
 &\lim_{j\rightarrow\infty}\sup_{k\geq j} \mathbb{E}\sup_{\varsigma\in [0, \mathbbm{t}_{j,k}^T]}\|(u_{j} -u_{k})(\varsigma)  \|_{H^{s-1}}^2 \leq C\lim_{j\rightarrow\infty}\sup_{k\geq j}\mathbb{E} \|J_{1/j}\textbf{y}_0-J_{1/k}\textbf{y}_0\|_{\mathbb{H}^{s}}=0.
\end{split}
\end{equation*}
Then the convergence \eqref{2.81} follows.

Now we prove \eqref{2.82}. For each $j\in \mathbb{N}^+$ and $\omega>0$, we apply the It\^{o} formula to get
\begin{equation*}
\begin{split}
\sup_{t\in [0,\omega\wedge \mathbbm{t}_{j}^T ]}\| \textbf{y}_j (t)\|_{\mathbb{H}^s}^2\leq& \| \textbf{y}_j (0)\|_{\mathbb{H}^s}^2+ \int_0^{\omega\wedge \mathbbm{t}_{j}^T} \left|\int_{\mathbb{T}^d}  2\Lambda^s \textbf{y}_j \cdot\Lambda^s B( \textbf{y}_j, \textbf{y}_j) \mathrm{d}x\right|\mathrm{d}r\\
&+ \int_0^{\omega\wedge \mathbbm{t}_{j}^T} \left|\int_{\mathbb{T}^d}2\Lambda^s \textbf{y}_j \cdot \Lambda^s  F(\textbf{y}_j) \mathrm{d}x\right|\mathrm{d}r +\int_0^{\omega\wedge \mathbbm{t}_{j}^T} \| G (r,\textbf{y} _j )\|_{L_2(\mathfrak{U}_1,\mathbb{H}^{s})}^2\mathrm{d}r \\
   &+  \sup_{t\in [0,\omega\wedge \mathbbm{t}_{j}^T ]} \left|\int_0^t\int_{\mathbb{T}^d}  2\Lambda^s\textbf{y}_j\cdot\Lambda^sG (r,\textbf{y}_j )\mathrm{d}x \mathrm{d} \mathcal {W}\right| \\
   =&\|J_{1/j}\textbf{y}_0\|_{\mathbb{H}^s}^2+\int_0^{\omega\wedge \mathbbm{t}_{j}^T} (T_1+T_2+T_3)\mathrm{d}r +\sup_{t\in [0,\omega\wedge \mathbbm{t}_{j}^T ]} T_4.
\end{split}
\end{equation*}
By using the Chebyshev inequality and the relationship
\begin{equation*}
\begin{split}
&\left\{\sup_{r\in [0, \omega\wedge \mathbbm{t}_{j}^T ]}\|\textbf{y}_{j}(r)\|_{\mathbb{H}^s}\geq \|J_{1/j}\textbf{y}_0\|_{\mathbb{H}^s}+3 \right\}\\
 &\quad \subset \left\{\int_0^{\omega\wedge \mathbbm{t}_{j}^T} (T_1+T_2+T_3)\mathrm{d}r > \frac{3}{2} \right\}\bigcup \left\{\sup_{t\in [0,\omega\wedge \mathbbm{t}_{j}^T ]} |\int_0^t T_4 \mathrm{d} \mathcal {W} |> \frac{3}{2} \right\},
\end{split}
\end{equation*}
we infer that
\begin{equation*}
\begin{split}
& \mathbb{P}\left\{\int_0^{\omega\wedge \mathbbm{t}_{j}^T} (T_1+T_2+T_3)\mathrm{d}r > \frac{3}{2}\right\} \leq C\mathbb{E} \sup_{[0,\omega\wedge \mathbbm{t}_{j}^T]} ( |T_1| + |T_2| + |T_3| )\\
&\quad\leq C\mathbb{E} \sup_{[0,\omega\wedge \mathbbm{t}_{j}^T]}\Big ( \|\textbf{y}_j\|_{\mathbb{W}^{1,\infty}}\|\textbf{y}_j\|_{\mathbb{H}^s}^2+  \mu^2 (t) \chi^2(\|\textbf{y}_j\|_{\mathbb{W}^{1,\infty}})(1+\|\textbf{y}_j\|_{\mathbb{H}^{s}}^2)  \Big)\\
 &\quad\leq C(1+M^2)(\omega\wedge T),
\end{split}
\end{equation*}
and
\begin{equation*}
\begin{split}
& \mathbb{P} \left\{\sup_{t\in [0,\omega\wedge \mathbbm{t}_{j}^T ]} \left|\int_0^t T_4 \mathrm{d} \mathcal {W} \right|> \frac{3}{2}\right\} \leq C\mathbb{E}\left(\int_0^{\omega\wedge \mathbbm{t}_{j}^T}\|\textbf{y}_j\|_{\mathbb{H}^s}^2\|G (r,\textbf{y}_j \|_{L_2(\mathfrak{A};\mathbb{H}^s)}^2 \mathrm{d} t\right)\\
&\quad \leq C\mathbb{E}\sup_{t\in[0,\omega\wedge \mathbbm{t}_{j}^T]} \Big( \mu^2 (t) \chi^2(\|\textbf{y}_j\|_{\mathbb{W}^{1,\infty}})\|\textbf{y}_j\|_{\mathbb{H}^s}^2
(1+\|\textbf{y}_j\|_{\mathbb{H}^{s}}^2) \Big) \\
&\quad\leq C (1+M^4)(\omega\wedge T).
\end{split}
\end{equation*}
Therefore, we get
$$
\mathbb{P} \left\{\sup_{r\in [0, \omega\wedge \mathbbm{t}_{j}^T ]}\|\textbf{y}_{j}(r)\|_{\mathbb{H}^s}\geq \|J_{1/j}\textbf{y}_0\|_{\mathbb{H}^s}+3 \right\} \leq C (\omega\wedge T) \rightarrow0,~~ \textrm{as} ~~ \omega\rightarrow 0.
$$
This completes the proof of Lemma \ref{lem:2.11}.
\end{proof}

Based on above lemmas, one can now prove the local well-posedness result of Theorem \ref{th1} in the sharp case $s> \frac{d}{2}+1$, $d\geq1$.

\begin{proof}[\emph{\textbf{Proof of Theorem \ref{th1} (1)}}] By Lemmas \ref{lem:2.9}-\ref{lem:2.11} and the uniform bound $\|\textbf{y}_0\|_{\mathbb{H}^s}< M$, one can conclude from the Abstract Cauchy Theorem (cf.  \cite[Lemma 7.1]{41}) that there exists a  stopping time $\mathbbm{t}$ with $\mathbb{P}\{0< \mathbbm{t}\leq T\}=1$ such that
\begin{equation*}
\begin{split}
 \sup_{j\geq 1}\sup_{t\in [0,\mathbbm{t}]}\|\textbf{y}_{j}\|_{\mathbb{H}^s}\leq CM+3,
\end{split}
\end{equation*}
and
$$
 \textbf{y}_{j} \rightarrow  \textbf{y} ~~ \textrm{in} ~~\mathcal {C}([0,\mathbbm{t}];\mathbb{H}^s(\mathbb{T}^d))~~  \textrm{as} ~~j\rightarrow \infty, ~~\mathbb{P}-a.s.
$$
Note that the approximate solutions $\{\textbf{y}_{j}\}_{j\geq 1}$ are continuous $\mathcal {F}_t$-adapted processes with values in
$\mathbb{H}^s(\mathbb{T}^d)$, and hence $\mathcal {F}_t$-predictable ones. As the pointwise limits preserve the measurability, it then follows that the limit process $\textbf{y} $ is also $\mathcal {F}_t$-predictable. By using the decomposition method as that in Subsection 2.4, we infer that $(\textbf{y}, \mathbbm{t} )$ is a local pathwise solution in the sense of Definition \ref{def1}.

Now, we prove that the solution map from $L^\infty(\Omega;\mathbb{H}^s(\mathbb{T}^d))$ into $L^2(\Omega;\mathcal {C}([0,\mathbbm{t}'];\mathbb{H}^s(\mathbb{T}^d)))$ is continuous, for some $\mathbbm{t}'>0$ $\mathbb{P}$-almost surely. Assume that $\textbf{y}_i(t)$ is solution to the SMEP2 with respect to the data $\textbf{y}_{0,i}$, $i=1,2$. Our aim is to find a $\delta>0$ small enough and a stopping time $\mathbbm{t}'$ such that whenever $\|\textbf{y}_{0,1}-\textbf{y}_{0,2}\|_{L^\infty(\Omega;\mathbb{H}^s)}<\delta$, there holds
\begin{equation*}
\begin{split}
\mathbb{E}\sup_{t\in [0,\mathbbm{t}']}\|\textbf{y}_1(t)-\textbf{y}_2(t)\|_{\mathbb{H}^s} ^2 <\epsilon,\quad \mathbb{P}\textrm{-a.s.}
\end{split}
\end{equation*}
To avoid the difficulty caused by the convection term $B(\textbf{y}_i,\textbf{y}_i)$,  similar to \eqref{2.61}, let us consider the mollified initial data $\{J_{1/j}\textbf{y}_{0,i}\}_{j>1}$, and the corresponding solutions are denoted by $\{\textbf{y}_{i,j}\}_{j>1}$, $i=1,2$. For any $T>0$, we define
\begin{equation*}
\begin{split}
 \mathbbm{t}_{i,j}^T \triangleq T\wedge \inf\left\{t>0;~~\|\textbf{y}_{i,j}(t)\|_{\mathbb{H}^s}^2\geq \|J_{1/j}\textbf{y}_{0,i} \|_{\mathbb{H}^s}^2+3\right\},\quad i=1,2.
\end{split}
\end{equation*}
In view of the proof for the Lemma \ref{2.9}-Lemma \ref{2.11}, one can conclude again from the Abstract Cauchy Theorem that, there exists a subsequence  $\{j_k\}$ of $\{j\}$ with $j_k\rightarrow \infty$ as $k\rightarrow\infty$, a sequence of stopping times $\{\bar{\mathbbm{t}}_{i,j_k}\} $ as well as a stopping time $\bar{\mathbbm{t}}_i$, such that
\begin{equation}\label{2.83}
\begin{split}
\mathbbm{t}_{i,j_k}^T\geq \bar{\mathbbm{t}}_{i,j_k} ~~ \textrm{for each}~~ k\geq 1,  \quad  \lim_{k\rightarrow\infty}\bar{\mathbbm{t}}_{i,j_k} = \bar{\mathbbm{t}}_i,~~ \mathbb{P}\textrm{-a.s.},
\end{split}
\end{equation}
and
\begin{equation}\label{2.84}
\begin{split}
 \lim_{k\rightarrow\infty}\sup_{t\in [0,\bar{\mathbbm{t}}_i]}\|\textbf{y}_i-\textbf{y}_{i,j_k}\|_{\mathbb{H}^s}=0,\quad \sup_{t\in [0,\bar{\mathbbm{t}}_i]} \|\textbf{y}_i(t)\|_{\mathbb{H}^s}\leq \| \textbf{y}_{0,i} \|_{\mathbb{H}^s} +3,\quad i=1,2.
\end{split}
\end{equation}
Moreover, there exists $\Omega_{i,j_k}\uparrow \Omega$ as $k\rightarrow\infty$ such that
\begin{equation}\label{2.85}
\begin{split}
 \textbf{1}_{\Omega_{i,j_k}} \sup_{t\in [0,\bar{\mathbbm{t}}_i]} \|\textbf{y}_{i,j_k}(t)\|_{\mathbb{H}^s}\leq \| \textbf{y}_{0,i} \|_{\mathbb{H}^s} +3,\quad \mathbb{P}\textrm{-a.s.},\quad i=1,2.
\end{split}
\end{equation}
Define $
\Omega_{ j_k}\triangleq\Omega_{1,j_k}\cap \Omega_{2,j_k}$. Clearly, $ \Omega_{ j_k} \uparrow \Omega$ as $k\rightarrow\infty $. It then follows from \eqref{2.84} and \eqref{2.85} that $\lim_{k\rightarrow\infty}\mathbb{E}\sup_{t\in [0,\bar{\mathbbm{t}}_i]}\|\textbf{y}_i-\textbf{1}_{\Omega_{ j_k}}\textbf{y}_{i,j_k}\|_{\mathbb{H}^s}^2=0$, $i=1,2$, which implies that for any $\epsilon>0$, there exists  a $k_0>0$ such that
\begin{equation}
\begin{split}
 \mathbb{E}\sup_{t\in [0,\bar{\mathbbm{t}}_i]}\|\textbf{y}_i-\textbf{1}_{\Omega_{ j_k}}\textbf{y}_{i,j_k}\|_{\mathbb{H}^s}^2<\frac{\epsilon}{20},\quad \forall k>k_0,\quad i=1,2.
\end{split}
\end{equation}
By \eqref{2.83}, we get for all $k\geq1$
\begin{equation}
\begin{split}
 &\mathbb{E}\sup_{t\in [0,\bar{\mathbbm{t}}_1\wedge \bar{\mathbbm{t}}_2]}\| \textbf{1}_{\Omega_{ j_k}}\textbf{y}_{1,j_k}(t)-\textbf{1}_{\Omega_{ j_k}}\textbf{y}_{2,j_k}(t)\|_{\mathbb{H}^s}^2\\
 &\quad \leq\mathbb{E}\sup_{t\in [0,\mathbbm{t}_{1,j_k}^T\wedge \mathbbm{t}_{2,j_k}^T]}\| \textbf{y}_{1,j_k}(t)-\textbf{y}_{2,j_k}(t)\|_{\mathbb{H}^s}^2\\
 &\quad\quad +\mathbb{E}\sup_{t\in [\bar{\mathbbm{t}}_1\wedge \bar{\mathbbm{t}}_2\wedge \bar{\mathbbm{t}}_{1,j_k} \wedge \bar{\mathbbm{t}}_{2,j_k} ,\bar{\mathbbm{t}}_1\wedge \bar{\mathbbm{t}}_2]}\| \textbf{1}_{\Omega_{ j_k}}\textbf{y}_{1,j_k}(t)-\textbf{1}_{\Omega_{ j_k}}\textbf{y}_{2,j_k}(t)\|_{\mathbb{H}^s}^2\\
 & \quad\triangleq \textrm{I}_k +\textrm{II}_k.
\end{split}
\end{equation}
For $\textrm{I}_k$, using a similar argument in Lemma \ref{lem:2.9} and \ref{lem:2.10}, one can deduce that
\begin{equation*}
\begin{split}
 \textrm{I} _k
\leq& C \Bigg( \mathbb{E}\sup_{\varsigma\in [0, \mathbbm{t}_{1,j_k}^T\wedge \mathbbm{t}_{2,j_k}^T]}\|u_{1,j_k} -u_{2,j_k} \|_{H^{s-1}}^2+\mathbb{E} \|J_{1/j_k}\textbf{y}_{0,1}-J_{1/j_k}\textbf{y}_{0,2}\|_{\mathbb{H}^s}^2 \\
& +\mathbb{E}(\|J_{1/j_k}u_{0,1}-  J_{1/j_k}u_{0,2}\|_{H^{s-1}}^2\|J_{1/j_k}\textbf{y}_{0,1}\|_{\mathbb{H}^{s+1}}^2)\Bigg)\\
\leq& C \Bigg( \mathbb{E}\sup_{t\in [0, \mathbbm{t}_{1,j_k}^T\wedge \mathbbm{t}_{2,j_k}^T]}\|u_{1,j_k} -u_{2,j_k} \|_{H^{s-1}}^2+\mathbb{E} \| \textbf{y}_{0,1}- \textbf{y}_{0,2}\|_{\mathbb{H}^s}^2 + \frac{1}{j_k^2}\mathbb{E} \| u_{0,1}-   u_{0,2}\|_{H^{s}}^2 \Bigg).
\end{split}
\end{equation*}
For the first term on the R.H.S. of last inequality, we refer back to the Eq.\eqref{2.72} and the estimates for $I_i$, $i=1,...,5$. Due to the continuity of the function $\tilde{\mu}$, the nondecreasing property of $\tilde{\chi}$ and the embedding from $H^{s}(\mathbb{T}^d)$ into $W^{1,\infty}(\mathbb{T}^d)$, one can deduce that $\mathbb{E}\sup_{t\in [0, \mathbbm{t}_{1,j_k}^T\wedge \mathbbm{t}_{2,j_k}^T]}\|u_{1,j_k} -u_{2,j_k} \|_{H^{s-1}}^2\leq C\mathbb{E} \| \textbf{y}_{0,1}- \textbf{y}_{0,2}\|_{\mathbb{H}^s}^2$, which indicates that
\begin{equation} \label{2.88}
\begin{split}
 \textrm{I}_k\leq C \mathbb{E} \| \textbf{y}_{0,1}- \textbf{y}_{0,2}\|_{\mathbb{H}^s}^2 + \frac{C}{j_k^2}\mathbb{E} \| u_{0,1}-   u_{0,2}\|_{H^{s}}^2 .
\end{split}
\end{equation}
For $\textrm{II}_k$, we get by using the convergence \eqref{2.83}
\begin{equation*}
\begin{split}
 \textrm{II}_k \leq&  \mathbb{E}\sup_{t\in [\bar{\mathbbm{t}}_1\wedge \bar{\mathbbm{t}}_2\wedge \bar{\mathbbm{t}}_{1,j_k} \wedge \bar{\mathbbm{t}}_{2,j_k} ,\bar{\mathbbm{t}}_1\wedge \bar{\mathbbm{t}}_2]}(4\| \textbf{y}_{0,1} \|_{\mathbb{H}^s}^2 +4\| \textbf{y}_{0,2} \|_{\mathbb{H}^s}^2 +72)\rightarrow 0,\quad k\rightarrow\infty,
\end{split}
\end{equation*}
which implies that there exists  $k_1>0$ such that $
\textrm{II}_k< \frac{\epsilon}{20}$, $\forall k> k_1$, we have
For fixed $k'=k_0+k_1$,
\begin{equation}\label{2.89}
\begin{split}
 \mathbb{E}\sup_{t\in [0,\bar{\mathbbm{t}}_i]}\|\textbf{y}_i-\textbf{1}_{\Omega_{ j_{k'}}}\textbf{y}_{i,j_{k'}}\|_{\mathbb{H}^s}^2&< \frac{\epsilon}{20},\quad i=1,2;\quad
   \textrm{II}_{k'} < \frac{\epsilon}{20}.
\end{split}
\end{equation}
Moreover, if $\| \textbf{y}_{0,1}- \textbf{y}_{0,2}\|_{L^\infty(\Omega;\mathbb{H}^s)}<\delta < \sqrt{\frac{\epsilon j_{k'}^2}{20C+20C j_{k'}^2}}$, then we get from \eqref{2.88} that
\begin{equation*}
\begin{split}
 \textrm{I}_{k'}\leq C \| \textbf{y}_{0,1}- \textbf{y}_{0,2}\|_{L^\infty(\Omega;\mathbb{H}^s)} ^2 + \frac{C}{j_{k'}^2} \| u_{0,1}-   u_{0,2}\|_{L^\infty(\Omega;\mathbb{H}^s)}^2 <\frac{\epsilon}{20},
\end{split}
\end{equation*}
which combined with the second estimate in \eqref{2.89}  lead to
\begin{equation}\label{2.90}
\begin{split}
 \mathbb{E}\sup_{t\in [0,\bar{\mathbbm{t}}_1\wedge \bar{\mathbbm{t}}_2]}\| \textbf{1}_{\Omega_{j_{k'}}}\textbf{y}_{1,j_{k'}}(t)-\textbf{1}_{\Omega_{ j_{k'}}}\textbf{y}_{2,j_{k'}}(t)\|_{\mathbb{H}^s}^2<\frac{\epsilon}{10} .
\end{split}
\end{equation}
Thereby, by choosing $\mathbbm{t}'=\bar{\mathbbm{t}}_1 \wedge \bar{\mathbbm{t}}_2$, if $\| \textbf{y}_{0,1}- \textbf{y}_{0,2}\|_{L^\infty(\Omega;\mathbb{H}^s)}<\delta$ with $\delta>0$ chosen as before, then we deduce from \eqref{2.89}$_1$ and  \eqref{2.90} that
\begin{equation*}
\begin{split}
&\mathbb{E}\sup_{t\in [0,\mathbbm{t}']}\|\textbf{y}_1(t)-\textbf{y}_2(t)\|_{\mathbb{H}^s} ^2 \\
& \quad \leq 3\bigg(\mathbb{E}\sup_{t\in [0,\mathbbm{t}']}\|\textbf{y}_1(t)- \textbf{1}_{\Omega_{j_{k'}}} \textbf{y}_{1,j_{k'}}(t)\|_{\mathbb{H}^s} ^2+\mathbb{E}\sup_{t\in [0,\mathbbm{t}']}\| \textbf{1}_{\Omega_{j_{k'}}} \textbf{y}_{1,j_{k'}}(t)-\textbf{1}_{\Omega_{j_{k'}}} \textbf{y}_{2,j_{k'}}(t)\|_{\mathbb{H}^s} ^2\\
&\quad\quad +\mathbb{E}\sup_{t\in [0,\mathbbm{t}']}\| \textbf{1}_{\Omega_{j_{k'}}} \textbf{y}_{2,j_{k'}}(t)-\textbf{y}_2(t)\|_{\mathbb{H}^s} ^2\bigg)  \\
&\quad  < 3(\frac{\epsilon}{20}+\frac{\epsilon}{10}+\frac{\epsilon}{20})< \epsilon.
\end{split}
\end{equation*}
The proof of Theorem \ref{th1} is now completed.
\end{proof}

\section{Global existence and blow-up criteria} \label{section3}

\subsection{Global result-I} In this subsection, we shall prove that the nonlocal-type multiplicative noises have a regularization effect on $t$-variable of solutions to the SMEP2.

\begin{proof} [\textbf{\emph{Proof of Theorem \ref{th3}.}}]
Recall that $W$ is a standard real-valued Brownian motion, the functionals $B(\cdot,\cdot)$ and $F(\cdot)$ are defined as in \eqref{1.8}, and the diffusion coefficient $G(t,\textbf{y})$ is now explicitly formulated by
$$
G(t,\textbf{y})=\begin{pmatrix}
c(1+\|\textbf{y}\|_{\mathbb{W}^{1,\infty}})^{\delta} u&0 \\
0&c(1+\|\textbf{y}\|_{\mathbb{W}^{1,\infty}})^{\delta}\gamma
\end{pmatrix},
$$
where $\textbf{y}=(m,\rho)$, $m=\Lambda u$ and $\rho=\Lambda\gamma$. Therefore, the diffusion terms can now be formulated as $G(t,\textbf{y})\textrm{d}\mathcal {W} = c(1+\|\textbf{y}\|_{\mathbb{W}^{1,\infty}})^{\delta}\textbf{y}\textrm{d} W$.
Under the assumptions stated in Theorem \ref{th3}, we conclude from the  Theorem \ref{th1} that the system \eqref{1.8} admits a local strong pathwise solution $(u,\gamma,\bar{\mathbbm{t}})$ in the sense of Definition \ref{def1}, where $\bar{\mathbbm{t}}$ denotes the maximum existence time. To finish the proof of  Theorem \ref{th3}, it is sufficient to prove that $\bar{\mathbbm{t}}=\infty$, $\mathbb{P}$-almost surely.

To this end, let us apply first the Friedrichs mollifier $J_\epsilon$ and then the differential operator $\Lambda^s$ to \eqref{1.8} to obtain
$$\mathrm{d}\Lambda^sJ_\epsilon\textbf{y}+ \Lambda^sJ_\epsilon B(\textbf{y},\textbf{y})\mathrm{d} t+\Lambda^sJ_\epsilon F(\textbf{y})\mathrm{d} t=c(1+\|\textbf{y}\|_{\mathbb{W}^{1,\infty}})^{\delta}\Lambda^s J_\epsilon \textbf{y} \mathrm{d}W(t).
$$
By applying the It\^{o} formula in Hilbert space to $\|J_\epsilon \textbf{y}(t)\|_{\mathbb{H}^s}^2=\|\Lambda^{s}J_\epsilon \textbf{y}(t)\|_{\mathbb{L}^2}^2$, we arrive at
\begin{equation*}
\begin{split}
 \|J_\epsilon \textbf{y}(t)\|_{\mathbb{H}^s}^2 =&\|J_\epsilon \textbf{y}_0\|_{\mathbb{H}^s}^2-2\int_0^t(\Lambda^s J_\epsilon \textbf{y},\Lambda^s (J_\epsilon B(\textbf{y},\textbf{y})+ J_\epsilon F(\textbf{y})))_{\mathbb{L}^2}\mathrm{d}r\\
 & + \int_0^t\|\Lambda^sJ_\epsilon G(r,\textbf{y}) \|_{\mathcal {L}_2(\mathfrak{U};\mathbb{L}^2)}^2 \mathrm{d} r+  \int_0^tc(1+\|\textbf{y}\|_{\mathbb{W}^{1,\infty}})^\delta\|J_\epsilon \textbf{y}\|_{\mathbb{H}^s}^2  \textrm{d}W(r).
\end{split}
\end{equation*}
In order to show that the lifespan can be extended to infinity, let us apply the It\^{o} formula again to the logarithmic functional $\ln(e+\|J_\epsilon \textbf{y}(t)\|_{\mathbb{H}^s}^2)$ to find
\begin{equation}\label{4..1}
\begin{split}
 \mathrm{d} \ln(e+\|J_\epsilon \textbf{y}(t)\|_{\mathbb{H}^s}^2)
  &= -\frac{2(\Lambda^s J_\epsilon \textbf{y},\Lambda^s (J_\epsilon B(\textbf{y},\textbf{y})+ J_\epsilon F(\textbf{y})))_{\mathbb{L}^2}}{e+\|J_\epsilon \textbf{y}(t)\|_{\mathbb{H}^s}^2} \mathrm{d} t \\
  &\quad + \frac{ \|\Lambda^sJ_\epsilon G(t,\textbf{y}) \|_{\mathcal {L}_2(\mathfrak{U};\mathbb{L}^2)}^2}{e+\|J_\epsilon \textbf{y}(t)\|_{\mathbb{H}^s}^2} \mathrm{d} t-  \frac{2 c^2(1+\|\textbf{y}\|_{\mathbb{W}^{1,\infty}})^{2\delta}\|J_\epsilon \textbf{y}\|_{\mathbb{H}^s}^4 }{(e+\|J_\epsilon \textbf{y}(t)\|_{\mathbb{H}^s}^2)^2} \mathrm{d} t\\
  &\quad+ \frac{2 c(1+\|\textbf{y}\|_{\mathbb{W}^{1,\infty}})^\delta\|J_\epsilon \textbf{y}\|_{\mathbb{H}^s}^2 }{e+\|J_\epsilon \textbf{y}(t)\|_{\mathbb{H}^s}^2} \textrm{d}W(t).
\end{split}
\end{equation}
Note that by using the special structure of the diffusion coefficient $G(t,\textbf{y})$, we have
\begin{equation} \label{4..2}
\begin{split}
\|\Lambda^sJ_\epsilon G(t,\textbf{y}) \|_{\mathcal {L}_2(\mathfrak{U};\mathbb{L}^2)}^2 &= c^2(1+\|\textbf{y}\|_{\mathbb{W}^{1,\infty}})^{2\delta} \|J_\epsilon u\|_{H^s}^2+c^2(1+\|\textbf{y}\|_{\mathbb{W}^{1,\infty}})^{2\delta} \|J_\epsilon \gamma\|_{H^s}^2\\
&= c^2(1+\|\textbf{y}\|_{\mathbb{W}^{1,\infty}})^{2\delta}\|J_\epsilon \textbf{y}\|_{\mathbb{H}^s}^2.
\end{split}
\end{equation}
Moreover, by using the commutator estimates, the Sobolev embedding $H^s(\mathbb{T}^d)\subset W^{1,\infty}(\mathbb{T}^d)$ for $s>1+\frac{d}{2}$, and the similar procedure as we did in the proof of Lemma \ref{lem:2.2}, we have
\begin{equation*}
\begin{split}
 &|(J_\epsilon\Lambda^s u,J_\epsilon\Lambda^s ( u\cdot \nabla u + \mathscr{L}_1(u)+ \mathscr{L}_2 (\gamma)  ))_{L^2}|\\
 &\quad \leq C(\| \gamma\|_{W^{1,\infty}}  \| u\|_{H^s}  \|\gamma\|_{H^{s}} + \| u\|_{W^{1,\infty}} \| u\|_{H^s} ^2),
\end{split}
\end{equation*}
and
$$
|(J_\epsilon \Lambda^s\gamma, J_\epsilon\Lambda^s\left(u\cdot \nabla \gamma +\mathscr{L}_3 (u,\gamma)\right))_{L^2}|\leq C (\| \gamma\|_{W^{1,\infty}}+\| u\|_{W^{1,\infty}})( \| u\|_{H^s} ^2+ \|\gamma\|_{H^{s}} ^2),
$$
for some positive constant $C$ independent of $\epsilon$. Therefore, we get from the last two inequalities that there is a constant $\varrho>0$ independent of $\epsilon$ such that
\begin{equation}\label{4..3}
\begin{split}
 |2(\Lambda^s J_\epsilon \textbf{y},\Lambda^s (J_\epsilon B(\textbf{y},\textbf{y})+ J_\epsilon F(\textbf{y})))_{\mathbb{L}^2} |\leq \varrho\|\textbf{y}\|_{\mathbb{W}^{1,\infty}} \|\textbf{y}\|_{\mathbb{H}^s}^2.
\end{split}
\end{equation}
Integrating both sides of \eqref{4..1} over $[0, t\wedge \mathbbm{t}_\ell]$, where
$$
\mathbbm{t}_\ell \triangleq \inf \{t\geq0;~\|\textbf{y}(t)\|_{\mathbb{H}^s}\geq  \ell \},\quad \forall\ell \in \mathbb{N},
$$
and then taking the mathematical expectation, it follows from the estimates \eqref{4..2}-\eqref{4..3} that
\begin{equation}\label{4..4}
\begin{split}
&\mathbb{E}\ln(e+\|J_\epsilon \textbf{y}(t\wedge\mathbbm{t}_\ell)\|_{\mathbb{H}^s}^2)\\
&\leq\mathbb{E}\ln(e+\|J_\epsilon \textbf{y}_0\|_{\mathbb{H}^s}^2)
  +\mathbb{E}\int_0^{t\wedge\mathbbm{t}_\ell}\frac{ \varrho\|\textbf{y}\|_{\mathbb{W}^{1,\infty}} \|\textbf{y}\|_{\mathbb{H}^s}^2}{e+\|J_\epsilon \textbf{y}(t)\|_{\mathbb{H}^s}^2} \mathrm{d} r \\
  &\quad + \mathbb{E}\int_0^{t\wedge\mathbbm{t}_\ell}\frac{ c^2(1+\|\textbf{y}\|_{\mathbb{W}^{1,\infty}})^{2\delta}\|J_\epsilon \textbf{y}\|_{\mathbb{H}^s}^2}{e+\|J_\epsilon \textbf{y}(r)\|_{\mathbb{H}^s}^2} \mathrm{d} r-  \mathbb{E}\int_0^{t\wedge\mathbbm{t}_\ell}\frac{2 c^2(1+\|\textbf{y}\|_{\mathbb{W}^{1,\infty}})^{2\delta}\|J_\epsilon \textbf{y}\|_{\mathbb{H}^s}^4 }{(e+\|J_\epsilon \textbf{y}(r)\|_{\mathbb{H}^s}^2)^2} \mathrm{d} r.
\end{split}
\end{equation}
As the sequence of functions $\{J_\epsilon \textbf{y}\}_{\epsilon>0}$ converges to $\textbf{y}$ in $\mathcal {C}([0,T];\mathbb{H}^s(\mathbb{T}^d))$ as $\epsilon\rightarrow 0$. $\mathbb{P}$-a.s., one can take the limit $\epsilon\rightarrow 0$ in the inequality \eqref{4..4} and using Dominated Convergence Theorem to get
\begin{equation}\label{4..5}
\begin{split}
&\mathbb{E}\ln(e+\|\textbf{y}(t\wedge\mathbbm{t}_\ell)\|_{\mathbb{H}^s}^2)\\
&\leq\mathbb{E}\ln(e+\|\textbf{y}_0\|_{\mathbb{H}^s}^2)
  +\mathbb{E}\int_0^{t\wedge\mathbbm{t}_\ell}\frac{ \varrho\|\textbf{y}\|_{\mathbb{W}^{1,\infty}} \|\textbf{y}\|_{\mathbb{H}^s}^2+c^2(1+\|\textbf{y}\|_{\mathbb{W}^{1,\infty}})^{2\delta}\| \textbf{y}\|_{\mathbb{H}^s}^2}{e+\|\textbf{y}(t)\|_{\mathbb{H}^s}^2} \mathrm{d} r\\
  &\quad -  \mathbb{E}\int_0^{t\wedge\mathbbm{t}_\ell}\frac{2 c^2(1+\|\textbf{y}\|_{\mathbb{W}^{1,\infty}})^{2\delta}\| \textbf{y}\|_{\mathbb{H}^s}^4 }{(e+\|\textbf{y}(t)\|_{\mathbb{H}^s}^2)^2} \mathrm{d} r.
\end{split}
\end{equation}
To get a better understanding of the terms on the R.H.S. of \eqref{4..5}, let us define
\begin{equation*}
\begin{split}
\mathcal {E}( \textbf{y})\triangleq&\frac{\varrho\| \textbf{y}\|_{\mathbb{W}^{1,\infty}} \| \textbf{y}\|_{\mathbb{H}^s}^2+ c^2(1+\|\textbf{y}\|_{\mathbb{W}^{1,\infty}})^{2\delta}\|  \textbf{y}\|_{\mathbb{H}^s}^2}{e+\| \textbf{y} \|_{\mathbb{H}^s}^2}-  \frac{2 c^2(1+\|\textbf{y}\|_{\mathbb{W}^{1,\infty}})^{2\delta}\|\textbf{y}\|_{\mathbb{H}^s}^4 }{(e+\| \textbf{y}\|_{\mathbb{H}^s}^2)^2} \\
&+\frac{c^2 (1+\|\textbf{y}\|_{\mathbb{W}^{1,\infty}})^{2\delta} \| \textbf{y}\|_{\mathbb{H}^s}^{4}}{(e+\| \textbf{y}\|_{\mathbb{H}^s}^2)^2(1+\ln(e+\| \textbf{y}\|_{\mathbb{H}^s}^2))}.
\end{split}
\end{equation*}
Then it follows from \eqref{4..5} that
\begin{equation} \label{4..6}
\begin{split}
&\mathbb{E}\ln(e+\|\textbf{y}(t\wedge\mathbbm{t}_\ell)\|_{\mathbb{H}^s}^2)\\
&\quad \leq\mathbb{E}\ln(e+\|\textbf{y}_0\|_{\mathbb{H}^s}^2)
  +\mathbb{E}\int_0^{t\wedge\mathbbm{t}_\ell}\left(\mathcal {E}( \textbf{y}(r))- \frac{c^2 (1+\|\textbf{y}\|_{\mathbb{W}^{1,\infty}})^{2\delta} \| \textbf{y}\|_{\mathbb{H}^s}^{4}}{(e+\| \textbf{y}\|_{\mathbb{H}^s}^2)^2(1+\ln(e+\| \textbf{y}\|_{\mathbb{H}^s}^2))}\right)\mathrm{d} r.
\end{split}
\end{equation}
We claim that the above defined functional $\mathcal {E}( \textbf{y})$ is uniformly bounded from above.

\textsf{Case 1 ($\delta>\frac{1}{2},c\neq 0$).} Due to the Sobolev embedding $\| \textbf{y}\|_{\mathbb{W}^{1,\infty}}\leq c_{\textrm{emb}}\| \textbf{y}\|_{\mathbb{H}^s}$ and the fact that the function $h(x)=\frac{x}{e+x}$ is monotony increasing for all $x\geq 0$, we find that for any $\delta >\frac{1}{2}$
\begin{equation*}
\begin{split}
\mathcal {E}(\textbf{y})\leq& \varrho\| \textbf{y}\|_{\mathbb{W}^{1,\infty}} +c^2(1+\|\textbf{y}\|_{\mathbb{W}^{1,\infty}})^{2\delta}\\
&-  2 c^2(1+\|\textbf{y}\|_{\mathbb{W}^{1,\infty}})^{2\delta}\left(\frac{\| \textbf{y}\|_{\mathbb{H}^s}^2}{e+\| \textbf{y}(t)\|_{\mathbb{H}^s}^2}\right)^2+c^2\frac{ (1+\|\textbf{y}\|_{\mathbb{W}^{1,\infty}})^{2\delta}  }{ 1+\ln(e+\| \textbf{y}\|_{\mathbb{H}^s}^2)}\\
\leq& (1+\|\textbf{y}\|_{\mathbb{W}^{1,\infty}})^{2\delta} \bigg(\frac{\varrho\| \textbf{y}\|_{\mathbb{W}^{1,\infty}}}{(1+\|\textbf{y}\|_{\mathbb{W}^{1,\infty}})^{2\delta}} +c^2 -2c^2\left(\frac{\| \textbf{y}\|_{\mathbb{H}^s}^2}{e+\| \textbf{y}(t)\|_{\mathbb{H}^s}^2}\right)^2\\
&+ c^2\frac{1}{ 1+\ln(e+\| \textbf{y}\|_{\mathbb{H}^s}^2)} \bigg)\\
\leq& (1+\|\textbf{y}\|_{\mathbb{W}^{1,\infty}})^{2\delta} \bigg[\frac{\varrho}{(1+\|\textbf{y}\|_{\mathbb{W}^{1,\infty}})^{2\delta-1}}+c^2 -2c^2\left(\frac{\| \textbf{y}\|_{\mathbb{W}^{1,\infty}}^2}{c_{\textrm{emb}}^2e+\| \textbf{y}(t)\|_{\mathbb{W}^{1,\infty}}^2}\right)^2 \\
&+ \frac{c^2}{ 1-2\ln(c_{\textrm{emb}})+\ln(c_{\textrm{emb}}^2e+\| \textbf{y}\|_{\mathbb{W}^{1,\infty}}^2)} \bigg].
\end{split}
\end{equation*}
Setting
$$
f_{\delta}(x)=\frac{\varrho}{(1+x)^{2\delta-1}}+c^2 -2c^2\left(\frac{x}{c_{\textrm{emb}}^2e+ x}\right)^2 + \frac{c^2}{ 1-2\ln(c_{\textrm{emb}})+\ln(c_{\textrm{emb}}^2e+x)},~~x\geq0.
$$
It is not difficult to see that the function $f(x)$ is continuous on $[0,\infty)$ and $\lim_{x\rightarrow +\infty}f(x)=-c^2<0$, for any $c\neq 0$. We deduce that
$$
\lim_{\|\textbf{y}\|_{\mathbb{W}^{1,\infty}}\rightarrow +\infty}\mathcal {E}(\textbf{y})\leq \lim_{\|\textbf{y}\|_{\mathbb{W}^{1,\infty}}\rightarrow +\infty} (1+\|\textbf{y}\|_{\mathbb{W}^{1,\infty}})^{2\delta}f(\|\textbf{y}\|_{\mathbb{W}^{1,\infty}})=-\infty.
$$
Thereby, since $[(1+x)^{2\delta}f(x)]|_{x=0}=\varrho+\frac{c^2}{2}>0$, it follows from the last inequality that there exists a positive constant $J$ such that $\mathcal {E}(\textbf{y})\leq J$.

\textsf{Case 2 ($\delta=\frac{1}{2},|c|> \sqrt{\varrho}$).} If $\delta =\frac{1}{2}$, then we have
\begin{equation*}
\begin{split}
f_{\frac{1}{2}}(x)&=\varrho+c^2 -2c^2\left(\frac{x}{c_{\textrm{emb}}^2e+ x}\right)^2 + \frac{c^2}{ 1-2\ln(c_{\textrm{emb}})+\ln(c_{\textrm{emb}}^2e+x)} \\
&\sim \varrho-c^2 , \quad \textrm{as} ~x\rightarrow +\infty,
\end{split}
\end{equation*}
where $\varrho>0$ is the universal constant in \eqref{4..3}. For any $|c|>\sqrt{\varrho}$, we have
$$
\lim_{\|\textbf{y}\|_{\mathbb{W}^{1,\infty}}\rightarrow +\infty}f_{\frac{1}{2}}(\textbf{y})= \varrho-c^2<0,
$$
which implies that $
\lim_{\|\textbf{y}\|_{\mathbb{W}^{1,\infty}}\rightarrow +\infty}\mathcal {E}(\textbf{y})=-\infty$, and hence the functional $\mathcal {E}(\textbf{y})$ is bounded by a positive constant from above.

In both cases, by \eqref{4..6}, one can find a positive constant $J>0$ such that
\begin{equation} \label{4..7}
\begin{split}
&\mathbb{E}\ln\left(e+\|\textbf{y}(t\wedge\mathbbm{t}_\ell)\|_{\mathbb{H}^s}^2\right)
+\mathbb{E}\int_0^{t\wedge\mathbbm{t}_\ell}\frac{c^2 (1+\|\textbf{y}\|_{\mathbb{W}^{1,\infty}})^{2\delta} \| \textbf{y}\|_{\mathbb{H}^s}^{4}}{(e+\| \textbf{y}\|_{\mathbb{H}^s}^2)^2(1+\ln(e+\| \textbf{y}\|_{\mathbb{H}^s}^2))}\mathrm{d} r\\
&\quad\leq\mathbb{E}\ln(e+\|\textbf{y}_0\|_{\mathbb{H}^s}^2)+Jt.
\end{split}
\end{equation}
Since $J_\epsilon \textbf{y}\rightarrow \textbf{y}$ in the topology of $\mathcal {C}([0,T];\mathbb{H}^s(\mathbb{T}^d))$ as $\epsilon\rightarrow\infty$, there exists a non-negative function $\psi(\epsilon)$ such that $\lim_{\epsilon \rightarrow 0}\psi(\epsilon) =0$ and
\begin{equation}\label{4..8}
\begin{split}
&\left|\frac{ \varrho\|\textbf{y}\|_{\mathbb{W}^{1,\infty}} \|\textbf{y}\|_{\mathbb{H}^s}^2+ c^2(1+\|\textbf{y}\|_{\mathbb{W}^{1,\infty}})^{2\delta}\|J_\epsilon \textbf{y}\|_{\mathbb{H}^s}^2}{e+\|J_\epsilon \textbf{y}\|_{\mathbb{H}^s}^2}-\frac{2 c^2(1+\|\textbf{y}\|_{\mathbb{W}^{1,\infty}})^{2\delta}\|J_\epsilon \textbf{y}\|_{\mathbb{H}^s}^4 }{(e+\|J_\epsilon \textbf{y}(r)\|_{\mathbb{H}^s}^2)^2}\right|\\
&\quad = \left|\mathcal {E}(\textbf{y})-\frac{c^2 (1+\|\textbf{y}\|_{\mathbb{W}^{1,\infty}})^{2\delta} \|\textbf{y}\|_{\mathbb{H}^s}^{4}}{(e+\| \textbf{y}\|_{\mathbb{H}^s}^2)^2(1+\ln(e+\| \textbf{y}\|_{\mathbb{H}^s}^2))}+ \psi(\epsilon)\right|\\
&\quad\leq J+\frac{c^2 (1+\|\textbf{y}\|_{\mathbb{W}^{1,\infty}})^{2\delta} \|\textbf{y}\|_{\mathbb{H}^s}^{4}}{(e+\| \textbf{y}\|_{\mathbb{H}^s}^2)^2(1+\ln(e+\| \textbf{y}\|_{\mathbb{H}^s}^2))}+ \psi(\epsilon).
\end{split}
\end{equation}
Now by using the BDG inequality, we deduce from \eqref{4..1} and \eqref{4..8} that for any $T>0$
\begin{equation}\label{4..9}
\begin{split}
&\mathbb{E}\sup_{r\in [T\wedge \mathbbm{t}_\ell]}\ln\left(e+\|J_\epsilon \textbf{y}(r)\|_{\mathbb{H}^s}^2\right)\\
&\leq\mathbb{E}\ln\left(e+\|J_\epsilon \textbf{y}_0\|_{\mathbb{H}^s}^2\right)
  + \mathbb{E}\int_0^{T\wedge\mathbbm{t}_\ell}\bigg(\frac{ \varrho\|\textbf{y}\|_{\mathbb{W}^{1,\infty}} \|\textbf{y}\|_{\mathbb{H}^s}^2+ c^2(1+\|\textbf{y}\|_{\mathbb{W}^{1,\infty}})^{2\delta}\|J_\epsilon \textbf{y}\|_{\mathbb{H}^s}^2}{e+\|J_\epsilon \textbf{y}\|_{\mathbb{H}^s}^2}\\
  &\quad -\frac{2 c^2(1+\|\textbf{y}\|_{\mathbb{W}^{1,\infty}})^{2\delta}\|J_\epsilon \textbf{y}\|_{\mathbb{H}^s}^4 }{(e+\|J_\epsilon \textbf{y}(r)\|_{\mathbb{H}^s}^2)^2} \bigg) \mathrm{d} r \\
  &\quad+\mathbb{E}\left(\sup_{r\in [T\wedge \mathbbm{t}_\ell]}(1+\ln(e+\| \textbf{y}\|_{\mathbb{H}^s}^2))\int_0^{T\wedge\mathbbm{t}_\ell}\frac{4 c^2(1+\|\textbf{y}\|_{\mathbb{W}^{1,\infty}})^{2\delta}\|J_\epsilon \textbf{y}\|_{\mathbb{H}^s}^4 }{(e+\|J_\epsilon \textbf{y} \|_{\mathbb{H}^s}^2)^2(1+\ln(e+\| \textbf{y}\|_{\mathbb{H}^s}^2))} \textrm{d}r\right)^{\frac{1}{2}}\\
  &\leq\mathbb{E}\ln\left(e+\|J_\epsilon \textbf{y}_0\|_{\mathbb{H}^s}^2\right)+\frac{1}{2}\mathbb{E} \sup_{r\in [T\wedge \mathbbm{t}_\ell]}(1+\ln(e+\|J_\epsilon\textbf{y}\|_{\mathbb{H}^s}^2))
  \\
  &\quad+ \mathbb{E}\int_0^{T\wedge\mathbbm{t}_\ell}\left(J+\frac{c^2 (1+\|\textbf{y}\|_{\mathbb{W}^{1,\infty}})^{2\delta} \|\textbf{y}\|_{\mathbb{H}^s}^{4}}{(e+\| \textbf{y}\|_{\mathbb{H}^s}^2)^2(1+\ln(e+\| \textbf{y}\|_{\mathbb{H}^s}^2))}+ \psi(\epsilon)\right) \mathrm{d} r \\
  &\quad+C\mathbb{E}\int_0^{T\wedge\mathbbm{t}_\ell}\frac{  c^2(1+\|\textbf{y}\|_{\mathbb{W}^{1,\infty}})^{2\delta}\|J_\epsilon \textbf{y}\|_{\mathbb{H}^s}^4 }{(e+\|J_\epsilon \textbf{y} \|_{\mathbb{H}^s}^2)^2(1+\ln(e+\|J_\epsilon\textbf{y}\|_{\mathbb{H}^s}^2))} \textrm{d}r \\
  &\leq C\mathbb{E}\ln\left(e+\|\textbf{y}_0\|_{\mathbb{H}^s}^2\right)+\frac{1}{2}\mathbb{E} \sup_{r\in [T\wedge \mathbbm{t}_\ell]}(1+\ln(e+\|J_\epsilon\textbf{y}\|_{\mathbb{H}^s}^2))
  +JT+ (J+\psi(\epsilon))T\\
  &\quad+C\mathbb{E}\int_0^{T\wedge\mathbbm{t}_\ell}\frac{  c^2(1+\|\textbf{y}\|_{\mathbb{W}^{1,\infty}})^{2\delta}\|J_\epsilon \textbf{y}\|_{\mathbb{H}^s}^4 }{(e+\|J_\epsilon \textbf{y} \|_{\mathbb{H}^s}^2)^2(1+\ln(e+\|J_\epsilon\textbf{y}\|_{\mathbb{H}^s}^2))} \textrm{d}r,
\end{split}
\end{equation}
where the second inequality used the estimate \eqref{4..7}. By absorbing the second term on the R.H.S. of \eqref{4..9}, taking the limit as $\epsilon \rightarrow 0$ and using again \eqref{4..7}, we get
\begin{equation} \label{4..10}
\begin{split}
\mathbb{E}\sup_{r\in [T\wedge \mathbbm{t}_\ell]}\ln\left(e+\|\textbf{y}(r)\|_{\mathbb{H}^s}^2\right) \leq C\left(\mathbb{E}\ln\left(e+\|\textbf{y}_0\|_{\mathbb{H}^s}^2\right)+JT\right),
\end{split}
\end{equation}
In view of the definition of $\mathbbm{t}_{\ell}$, we see that $\mathbbm{t}_{\ell}\nearrow \bar{\mathbbm{t}}$ as $\ell\rightarrow\infty$ and
\begin{equation*}
\begin{split}
\{\bar{\mathbbm{t}}< T\} \subset   \{\mathbbm{t}_{\ell}< T\} &\subset  \left\{\sup_{t\in [0,T]} \| \textbf{y}(t)\|_{\mathbb{H}^s}^2 \geq \ell^2 \right\}\\
&\subset  \left\{\sup_{t\in [0,T]}\ln(e+\| \textbf{y}(t)\|_{\mathbb{H}^s}^2) \geq \ln(e+\ell^2) \right\}.
\end{split}
\end{equation*}
By using the Chebyshev inequality to above  events and then using the uniform estimate \eqref{4..10}, we get for any $T>0$
\begin{equation*}
\begin{split}
0\leq\mathbb{P}\{\bar{\mathbbm{t}}< T\} &\leq \mathbb{P}\left\{\sup_{t\in [0,T]}\ln(e+\| \textbf{y}(t)\|_{\mathbb{H}^s}^2) \geq \ln(e+m^2) \right\}\\
&\leq \frac{\mathbb{E} \sup\limits_{t\in [0,T]}\ln\left(e+\| \textbf{y}(t)\|_{\mathbb{H}^s}^2\right) }{\ln(e+m^2)} \\
&\leq \frac{C\left(\mathbb{E}\ln\left(e+\|\textbf{y}_0\|_{\mathbb{H}^s}^2\right)+JT\right)}{\ln(e+m^2)} \rightarrow 0,
\end{split}
\end{equation*}
as $m\rightarrow\infty$, which implies that $\mathbb{P}\{\bar{\mathbbm{t}}< k\}=0$  for any $k\in \mathbb{N}$. It follows that
$$
\mathbb{P}\{\bar{\mathbbm{t}}=\infty\} =1-\mathbb{P}\Bigg(\bigcup_{k\in \mathbb{N}}\{\bar{\mathbbm{t}}<k\}\Bigg)\geq 1-\sum_{k\in \mathbb{N}^+}\mathbb{P}  \{\bar{\mathbbm{t}}<k\} =1.
$$
By means of Theorem \ref{th1} and the similar proof in Theorem \ref{th1}(2), we find that the stoping time $\bar{\mathbbm{t}} $ is actually the maximal existence time $\overline{\mathbbm{t}}$ of the strong pathwise solution to the SMEP2  in the sense of Definition of \ref{def1}. As a consequence, the local strong pathwise solution $(u,\gamma,\bar{\mathbbm{t}})$ is actually a global-in-time  one. The proof of Theorem \ref{th3} is now completed.
\end{proof}

\subsection{Global result-II} In order to prove Theorem \ref{th4}, we shall first transform \eqref{1.15} into a system of random PDEs  (note that $\delta_1=\delta_2=0$ in present case). Define
\begin{equation}\label{trans}
\begin{split}
\mu (t)=e^{\frac{1}{2} c ^2 t -c W(t)}, ~~~\tilde{u}(\omega,t,x)= \mu (t)u(\omega,t,x) ~~~\textrm{and}~~~ \tilde{\gamma}(\omega,t,x)= \mu(t)\gamma(\omega,t,x).
\end{split}
\end{equation}
In terms of \eqref{1.15}$_1$, we get
\begin{equation*}
\begin{split}
 \mathrm{d}\tilde{u}=& \mu \mathrm{d}u +u \mathrm{d}\mu+ \mathrm{d}u \mathrm{d}\mu\\
=&\mu \left(cu\mathrm{d}W-\left( u\cdot \nabla u + \mathscr{L}_1(u) + \mathscr{L}_2 (\gamma) \right)\mathrm{d} t \right) +u (c^2\mu  \mathrm{d}t - c\mu\mathrm{d}W)- c^2  \mu u \mathrm{d} t \\
=& -\mu\left( u\cdot \nabla u + \mathscr{L}_1(u) + \mathscr{L}_2 (\gamma) \right)\mathrm{d} t  \\
=& -\left( \mu^{-1}\tilde{u}\cdot \nabla \tilde{u} + \mu^{-1}\mathscr{L}_1(\tilde{u}) + \mu^{-1}\mathscr{L}_2 (\tilde{\gamma}) \right)\mathrm{d} t.
\end{split}
\end{equation*}
In a similar manner, one can also deduce from \eqref{1.15}$_2$ that
\begin{equation*}
\begin{split}
 \mathrm{d}\tilde{\gamma}= -\left( \mu^{-1}\tilde{u}\cdot \nabla \tilde{\gamma} + \mu^{-1}\mathscr{L}_3 (\tilde{u},\tilde{\gamma})\right)\mathrm{d} t.
\end{split}
\end{equation*}
Therefore, the system \eqref{1.15} can be reformulated as
\begin{equation}\label{4.15}
\left\{
\begin{aligned}
&\partial_t\tilde{u}+  \mu^{-1}\tilde{u}\cdot \nabla \tilde{u} + \mu^{-1}\mathscr{L}_1(\tilde{u}) + \mu^{-1}\mathscr{L}_2 (\tilde{\gamma}) =0,\\
&\partial_t\tilde{\gamma}  +  \mu^{-1}\tilde{u}\cdot \nabla \tilde{\gamma} + \mu^{-1}\mathscr{L}_3 (\tilde{u},\tilde{\gamma})=0,\\
&\tilde{u}|_{t=0}=u_0,\quad \tilde{\gamma}|_{t=0}=\gamma_0 ,
\end{aligned}\quad t\geq 0,~x\in \mathbb{T}^d,~~ \mathbb{P}\textrm{-a.s.,}
\right.
\end{equation}
where the nonlocal terms $\mathscr{L}_1(\cdot),\mathscr{L}_2(\cdot)$ and $\mathscr{L}_3(\cdot,\cdot)$ are defined as before. Given a $\mathcal {F}_0$--adapted initial data $(u_0,\gamma_0)\in L^2(\Omega;\mathbb{H}^s(\mathbb{T}^d))$, Theorem \ref{th1} indicates that the system \eqref{1.15} admits a unique local maximal pathwise solution $(u,\gamma)\in \mathcal {C} ([0,\bar{\mathbbm{t}}),\mathbb{H}^s(\mathbb{T}) \cap \mathcal {C} ^1 ([0,\bar{\mathbbm{t}}),\mathbb{H}^{s-1}(\mathbb{T}))$ $\mathbb{P}$-almost surely. According to the transformation \eqref{trans}, one find that the pair $(\tilde{u},\tilde{\gamma})$ satisfies the random system \eqref{4.15}, which ensures the existence of a unique local strong solution $(\tilde{u},\tilde{\gamma})\in \mathcal {C}  ([0,\bar{\mathbbm{t}}),\mathbb{H}^s(\mathbb{T}) \cap \mathcal {C} ^1 ([0,\bar{\mathbbm{t}}),\mathbb{H}^{s-1}(\mathbb{T}))$ $\mathbb{P}$-almost surely.

\begin{proof} [\textbf{\emph{Proof of Theorem \ref{th4}.}}]
Applying the Littlewood-Paley block $\triangle_j$ (cf. \cite[Chapter 2]{bahouri2011fourier}) to Eq.\eqref{4.15}$_1$ , we get
\begin{equation*}
\begin{split}
&\partial_t\triangle_j\tilde{u}+  \mu^{-1} \tilde{u}\cdot \nabla \triangle_j\tilde{u} \\
 &\quad =\mu^{-1} (\tilde{u}\cdot \nabla \triangle_j\tilde{u}-  \triangle_j(\tilde{u}\cdot \nabla \tilde{u})) -\mu^{-1}(\triangle_j\mathscr{L}_1(\tilde{u}) +\triangle_j\mathscr{L}_2 (\tilde{\gamma})),\quad \mathbb{P}\textrm{-a.s.}
\end{split}
\end{equation*}
Multiplying both sides of above equation by $\triangle_j\tilde{u}$ and integrating on $\mathbb{T}$, we get
\begin{equation}\label{4.16}
\begin{split}
\frac{1}{2} \frac{\mathrm{d}}{\mathrm{d}t}\|\triangle_j\tilde{u}\|_{L^2}^2=&\frac{1}{2}\mu^{-1} (\textrm{div} \tilde{u} ,|\triangle_j\tilde{u}|^2)_{L^2} +\mu^{-1} \Big((\tilde{u}\cdot \nabla \triangle_j\tilde{u}-  \triangle_j(\tilde{u}\cdot \nabla \tilde{u})) ,\triangle_j\tilde{u} \Big)_{L^2}\\
&-\mu^{-1}(\triangle_j\mathscr{L}_1(\tilde{u}) +\triangle_j\mathscr{L}_2 (\tilde{\gamma}),\triangle_j\tilde{u} )_{L^2}\\
\leq & \frac{1}{2}\mu^{-1} \|\textrm{div} \tilde{u}\|_{L^\infty}\|\triangle_j\tilde{u}\|^2_{L^2}+C\mu^{-1}c_j2^{-js} \|\nabla\tilde{u}\|_{L^\infty} \|\triangle_j\tilde{u}\|_{L^2} \|\tilde{u}\|_{B_{2,2}^s}\\
&+\mu^{-1}\|\triangle_j\mathscr{L}_1(\tilde{u}) +\triangle_j\mathscr{L}_2 (\tilde{\gamma})\|_{L^2}\|\triangle_j\tilde{u} \|_{L^2},\quad \mathbb{P}\textrm{-a.s.,}
\end{split}
\end{equation}
where we used the Lie bracket $[A,B]=AB-BA$, and the following commutator estimate (cf. \cite{bahouri2011fourier}) to the second term on the R.H.S. of \eqref{4.16}.
\begin{equation*}
\begin{split}
 \|[\triangle_j,\tilde{u}\cdot \nabla ]\tilde{u}\|_{L^2}  \leq Cc_j2^{-js} \|\nabla\tilde{u}\|_{L^\infty} \|\tilde{u}\|_{B_{2,2}^s},\quad \|\{c_j\}_{j\geq -1}\|_{l^2}=1.
\end{split}
\end{equation*}
So we get from \eqref{4.16} that
\begin{equation}\label{4.17}
\begin{split}
 \frac{\mathrm{d}}{\mathrm{d}t}2^{js}\|\triangle_j\tilde{u}\|_{L^2}\leq & \frac{1}{2}\mu^{-1} \|\nabla\tilde{u}\|_{L^\infty}2^{js}\|\triangle_j\tilde{u}\|_{L^2}+C\mu^{-1}c_j \|\nabla\tilde{u}\|_{L^\infty} \|\tilde{u}\|_{B_{2,2}^s}\\
&+\mu^{-1}2^{js}\|\triangle_j\mathscr{L}_1(\tilde{u}) +\triangle_j\mathscr{L}_2 (\tilde{\gamma})\|_{L^2},\quad \mathbb{P}\textrm{-a.s.}
\end{split}
\end{equation}
Taking the $l^2$-norm on both sides of \eqref{4.17} with respect to $j$ leads to
\begin{equation}\label{4.18}
\begin{split}
 \frac{\mathrm{d}}{\mathrm{d}t}\| \tilde{u}\|_{B_{2,2}^2}\leq & C\mu^{-1} \|\nabla\tilde{u}\|_{L^\infty}\| \tilde{u}\|_{B_{2,2}^s} +\mu^{-1}\| \mathscr{L}_1(\tilde{u}) + \mathscr{L}_2 (\tilde{\gamma})\|_{B_{2,2}^s},\quad \mathbb{P}\textrm{-a.s.}
\end{split}
\end{equation}
Next, by applying the blocks  $\triangle_j$ to Eq.\eqref{4.15}$_2$ yields
\begin{equation*}
\begin{split}
 \partial_t\triangle_j\tilde{\gamma}+  \mu^{-1}\tilde{u}\cdot \nabla \triangle_j\tilde{\gamma} +  \mu^{-1}[\triangle_j, \tilde{u}\cdot \nabla] \tilde{\gamma}+ \mu^{-1}\triangle_j\mathscr{L}_3 (\tilde{u},\tilde{\gamma})=0.
\end{split}
\end{equation*}
Multiplying the both sides of last equation by $\triangle_j\tilde{\gamma}$, integrating the resulted inequality on $\mathbb{T}$ and applying the commutator estimate (cf. Lemma 2.100 in \cite{bahouri2011fourier})
$$
\|\{2^{js}\|[\triangle_j, \tilde{u}\cdot \nabla] \tilde{\gamma}\|_{L^2}\}_{j\geq -1}\|_{l^2} \leq C(\|\nabla\tilde{u}\|_{L^\infty} \|\tilde{\gamma}\|_{B_{2,2}^s}+\|\nabla\tilde{\gamma}\|_{L^\infty} \|\nabla\tilde{u}\|_{B_{2,2}^{s-1}}).
$$
After taking the $l^2$-norm $j\geq -1$ and simplifying the terms, we get
\begin{equation}\label{4.19}
\begin{split}
 \frac{\mathrm{d}}{\mathrm{d}t}\| \tilde{\gamma}\|_{B_{2,2}^s}\leq & C\mu^{-1} (\|\nabla\tilde{u}\|_{L^\infty}\| \tilde{\gamma}\|_{B_{2,2}^s}+\|\nabla\tilde{\gamma}\|_{L^\infty}\| \tilde{u}\|_{B_{2,2}^s}) +\mu^{-1}\| \mathscr{L}_3 (\tilde{u},\tilde{\gamma})\|_{B_{2,2}^s},\quad \mathbb{P}\textrm{-a.s.}
\end{split}
\end{equation}
Since the operator $(I-\Delta)^{-1}\textrm{div}$, $(I-\Delta)^{-1} $ are $S^{-1}$ multipliers, by using the similar method as we did in the proof of Lemma \ref{lem:2.2}, we get
\begin{equation} \label{4.20}
\begin{split}
 \| \mathscr{L}_1(\tilde{u}) + \mathscr{L}_2 (\tilde{\gamma})\|_{B_{2,2}^s}  & \leq C (\|\tilde{u}\|_{L^\infty} +\|\nabla\tilde{u}\|_{L^\infty}) \|\tilde{u}\|_{B_{2,2}^s}+C(\| \tilde{\gamma}\|_{L^\infty}+\|\nabla\tilde{\gamma}\|_{L^\infty}) \| \tilde{\gamma}\|_{B_{2,2}^{s}} ,
\end{split}
\end{equation}
and
\begin{equation}\label{4.21}
\begin{split}
 \| \mathscr{L}_3 (\tilde{u},\tilde{\gamma})\|_{B_{2,2}^s}&\leq C(\| \tilde{\gamma}\|_{L^\infty}+\|\nabla\tilde{\gamma}\|_{L^\infty}) \|\tilde{u}\|_{B_{2,2}^s}+C\|\nabla \tilde{u}\| _{L^\infty}\|\tilde{\gamma}\|_{B_{2,2}^s} .
\end{split}
\end{equation}
By \eqref{4.18}-\eqref{4.21} and the equivalence  $B_{2,2}^s(\mathbb{T})\approx H^s(\mathbb{T})$, we deduce that
\begin{equation}\label{4.22}
\begin{split}
 &\frac{\mathrm{d}}{\mathrm{d}t}(\| \tilde{u}(t)\|_{H^s}+\|\tilde{\gamma}(t)\|_{H^s}) \\
 & \quad \leq   C_1\mu^{-1} (\| \tilde{u}(t)\|_{W^{1,\infty}}+\| \tilde{\gamma}(t)\|_{W^{1,\infty}})(\| \tilde{u}(t)\|_{H^s}+\|\tilde{\gamma}(t)\|_{H^s}) ,\quad \mathbb{P}\textrm{-a.s.,}
\end{split}
\end{equation}
for some positive constant $C_1$ depending only on $s$ and $d$.

Introducing
$$
\phi (t)\triangleq e^{-\frac{c^2}{2}t}\tilde{u}(t)=e^{ -c W(t)}u(t),\quad \varphi(t)\triangleq e^{-\frac{c^2}{2}t}\tilde{\gamma}(t)=e^{ -c W(t)}\gamma(t).
$$
As $\mu(t)=e^{\frac{c^2}{2}t- c W(t)}$, it then follows from \eqref{4.22} that
\begin{equation}\label{4.23}
\begin{split}
& \frac{\mathrm{d}}{\mathrm{d}t}(\| \phi (t)\|_{H^s}+\|\varphi(t)\|_{H^s}) + \frac{c^2}{2}(\| \phi (t)\|_{H^s}+\|\varphi(t)\|_{H^s})\\
 &\quad  \leq  C_1e^{ tW(t)}(\| \phi (t)\|_{W^{1,\infty}}+\| \varphi(t)\|_{W^{1,\infty}})(\| \phi (t)\|_{H^s}+\|\varphi(t)\|_{H^s}) ,\quad \mathbb{P}\textrm{-a.s.}
\end{split}
\end{equation}
For any $\kappa\geq 2$, define the stopping times
\begin{equation*}
\begin{split}
\tilde{\mathbbm{t}}&\triangleq \inf_{t\geq0}\left\{e^{ tW(t)}(\| \phi (t)\|_{W^{1,\infty}}+\| \varphi(t)\|_{W^{1,\infty}})\geq \frac{c^2}{2\kappa C_1}\right\}\\
&=\inf_{t\geq0}\left\{ \| u (t)\|_{W^{1,\infty}}+\| \gamma(t)\|_{W^{1,\infty}} \geq \frac{c^2}{2\kappa C_1}\right\}.
\end{split}
\end{equation*}
As $s>1+\frac{d}{2}$, there is a Sobolev embedding constant  $C_2>0$ such that
$$
\| u_0\|_{W^{1,\infty}}+\| \gamma_0\|_{W^{1,\infty}}\leq C_2 (\| u_0\|_{H^s}+\| \gamma_0\|_{H^s}) \leq  \frac{c^2}{4R\kappa C_1 },
$$
where we used the assumption $ \| u_0\|_{H^s}+\| \gamma_0\|_{H^s} \leq  \frac{c^2}{4R\kappa C_1C_2} $, $R>1$. From the definition of $\tilde{\mathbbm{t}}$, we find that $\tilde{\mathbbm{t}}>0$ $\mathbb{P}$-almost surely. Moreover, \eqref{4.23} indicates
\begin{equation*}
\begin{split}
 \frac{\mathrm{d}}{\mathrm{d}t}(\| \phi (t)\|_{H^s}+\|\varphi(t)\|_{H^s}) + \left(\frac{c^2}{2}-\frac{c^2}{2\kappa }\right)(\| \phi (t)\|_{H^s}+\|\varphi(t)\|_{H^s}) \leq  0,
\end{split}
\end{equation*}
for all $t\in [0,\tilde{\mathbbm{t}})$ $\mathbb{P}$-almost surely, and so $\mathbb{P}$-almost surely
\begin{equation}\label{4.24}
\begin{split}
 \| u(t)\|_{H^s}+\|\gamma(t)\|_{H^s} &\leq e^{cW(t)-(\frac{c^2}{2}-\frac{c^2}{2\kappa  })t}(\| u_0\|_{H^s}+\| \gamma_0\|_{H^s}) \\
 &\leq e^{cW(t)-\frac{1}{2}(\frac{c^2}{2}-\frac{c^2}{2\kappa  })t}e^{-\frac{1}{2}(\frac{c^2}{2}-\frac{c^2}{2\kappa  })t}(\| u_0\|_{H^s}+\| \gamma_0\|_{H^s}),\quad \forall t\in [0,\tilde{\mathbbm{t}}).
\end{split}
\end{equation}
Define the stopping time
\begin{equation*}
\begin{split}
\tilde{\tilde{\mathbbm{t}}}(R)\triangleq \inf \left\{t\geq0;~~ e^{cW(t)-\frac{1}{2}(\frac{c^2}{2}-\frac{c^2}{2\kappa  })t}\geq R \right\},\quad \forall R>1.
\end{split}
\end{equation*}
Clearly, $\tilde{\tilde{\mathbbm{t}}}(R)>0$ $\mathbb{P}$-almost surely, and we get from \eqref{4.24} and the conditions on initial data that
\begin{equation*}
\begin{split}
 \| u(t)\|_{H^s}+\|\gamma(t)\|_{H^s} \leq \frac{c^2}{4R\kappa C_1C_2}R e^{-\frac{1}{2}(\frac{c^2}{2}-\frac{c^2}{2\kappa  })t} \leq \frac{c^2 }{4\kappa C_1C_2}  ,\quad \forall t\in [0,\tilde{\mathbbm{t}}\wedge \tilde{\tilde{\mathbbm{t}}}(R)).
\end{split}
\end{equation*}
According to the definition of $\tilde{\mathbbm{t}}$, the above bound shows that $
\tilde{\mathbbm{t}} \wedge \tilde{\tilde{\mathbbm{t}}}(R)=\tilde{\tilde{\mathbbm{t}}}(R) $.
Therefore,
\begin{equation*}
\begin{split}
\sup_{t\in [0,\tilde{\tilde{\mathbbm{t}}}(R)]}( \| u(t)\|_{W^{1,\infty}}+\|\gamma(t)\|_{W^{1,\infty}}) \leq  C_2\sup_{t\in [0,\tilde{\tilde{\mathbbm{t}}}(R)]}( \| u(t)\|_{H^s}+\|\gamma(t)\|_{H^s})\leq \frac{c^2 }{4\kappa C_1 },
\end{split}
\end{equation*}
which implies that $\tilde{\mathbbm{t}}\geq \tilde{\tilde{\mathbbm{t}}}(R)$. Observing from last estimate  that the maximal pathwise solution $(u,\gamma,\mathbbm{t})$ of \eqref{4.15} is global in time on the set $\{\tilde{\tilde{\mathbbm{t}}}(R)=\infty\}$, that is, on the set where $\Theta(t)\triangleq e^{cW(t)-\frac{1}{2}(\frac{c^2}{2}-\frac{c^2}{2\kappa  })t}$ always stay below $R$. To make sense the existence of global solution, one have to estimate the probability $\mathbb{P}\{\tilde{\tilde{\mathbbm{t}}}(R)=\infty\}$.

First note that on  the event $\{\tilde{\tilde{\mathbbm{t}}}(R)=\infty\}$, $
0<\Theta(t) \leq R$, for all $t\geq0 $.
Note that $\Theta(t)$ is a geometric Brownian motion satisfying
\begin{equation*}
\begin{split}
 \mathrm{d}\Theta(t)=\left(\frac{c^2}{4}+\frac{c^2}{4\kappa  } \right)\Theta(t) \mathrm{d}t + c \Theta(t)\mathrm{d}W(t).
\end{split}
\end{equation*}
Applying the Ito\^{o} formula to $\Theta^\lambda(t)$, $\lambda\in \mathbb{R}$, we find
\begin{equation}\label{4.25}
\begin{split}
 \mathrm{d}\Theta^\lambda (t)&= \lambda\Theta^{\lambda-1} (t)\mathrm{d} \Theta(t) + \frac{\lambda(\lambda-1)}{2} \Theta^{\lambda-2}(t) \mathrm{d} \Theta(t)\mathrm{d} \Theta(t)\\
 &= \left[\left(\frac{c^2}{4}+\frac{c^2}{4\kappa  } \right)\lambda+\frac{ c^2 \lambda(\lambda-1)}{2} \right]\Theta^{\lambda} (t) \mathrm{d}t + c \lambda\Theta^{\lambda} (t)\mathrm{d}W(t).
\end{split}
\end{equation}
Integrating \eqref{4.25} up to $t\wedge \tilde{\tilde{\mathbbm{t}}}(R)$ and taking the expectation value, we have
\begin{equation*}
\begin{split}
 \mathbb{E}\left[\Theta^\lambda (t\wedge \tilde{\tilde{\mathbbm{t}}}(R)) \right]= 1+\mathbb{E}\int_0^{t\wedge \tilde{\tilde{\mathbbm{t}}}(R)}\bigg[\left(\frac{c^2}{4}+\frac{c^2}{4\kappa  } \right)\lambda+\frac{ c^2 \lambda(\lambda-1)}{2} \bigg] \Theta^{\lambda} (t) \mathrm{d}r.
\end{split}
\end{equation*}
Choosing $\lambda=\frac{1}{2}-\frac{1}{2\kappa}$ in last equality, it follows that
\begin{equation}\label{4.26}
\begin{split}
 \mathbb{E}\left[\Theta^{\frac{1}{2}-\frac{1}{2\kappa}} (t\wedge \tilde{\tilde{\mathbbm{t}}}(R)) \right]= 1,\quad \forall t>0.
\end{split}
\end{equation}
Using the continuity  of the measures and the fact that $\tilde{\tilde{\mathbbm{t}}}(R)\leq \mathbbm{t}$ is increasing in $R$, we get
\begin{equation*}
\begin{split}
 \mathbb{P}\{\mathbbm{t}=\infty\}&\geq\mathbb{P}\{\tilde{\tilde{\mathbbm{t}}}(R)=\infty\} = \mathbb{P}\left\{\bigcap_{n}(\tilde{\tilde{\mathbbm{t}}}(R)>n)\right\}=\lim_{n\rightarrow\infty} \mathbb{P}\{\tilde{\tilde{\mathbbm{t}}}(R)>n\}\\
 & \geq \lim_{n\rightarrow\infty} \mathbb{P}\{\Theta^{\frac{1}{2}-\frac{1}{2\kappa}} (n\wedge \tilde{\tilde{\mathbbm{t}}}(R))< R^{\frac{1}{2}-\frac{1}{2\kappa}}\}\\
 & \geq 1-\lim_{n\rightarrow\infty} \mathbb{P}\{\Theta^{\frac{1}{2}-\frac{1}{2\kappa}} (n\wedge \tilde{\tilde{\mathbbm{t}}}(R))\geq R^{\frac{1}{2}-\frac{1}{2\kappa}}\}\\
 & \geq 1-\frac{1}{R^{\frac{1}{2}-\frac{1}{2\kappa}}}\lim_{n\rightarrow\infty} \mathbb{E} \left[\Theta^{\frac{1}{2}-\frac{1}{2\kappa}} (n\wedge \tilde{\tilde{\mathbbm{t}}}(R))\right] =1- \frac{1}{R^{\frac{1}{2}-\frac{1}{2\kappa}}},
\end{split}
\end{equation*}
where the third inequality used the Chebyshev inequality, and the final limit used the identity \eqref{4.26} with $t=n\in \mathbb{N}^+$. This completes the proof of Theorem \ref{th4}.
\end{proof}

\subsection{Blow-up phenomena}
 In this subsection, we prove that, if $\delta_1= \delta_2=0$ and $d=1$, then the strong pathwise solutions to the SMEP2 will blow up in finite time  with some shape condition on initial data. In present case, the system \eqref{4.15} reduces to the following one dimensional random PDEs:
\begin{equation}\label{4.27}
\left\{
\begin{aligned}
&\partial_t\tilde{u}+ \mu ^{-1}\tilde{u} \tilde{u}_x +\mu ^{-1}\partial_xG\star (\tilde{u}^2+\frac{1}{2}\tilde{u}_x^2+\frac{1}{2}\tilde{\gamma}^2-\tilde{\gamma}_x^2)=0,\\
&\partial_t\tilde{\gamma}  + \mu ^{-1}\tilde{u}\tilde{\gamma}_x+ \mu ^{-1}G\star ((\tilde{u}_x\tilde{\gamma}_x)_x+\tilde{u}_x\tilde{\gamma})=0,\\
&\tilde{u}|_{t=0}=u_0,\quad \tilde{\gamma}|_{t=0}=\gamma_0 ,
\end{aligned}\quad t\geq 0,~x\in \mathbb{T}^d.
\right.
\end{equation}

The following lemma tells us that the solutions to \eqref{4.27} are $H^1$-conserved.

\begin{lemma} \label{lem:4.1}
Assume that $0\neq c\in \mathbb{R}$, and $ (u_0,\gamma_0)$ is a $\mathbb{H}^s$-valued $\mathcal {F}_0$-measurable initial data in $L^2(\Omega;\mathbb{H}^s(\mathbb{T}^d))$.  Let $(u,\gamma, \bar{\mathbbm{t}})$ be the maximum local strong solution to the system \eqref{4.27}.
Then we have
\begin{equation*}
\begin{split}
E( t)&\triangleq\int_\mathbb{T}\left(\tilde{u}^2(t)+(\partial_xu)^2(t)+\tilde{\gamma}^2(t)+(\partial_x \gamma)^2 (t)
\right) \mathrm{d} x\\
&=\int_\mathbb{T}\left(u_0^2+(\partial_xu_{0 })^2+\gamma_0^2+(\partial_x\gamma_{0 })^2\right) \mathrm{d} x=E(0),\quad \mathbb{P}\textrm{-a.s.},
\end{split}
\end{equation*}
for all $t>0$. Moreover, we have for all $t\geq0$
\begin{equation*}
\begin{split}
 \|u(t)\|_{L^\infty}^2+\|\gamma(t)\|_{L^\infty}^2 \leq  \frac{1}{2}(\|u_0\|_{H^1}^2+\|\gamma_0\|_{H^1}^2 ),\quad \mathbb{P}\textrm{-a.s.}
\end{split}
\end{equation*}
\end{lemma}

\begin{proof}[\emph{\textbf{Proof.}}]
By using a density argument, it is sufficient to prove Lemma \ref{lem:4.1} in the case of $s>3$.  Differentiating the equation in \eqref{4.27} with respect to $x$
and using the identity $\partial_x^2G \star h= G \star h-h$ for all  $h\in L^2(\mathbb{T})$, we get
\begin{equation*}
\begin{split}
\partial_t\tilde{u}_x+ \mu ^{-1}\tilde{u} \tilde{u}_{xx} + \mu ^{-1} \tilde{u}_x^2 +\mu ^{-1} G\star f= \mu ^{-1}  f ,\quad \mathbb{P}\textrm{-a.s.},
\end{split}
\end{equation*}
and
 \begin{equation*}
\begin{split}
\partial_t\tilde{\gamma}_x  + \mu ^{-1}\tilde{u}_x\tilde{\gamma}_x + \mu ^{-1}\tilde{u}\tilde{\gamma}_{xx}+ \mu ^{-1}\partial_xG\star g=0,\quad \mathbb{P}\textrm{-a.s.},
\end{split}
\end{equation*}
where $f= (\tilde{u}^2+\frac{1}{2}\tilde{u}_x^2+\frac{1}{2}\tilde{\gamma}^2-\tilde{\gamma}_x^2)$ and $g=(\tilde{u}_x\tilde{\gamma}_x)_x+\tilde{u}_x\tilde{\gamma}$.
Using previous equations and integrating by parts on $\mathbb{T}$, we obtain
 \begin{equation*}
\begin{split}
 \frac{ \mathrm{d}}{ \mathrm{d} t} E(t)
   =&  2\int_\mathbb{T} \bigg( -\mu ^{-1}\tilde{u}\partial_xG\star f   - \frac{1}{2}\mu ^{-1} \tilde{u}_x^3 -\mu ^{-1} \tilde{u}_xG\star f+\mu ^{-1}  \tilde{u}_xf \\
 &+\frac{1}{2}\mu ^{-1}\tilde{u}_x\tilde{\gamma}^2-\mu ^{-1}\tilde{\gamma}G\star g  -\mu ^{-1}\tilde{u}_x\tilde{\gamma}_x^2 +\frac{1}{2}\mu ^{-1}\tilde{u}_x\tilde{\gamma}_x^2 - \mu ^{-1}\tilde{\gamma}_x\partial_xG\star g \bigg) \mathrm{d} x\\
 =&  2\int_\mathbb{T} \left(\frac{1}{2}\mu ^{-1}\tilde{u}_x\tilde{\gamma}^2+\mu ^{-1}  \tilde{u}_xf-\frac{1}{2}\mu ^{-1}\tilde{u}_x\tilde{\gamma}_x^2- \frac{1}{2}\mu ^{-1} \tilde{u}_x^3 -\mu ^{-1}\tilde{\gamma} g \right) \mathrm{d} x= 0,
\end{split}
\end{equation*}
 which implies the desired identity. The $L^\infty$-estimate follows from the embedding $H^1(\mathbb{T})\subset L^\infty (\mathbb{T})$, and this completes the proof of Lemma \ref{lem:4.1}.
\end{proof}

Considering the random characteristic flow
\begin{equation}\label{4.28}
\left\{
\begin{aligned}
& \frac{\mathrm{d}\Phi (\omega,t,x) }{\mathrm{d} t}= \mu ^{-1}(t)\tilde{u}(\omega,t,\Phi (t,x)) ,\cr
&\Phi (\omega,0,x) = x,
\end{aligned}\quad t\geq 0,~x\in \mathbb{T}^d,
\right.
\end{equation}
where the function $\tilde{u}$ denotes the unique solution to \eqref{4.27}.

\begin{lemma} \label{lem:4.3}
Let $s\geq3$, and $(u_0,\gamma_0)\in L^2(\Omega;\mathbb{H}^s(\mathbb{T}))$ be a $\mathcal {F}_0$-measurable initial data. Assume that $(u,\gamma,\mathbbm{t})$ is the associated  local pathwise solution of \eqref{4.27}. Then

\begin{itemize}
\item [(1)] Eq.\eqref{4.28} has a unique solution $\Phi\in  \mathcal {C} ^1([0,\mathbbm{t})\times \mathbb{T})$, $\mathbb{P}$-almost surely. For a.e. $\omega \in \Omega$, the map  $\Phi (\omega,t,\cdot):\mathbb{T}\rightarrow\mathbb{T}$ is an increasing diffeomorphism of $\mathbb{T}$, and
\begin{eqnarray*}
\mathbb{P}\{\Phi _x(\omega,t,x)>0,~\forall (t,x)\in [0,\mathbbm{t})\times \mathbb{T}\}=1,
\end{eqnarray*}
Moreover, for all $(t,x)\in [0,\mathbbm{t})\times\mathbb{T}$,
\begin{eqnarray*}\label{4.29}
\tilde{\rho}( t, \Phi (\omega,t,x) )\Phi_x(\omega,t,x)=\rho_0(x),\quad \mathbb{P}\textrm{-a.s.}
\end{eqnarray*}

\item [(2)] If there exists a $M>0$ such that $u_x(t,x)\geq -M$ for all $(t,x)\in [0,\mathbbm{t})\times\mathbb{T}$, then
\begin{eqnarray*}\label{4.30}
\|\tilde{\rho}( t, \cdot)\|_{L^\infty} \leq e^{Mt}\|\rho_0 \|_{L^\infty},
\end{eqnarray*}
for all $t\in[0,\mathbbm{t})$, $\mathbb{P}$-almost surely.
\end{itemize}
\end{lemma}

\begin{proof}[\emph{\textbf{Proof.}}]
 For fixed $x\in \mathbb{T}$, Eq.\eqref{4.28} is a random ODE. For a.e. $\omega\in \Omega$, the generator $\mu^{-1}\tilde{u}$ is bounded and Lipschitz continuous in $x$.  Then one can conclude from the classical theory for ODEs that Eq.\eqref{4.28} has a unique solution $\Phi(\omega,t,x)\in  \mathcal {C}  ^1([0,\tau)\times\mathbb{T})$,  $\mathbb{P}$-almost surely.

 Differentiating \eqref{4.28} with respect to $x$,  we get  for a.e. $\omega\in \Omega$,
\begin{equation}\label{4.31}
\left\{
\begin{aligned}
&\frac{\mathrm{d}\Phi_x(\omega,t,x)}{\mathrm{d} t}= \mu^{-1}(\omega,t) u_x (\omega,t,\Phi(\omega,t,x))\Phi_x(\omega,t,x),\cr
&\Phi_x(\omega,0,x)=1,
\end{aligned}\quad t\geq 0,~x\in \mathbb{T}^d,
\right.
\end{equation}
which implies that
$$
\Phi_x(\omega,t,x)= e^ { \int_0^t \mu^{-1}(\omega,t) u_x (\omega,t,q(\omega,t,x)) \mathrm{d} s } >0,\quad                       \mathbb{P}\mbox{-a.s.}
$$
Hence, the function $\Phi(\omega,t,x)$ is an increasing diffeomorphism of $\mathbb{T}$ before blow-up $\mathbb{P}$-almost surely. For a.e. $\omega\in\Omega$, we get from \eqref{4.27}, \eqref{4.31} that
 \begin{equation*}
\begin{split}
  &\frac{\mathrm{d}}{\mathrm{d} t}(\tilde{\rho}( t, \Phi (\omega,t,x) )\Phi_x(\omega,t,x))\\
 & \quad=(\partial_t \tilde{\rho})( t, \Phi (\omega,t,x) )\Phi_x(\omega,t,x)+ \tilde{\rho}_x( t, \Phi (\omega,t,x) ) \Phi_x(\omega,t,x)\Phi_t (\omega,t,x)\\
 & \quad \quad + \tilde{\rho}( t, \Phi (\omega,t,x) )\Phi_{xt}(\omega,t,x)\\
 &\quad = \Big(-\mu ^{-1}(\tilde{\rho} \tilde{u})_x(\omega, t, \Phi (\omega,t,x) ) + \mu ^{-1}\tilde{\rho}_x( t, \Phi (\omega,t,x) )\tilde{u}(t,\Phi (\omega,t,x))  \\
 & \quad \quad + \mu^{-1}\tilde{\rho}( t, \Phi (\omega,t,x) )  \tilde{u}_x (\omega,t,\Phi(\omega,t,x))\Big)\Phi_x(\omega,t,x) = 0.
\end{split}
\end{equation*}
Integrating above equation leads to the desired identity. Moreover, by using the iterated logarithm
$\limsup_{t\rightarrow\infty}\frac{W(t)}{\sqrt{2t \log\log t}}=1$, we get $
 \sup_{t> 0}\mu^{-1}(\omega,t)=\sup_{t> 0}e^{c W(t)-\frac{c^2}{2}t}\leq C<\infty$, $ \mathbb{P}$-a.s.
Thereby,
\begin{eqnarray*}
\|\tilde{\rho}( t, \cdot)\|_{L^\infty}=\|\tilde{\rho}( t, \Phi (\omega,t,\cdot))\|_{L^\infty}\leq e^ { -\int_0^t \mu^{-1}(\omega,r) u_x (\omega,r,\Phi(r,x)) \mathrm{d}r } \|\rho_0 \|_{L^\infty}\leq e^ {Mt} \|\rho_0 \|_{L^\infty} .
\end{eqnarray*}
This finishes the proof of Lemma \ref{lem:4.3}.
\end{proof}

Based on above lemmas, we can now give the proof of the main result.

\begin{proof}[\emph{\textbf{Proof  of Theorem \ref{th5}.}}] The proof  will be divided into two steps.

{\textsf{Step 1:}} We show that for any $s>\frac{3}{2}$
 \begin{equation}\label{4.32}
\begin{split}
\mathbb{P}\Big(\textbf{1}_{\{\lim _{t\rightarrow \mathbbm{t}} \inf_{x\in \mathbb{T}}u_x(t,x)=-\infty\}}=\textbf{1}_{\{\lim _{t\rightarrow \mathbbm{t}} \|\textbf{y}(t)\|_{H^s}= \infty \} }\Big)=1,
\end{split}
\end{equation}
which combined with Theorem \ref{th1}(1) yield that the singularities of the solutions can occur only in the form of wave breaking.

By using the Sobolev embedding theorem, we get $\{\omega;~\lim _{t\rightarrow \mathbbm{t}} \inf_{x\in \mathbb{T}}u_x(t,x)=-\infty\}\subseteq \{\omega;~\lim _{t\rightarrow \mathbbm{t}} \|\textbf{y}(t)\|_{H^s}= \infty \}$. Conversely, we will prove
$$
A\triangleq\left\{\lim _{t\rightarrow \mathbbm{t}} \|\textbf{y}(t)\|_{H^s}= \infty \right\}^C      \supseteq B \triangleq\left\{\lim _{t\rightarrow \mathbbm{t}} \inf_{x\in \mathbb{T}}u_x(t,x)=-\infty\right\}^C.
$$
Notice that the even $B$ happens if and only if there exists a positive constant $M=M(\omega)>0$ such that
$\tilde{u}_x(t,x)\geq -M$ for all $(t,x)\in \mathbb{R}^+ \times \mathbb{T}$,$\mathbb{P}$-almost surely. It follows from Lemma \ref{lem:4.3} that for all $t>0$
$$
\|\tilde{\rho}( t, \cdot)\|_{L^\infty} \leq e^ {Mt} \|\rho_0 \|_{L^\infty} ~~~   \textrm{on}~~ B.
$$
In order to prove $B\subseteq A$, we shall utilize another equivalent form of the system \eqref{4.27}:
\begin{equation}\label{4.33}
\left\{
\begin{aligned}
& \partial_t\tilde{m} +   \mu ^{-1}\tilde{u}\tilde{m}_x+2 \mu ^{-1}\tilde{m}\tilde{u}_x+ \mu ^{-1}\tilde{\rho}\bar{\tilde{\rho}}_x=0,\\
&\partial_t\tilde{\rho} + \mu ^{-1}(\tilde{\rho} \tilde{u})_x=0,\\
& \tilde{\rho}= (1-\partial_x^2 ) (\bar{\tilde{\rho}}-\bar{\tilde{\rho}}_0) \\
&\tilde{m}|_{t=0}=(1-\partial_x^2)  u_0,\quad \tilde{\rho}|_{t=0}=(1-\partial_x^2)  \gamma_0 ,
\end{aligned}\quad t\geq 0,~x\in \mathbb{T}^d,
\right.
\end{equation}
where $\tilde{m}= (1-\partial_x^2)\tilde{u}$, $\tilde{\rho}= (1-\partial_x^2)\tilde{\gamma}$.

Multiplying $\eqref{4.33}_1$ by $\tilde{m}$ and integrating on $\mathbb{T}$, we get
 \begin{equation*}
\begin{split}
\frac{1}{2}\frac{\mathrm{d}}{\mathrm{d} t} \int_\mathbb{T}\tilde{m}^2\mathrm{d} x&=-\int_\mathbb{T}\tilde{m} (\mu ^{-1}\tilde{u}\tilde{m}_x+2 \mu ^{-1}\tilde{m}\tilde{u}_x+ \mu ^{-1}\tilde{\rho}\bar{\tilde{\rho}}_x )\mathrm{d} x\\
&=-\int_\mathbb{T} ( \frac{3}{2} \mu ^{-1}\tilde{m}^2\tilde{u}_x+ \mu ^{-1}\tilde{m}\tilde{\rho}\bar{\tilde{\rho}}_x )\mathrm{d} x.
\end{split}
\end{equation*}
Repeating the same procedure to \eqref{4.33}$_2$, we obtain $
\frac{1}{2}\frac{\mathrm{d}}{\mathrm{d} t} \int_\mathbb{T}\tilde{\rho}^2\mathrm{d} x = -\frac{1}{2}\int_\mathbb{T} \mu ^{-1} \tilde{u}_x\tilde{\rho}^2 \mathrm{d} x$.
Thereby, we arrive at
 \begin{equation}\label{4.34}
\begin{split}
\frac{1}{2}\frac{\mathrm{d}}{\mathrm{d} t} \int_\mathbb{T}(\tilde{m}^2+\tilde{\rho}^2)\mathrm{d} x =& -\int_\mathbb{T} ( \frac{3}{2} \mu ^{-1}\tilde{m}^2\tilde{u}_x+ \mu ^{-1}\tilde{m}\tilde{\rho}\bar{\tilde{\rho}}_x )\mathrm{d} x-\frac{1}{2}\int_\mathbb{T} \mu ^{-1} \tilde{u}_x\tilde{\rho}^2 \mathrm{d} x\\
\leq&\frac{3M}{2}\int_\mathbb{T}(\mu ^{-1}\tilde{m}^2  + \mu ^{-1}  \tilde{\rho}^2) \mathrm{d} x+ \|\bar{\tilde{\rho}}_x\|_{L^\infty} \int_\mathbb{T} \mu ^{-1}\tilde{m}\tilde{\rho} \mathrm{d} x,\quad \mathbb{P}\textrm{-a.s.}
\end{split}
\end{equation}
Meanwhile, it follows from the formulation \eqref{4.33}$_3$ that
$$
\|\bar{\tilde{\rho}}_x\|_{L^\infty}=\|(\bar{\tilde{\rho}}-\bar{\tilde{\rho}}_0)_x\|_{L^\infty}=\|(1-\partial_x^2)^{-1}\partial_x\tilde{\rho}\|_{L^\infty}=\|\tilde{\gamma}_x(t,\cdot)\|_{L^\infty}.
$$
Moreover, by using the identity $\partial_xG\star f = G\star f-f$, the equation for $\tilde{\gamma}_x$ (cf. proof of Lemma \ref{lem:4.1}), the $H^1$-conservation law as well as the fact that $\Phi$ is an increasing diffeomorphism of $\mathbb{T}$, we have
 \begin{equation*}
\begin{split}
\left| \frac{\mathrm{d}\tilde{\gamma}_x(t,\Phi(\omega,t,x) )}{\mathrm{d} t }\right|  =&\left|\tilde{\gamma}_{xt}(t,\Phi(\omega,t,x) ) +\mu ^{-1}\tilde{\gamma}_{xx}(t,\Phi(\omega,t,x) )u(\omega,t,\Phi(\omega,t,x)) \right|\\
 \leq& \left|[ \mu ^{-1}\tilde{u}_x\tilde{\gamma}_x +\mu ^{-1}\partial_xG\star ((\tilde{u}_x\tilde{\gamma}_x)_x+\tilde{u}_x\tilde{\gamma})](t,\Phi(\omega,t,x) ) \right|\\
  \leq& \mu ^{-1}(\|G\|_{L^\infty}+\|G_x\|_{L^\infty})(\|\tilde{u}_x\tilde{\gamma}_x\|_{L^1} +\|\tilde{u}_x\tilde{\gamma}\|_{L^1})
  \leq  CE^2(0), \quad \textrm{on}~~ B,
\end{split}
\end{equation*}
which implies that
 \begin{equation}\label{4.35}
\begin{split}
\|\tilde{\gamma}_x(t,\cdot)\|_{L^\infty}=\|\tilde{\gamma}_x(t,\Phi(\omega,t,\cdot))\|_{L^\infty}\leq CE^2(0)t+\| \gamma _{0,x}\|_{L^\infty}, \quad \textrm{on}~~ B.
\end{split}
\end{equation}
Thereby, we get from \eqref{4.34}-\eqref{4.35} that
\begin{equation*}
\begin{split}
\frac{1}{2}\frac{\mathrm{d}}{\mathrm{d} t} \int_\mathbb{T}(\tilde{m}^2+\tilde{\rho}^2)\mathrm{d} x  \leq  C( E^2(0)t+\| \gamma _{0,x}\|_{L^\infty}+M) \int_\mathbb{T} (\tilde{m}^2+\tilde{\rho}^2) \mathrm{d} x,\quad \textrm{on}~~ B.
\end{split}
\end{equation*}
An application of the Gronwall inequality yields that the solution $\|(\tilde{u},\tilde{\gamma})(t)\|_{\mathbb{H}^s}< \infty$  on $B$, for any $t>0$, which shows that $B\subseteq A$. Hence, $A=B$ $\mathbb{P}$-almost surely. This proves \eqref{4.32}.

{\textsf{Step 2:}} Define the quantity
$$
H(\omega,t)\triangleq \inf_{x\in \mathbb{T}}\tilde{u}_x(\omega,t,x).
$$
Theorem \ref{th1} implies that the first component of the solution $\tilde{u} \in \mathcal {C} ^1([0,\mathbbm{t});H^{2}(\mathbb{T}))$ $\mathbb{P}$-almost surely for $s>3$, it follows from Constantin's theorem  (cf. Theorem 2.1 in \cite{constantin1998}) that, for almost $\omega\in \Omega$, there exists a point $\xi(\omega,t)\in \mathbb{R}$ such that
$$
H(\omega,t)=\inf_{x\in \mathbb{T}}\tilde{u}_x(\omega,t,x)=\tilde{u}_x(\omega,t,\xi(\omega,t)),\quad \tilde{u}_{xx}(\omega,t,\xi(\omega,t))=0.
$$
Moreover, the function $H(\omega,t) $ is absolutely continuous in $t$ $\mathbb{P}$-almost surely, and
\begin{equation}\label{4.36}
\begin{split}
\frac{\mathrm{d}}{\mathrm{d} t}H(\omega,t)= \tilde{u}_{xt}(\omega,t,\xi(\omega,t)),\quad \mathbb{P}\textrm{-a.s.}
\end{split}
\end{equation}
Thanks to the facts of $
\mu ^{-1}(t)> 0$, $G\star f(x)\geq 0$, if $f\geq 0$, we get from \eqref{4.36} and the estimate $\|f\|_{L^\infty}\leq \frac{1}{\sqrt{2}}\|f\|_{H^1}$ that
\begin{equation}\label{4.37}
\begin{split}
 \frac{\mathrm{d}}{\mathrm{d} t}H( t)&= \left(\mu ^{-1}  f-\mu ^{-1}\tilde{u} \tilde{u}_{xx} - \mu ^{-1} \tilde{u}_x^2 -\mu ^{-1} G\star f\right)(t,\xi(t))\\
 &\leq  -\frac{1}{2}\mu ^{-1} H^2( t)+\mu ^{-1} \tilde{u}^2(t,\xi(t))+\frac{1}{2}\mu ^{-1} \tilde{\gamma}^2(t,\xi(t)) +\mu ^{-1} G\star ( \tilde{\gamma}_x^2)(t,\xi(t))\\
 &\leq  -\frac{1}{2}\mu ^{-1} H^2( t)+\frac{1}{2} \mu ^{-1} \|\tilde{u}(t,\cdot)\|_{H^1}^2 +\frac{1}{4} \mu ^{-1} \|\tilde{\gamma}(t,\cdot)\|_{H^1}^2+ \frac{1}{2}\mu ^{-1}  \|  \tilde{\gamma}_x(t,\cdot) \|_{L^2}^2\\
 &\leq  -\frac{1}{2}\mu ^{-1} H^2( t)+\frac{1}{2} \mu ^{-1} E (0),
\end{split}
\end{equation}
for all $t\in [0,\mathbbm{t})$ $\mathbb{P}$-almost surely. From the assumption $H(0)= \inf_{x\in \mathbb{T}} (\partial_xu)(x)\leq (\partial_xu)(x_0)\leq - \sqrt{E (0)}$ $\mathbb{P}$-almost surely, one can deduce that
\begin{equation}\label{4.38}
\begin{split}
 H(t)\leq - \sqrt{E (0)}, \quad \textrm{for any}~~ t\in [0,\mathbbm{t}),~~ \mathbb{P}\textrm{-a.s.}
\end{split}
\end{equation}
Otherwise, we define
$$
\mathbbm{t}'\triangleq \inf\left\{t\geq0;~~H(\omega,t)>- \frac{1}{2}\sqrt{E (0)}\right\} \wedge \mathbbm{t}.
$$
Clearly, $\mathbb{P}\{\mathbbm{t}'>0\}=1$. By \eqref{4.36}, there is an event $\Omega'\in \Omega $ with $\mathbb{P}(\Omega')=1$ such that $H(\omega',t)$ is absolutely continuous for all $\omega'\in\Omega'$. If $\mathbbm{t}' = \mathbbm{t}$ $\mathbb{P}$-almost surely does not hold, then there must be a subset $\Omega''\subset \Omega$ with $\mathbb{P}(\Omega'')>0$ such that $0<\mathbbm{t}' (\omega'')< \mathbbm{t}(\omega'')$ for all $\omega''\in \Omega '$. For any $\omega'''\in \Omega'\cap \Omega ''$, we get from the continuity of $H(t)$ that $H(\omega''',\mathbbm{t}' (\omega'''))=- \sqrt{E (0)}$. However, in view of \eqref{4.37}, we get by using again the continuity of $H(t)$ that
$$
\frac{\mathrm{d}}{\mathrm{d} t}H(\omega''',t) <0,\quad \textrm{ for} ~~ t\in [0,\mathbbm{t}' (\omega''')),
$$
which implies $
H(\omega''',\mathbbm{t}' (\omega'''))<- \sqrt{E (0)}$.
This is a contradiction, so $\mathbbm{t}'=\mathbbm{t}$ $\mathbb{P}$-almost surely, and \eqref{4.38} holds.

Since $H(t)< H(0)<- \sqrt{E (0)}< 0$ for all $t\in [0,\mathbbm{t})$, we have
$$
0<\frac{H(0)}{H(t)}<1,\quad 0<\frac{E(0)}{H^2(0)}<1,\quad \forall t\in [0,\mathbbm{t}),~~~\mathbb{P}\textrm{-a.s.}
$$
It then follows from \eqref{4.37} that
\begin{equation*}
\begin{split}
-\frac{\mathrm{d}}{\mathrm{d} t}\frac{1}{H( t)}= \frac{1}{H^2( t)}\frac{\mathrm{d}}{\mathrm{d} t}H( t) &\leq  -\frac{1}{2}\mu ^{-1}  +\frac{1}{2} \mu ^{-1}\frac{E (0)}{H^2( t) }\\
&= -\frac{1}{2}\mu ^{-1}  +\frac{1}{2} \mu ^{-1}\frac{E(0)}{H^2( 0)}\frac{H^2( 0) }{H^2( t) }\\
&\leq -\frac{1}{2}\mu ^{-1} \left(1-  \frac{E(0)}{H^2( 0)} \right),\quad \forall t\in [0,\mathbbm{t}),~~~\mathbb{P}\textrm{-a.s.}
\end{split}
\end{equation*}
Integrating the last inequality leads  to
\begin{equation}\label{4.39}
\begin{split}
-\frac{1}{H( 0)}  \geq \frac{1}{H( t)}-\frac{1}{H( 0)}  \geq  \frac{1}{2}\left(1-  \frac{E(0)}{H^2( 0)} \right)\int_0^te^{cW(r)-\frac{1}{2}c^2r} \mathrm{d} r,\quad \forall t\in [0,\mathbbm{t}),~~~\mathbb{P}\textrm{-a.s.}
\end{split}
\end{equation}
For any given $\lambda\in (0,1)$ and $c\neq 0$, define
$$
\Omega _\lambda\triangleq \left\{\omega;~~e^{cW(t)-\frac{3}{8}c^2t} \geq \lambda e^{ -\frac{3}{8}c^2t}~~\textrm{for all} ~~t\right\}.
$$
It follows from \eqref{4.39} that
\begin{equation*}
\begin{split}
-\frac{1}{H( 0)} &\geq  \frac{1}{2}\left(1-  \frac{E(0)}{H^2( 0)} \right)\int_0^\mathbbm{t}e^{cW(r)-\frac{3}{8}c^2r} e^{-\frac{1}{8}c^2r}\mathrm{d} r \\
&\geq  \frac{\lambda}{2}\left(1-  \frac{E(0)}{H^2( 0)} \right)\int_0^\mathbbm{t} e^{-\frac{1}{2}c^2r}\mathrm{d} r\\
&=\frac{ \lambda}{c^2}(1-e^{-\frac{1}{2}c^2\mathbbm{t}} )\left(1-  \frac{E(0)}{H^2( 0)} \right),\quad \textrm{on} ~~\Omega _\lambda.
\end{split}
\end{equation*}
Assume that $\mathbbm{t}(\omega)=\infty$ on some subset $ \Omega_\lambda '\subset \Omega_\lambda$ with positive probability, we deduce from the last inequality that
\begin{equation}\label{4.40}
\begin{split}
 \frac{1}{H( 0)} +  \frac{\lambda}{c^2} \left(1-  \frac{E(0)}{H^2( 0)} \right)\leq 0,\quad \textrm{on} ~~\Omega _\lambda'.
\end{split}
\end{equation}
Notice that
 $$
   \varsigma_1=\frac{-c^2 + \sqrt{c^4+ 4\lambda^2E(0)}}{2\lambda} > 0, \quad\varsigma_2= \frac{-c^2 - \sqrt{c^4+ 4\lambda^2E(0)}}{2\lambda} <0,
$$ are two real roots to the quadratic equation $\lambda x^2+c^2x-\lambda E(0)=0$. Hence if $H(0)\leq (\partial_xu_0)(x_0)< \varsigma_2 $ $\mathbb{P}$-almost surely, then  $\lambda H^2( 0)+c^2H( 0)-\lambda E(0)>0$ on $\Omega _\lambda'$, which contradicts to \eqref{4.40}, and this shows that $\mathbbm{t}(\omega)<\infty$ for almost every $\omega\in\Omega_\lambda$, i.e.,  $\Omega _\lambda \subseteq \{\mathbbm{t}<\infty\}$. Thereby, we get from $\Omega _\lambda \supset   \{ e^{cW(t)} \geq \lambda  ~~\textrm{for all} ~~t\}$ that
\begin{equation*}
\begin{split}
 \mathbb{P}\{\mathbbm{t}<\infty\} \geq \mathbb{P}\left\{ e^{cW(t) } \geq \lambda ~~ \textrm{for all} ~~t\right\}> 0,
\end{split}
\end{equation*}
which combined with \eqref{4.32} yield that the solution $(\tilde{u},\tilde{\gamma})$ breaks in finite time with positive probability.
 This completes the proof of Theorem \ref{th5}.
\end{proof}

\section{Acknowledgement}
The author wishes to thank the anonymous reviewers for their valuable comments and suggestions which resulted in a significant enhancement to the article. The author also would like to express his sincere gratitude to Professor Darryl D. Holm for sharing insights on the background of MEP2 system, and to Postdoctor Hao Tang for many fruitful discussions on the first version of Theorem \ref{th3}. Last but not least, the author warmly thanks Professor Bin Liu for beneficial discussions and supporting during the preparation of this paper.

This work was partially supported by the National Natural Science Foundation of China (Grant No. 12231008), and the National Key Research and Development Program of China (Grant No.  2023YFC2206100).

\bibliographystyle{apa}
\bibliography{zhangl}

\end{document}